\documentclass[12pt]{iopart}

\usepackage{dsfont}
\usepackage{color,xcolor,multirow}

\newcommand{\tp}[1]{\textcolor{black}{#1}}

 \newcommand{\h}[1]{\mathbf{#1}}
 \usepackage{graphicx}
 \expandafter\let\csname equation*\endcsname\relax 
\expandafter\let\csname endequation*\endcsname\relax 
\usepackage{amsmath,eqnarray,enumerate}
 \usepackage{bm}
 \usepackage{algorithmic,algorithm}
 \usepackage{amsthm,amssymb,cleveref}
\newtheorem{remark}{Remark}
\newtheorem{assumption}{Assumption}
\newtheorem{lemma}{Lemma}
\newtheorem{definition}{Definition}
\newtheorem{theorem}{Theorem}

\usepackage{url}

\usepackage[pagewise]{lineno}
\begin{document}
	
	\title[$L_1/L_2$ on the gradient] {Minimizing $L_1$ over $L_2$ norms on the gradient}
	
	\author{Chao Wang$^1$, Min Tao$^2$, Chen-Nee Chuah$^1$, James Nagy$^3$, and Yifei Lou$^4$}
	\address{$^1$Department of Electrical and Computer Engineering, University of California Davis, Davis, CA 95616, USA}
\address{$^2$Department of Mathematics, National Key Laboratory for Novel Software Technology, Nanjing University, Nanjing 210093, China}
\address{$^3$Department of Mathematics, Emory University, Atlanta, GA 30322, USA}
	\address{$^4$Department of Mathematical Sciences, University of Texas at Dallas, Richardson, TX 75080, USA}

	\ead{
		\mailto{chaowang.hk@gmail.com},
		\mailto{taom@nju.edu.cn},
		\mailto{chuah@ucdavis.edu},
		\mailto{jnagy@emory.edu},
		\mailto{yifei.lou@utdallas.edu}}
	
	\begin{abstract}
In this paper, we study the $L_1/L_2$ minimization on the gradient for imaging applications.  Several recent works have demonstrated that  $L_1/L_2$  is better than the $L_1$ norm when approximating the $L_0$ norm to promote sparsity. Consequently, we postulate that applying $L_1/L_2$ on the gradient is better than the classic total variation (the $L_1$ norm on the gradient) to enforce the sparsity of the image gradient. To verify our hypothesis, we consider a constrained formulation to reveal empirical evidence on the superiority of $L_1/L_2$  over $L_1$ when recovering piecewise constant signals from low-frequency measurements. Numerically, we design a specific splitting scheme, under which we can prove   subsequential \tp{and global} convergence for the alternating direction method of multipliers (ADMM) \tp{under certain conditions}. Experimentally, we demonstrate visible improvements of $L_1/L_2$  over $L_1$ and other nonconvex regularizations for image recovery from low-frequency measurements and two medical applications of MRI and CT reconstruction. All the numerical results show the efficiency of our proposed approach.
	\end{abstract}
	
	\vspace{2pc}
	
	\submitto{Inverse Problems}

	\section{Introduction}
Regularization methods play an important role in inverse problems to refine the solution space by prior knowledge and/or special structures.
For example, the celebrated total variation (TV) \cite{rudinOF92} prefers piecewise constant images, while total generalized variation (TGV)  \cite{Bredies2010TGV} and fractional-order TV \cite{chen2015Fractional,Zhang2015Atotal} tend to preserve piecewise smoothness of an image. TV can be defined either isotropically or anisotropically. The anisotropic TV	 \cite{osherE03} \tp{in the discrete setting} is equivalent to applying the $L_1$ norm on the image gradient. As the $L_1$ norm is often used to enforce a signal being \textit{sparse,} one can interpret the TV regularization as to promote the sparsity of gradient vectors. 


To find the sparsest signal, it is straightforward  to minimize the $L_0$ norm (counting the number of nonzero elements), which is unfortunately NP-hard \cite{natarajan95}. A popular approach involves the convex relaxation of using the $L_1$ norm to replace the ill-posed $L_0$ norm, with the equivalence between $L_1$ and $L_0$ for sparse signal recovery given in terms of restricted isometry property (RIP) \cite{CRT}. 
However, Fan and Li  \cite{fan2001variable} pointed out that   the $L_1$ approach is biased towards  large coefficients, and proposed to minimize a nonconvex regularization, called smoothly clipped absolute deviation (SCAD). Subsequently, various nonconvex  functionals emerged  such as 
minimax concave penalty (MCP) \cite{zhang2010nearly}, capped $L_1$  \cite{zhang2009multi,shen2012likelihood,louYX16}, and transformed $L_1$ \cite{lv2009unified,zhangX17,zhangX18}. 
Following the literature on sparse signal recovery, there is a trend to apply a nonconvex regularization on the gradient to deal with images. For instance, Chartrand \cite{chartrand07} discussed both the $L_p$    norm with $0<p<1$ for sparse signals and $L_p$ on the gradient for magnetic resonance imaging (MRI), while   MCP   on the gradient was proposed in  \cite{you2019nonconvex}.

\tp{Recently, a scale-invariant functional 
$L_1/L_2$  was examined, which gives promising results in recovering sparse signals \cite{l1dl2,L1dL2_accelerated,TaoLou2020} and sparse gradients \cite{wang2020limitedct}. In this paper,  we rely on a constrained formulation to characterize some scenarios, under which the quotient of the $L_1$   and $L_2$ norms on the gradient performs well. In particular, we borrow the analysis of a \textit{super-resolution} problem, which refers to recovering a sparse signal from its low-frequency measurements.} Cand\'es 
 and Fernandez-Granda \cite{candesF13} proved that if a signal has  spikes (locations of nonzero elements) that are sufficiently separated, then the $L_1$ minimization yields an exact recovery for super-resolution. We innovatively design a certain type of piecewise constant signals that lead to well-separated spikes after taking the gradient. Using such signals, we  empirically demonstrate that the TV minimization can find the desired solution under a similar separation condition as in \cite{candesF13}. We also illustrate that $L_1/L_2$ can deal with less separated spikes in gradient,  and is better at preserving image contrast than $L_1$. These empirical evidences show \tp{
 $L_1/L_2$ holds great potentials in promoting sparse gradients and preserving image contrasts.}
 To the best of our knowledge, it is the first time to relate the exact recovery of gradient-based methods to minimum separation and image contrast in a super-resolution problem.






Numerically, we consider the same splitting scheme used in an unconstrained formulation \cite{wang2020limitedct} to minimize the $L_1/L_2$ on the gradient, followed by the alternating direction method of multipliers (ADMM) \cite{boydPCPE11admm}. We formulate the linear constraint using an indicator function, which is not strongly convex. As a result, the convergence analysis in the unconstrained model \cite{wang2020limitedct} is not directly applicable to this problem. We   utilize the property of indicator function as well as the optimality conditions for constrained optimization problems to prove that the sequence generated by the proposed algorithm has a subsequence converging to a stationary point.
\tp{Under a stronger assumption, we can establish the convergence of the entire sequence, referred to as \textit{global convergence}. }


We present some algorithmic insights on  computational efficiency of our proposed algorithm for nonconvex optimization. In particular, we show an additional box constraint   not only prevents the solution from being stuck at local minima, but also stabilizes the algorithm. Furthermore, we discuss algorithmic behaviors on two types of applications: MRI and computed tomography (CT). For the MRI reconstruction, a subproblem in ADMM has a closed-form solution, while an iterative solver is required  for CT. As the accuracy of the subproblem varies between MRI and CT, we shall alter  internal settings of our algorithm accordingly.
In summary, this paper relies on a constrained formulation to discuss theoretical and computational aspects of a nonconvex regularization for imaging problems. The major contributions are three-fold:
\begin{enumerate}[(i)]
    \item We reveal  empirical evidences towards exact recovery of piecewise constant signals and demonstrate the superiority of $L_1/L_2$ on the gradient over TV.
    \item \tp{We establish the subsequential  convergence of the proposed algorithm and explore the global convergence under the certain assumptions. }
    \item We conduct extensive experiments to characterize computational efficiency of  our algorithm and discuss how internal settings can be customized to cater to specific imaging applications, such as MRI and limited-angle CT reconstruction. Numerical results highlight the superior performance of our approach over other gradient-based regularizations.
\end{enumerate}

The rest of the paper is organized as follows. \Cref{sect:prelim} defines the notations that will be used through the paper, and gives a brief review on the related works. The empirical evidences for TV's exact recovery and advantages of the proposed model  are given in \Cref{sec:toy_example}. The numerical scheme is detailed in \Cref{sect:model}, followed by convergence analysis in \Cref{sect:conv}. \Cref{sec:experiments} presents three types of imaging applications: super-resolution, MRI  and  CT reconstruction problems. Finally, conclusions and future works are given in \Cref{sec:conclusions}.

\section{Preliminaries}
	\label{sect:prelim}
	
We use a bold letter to denote a vector, a capital letter to denote a matrix or linear operator, and a calligraphic letter for a functional space.  We use $\odot$ to denote the component-wise multiplication of two vectors. When a function (e.g., sign, max, min) applies to a vector, it returns a vector with corresponding component-wise operation.

We adopt a discrete setting to describe the related models.  
	Suppose a two-dimensional (2D) image is defined on an $m\times n$ Cartesian grid. By using a standard linear index, we can represent a 2D image as a vector, i.e., the $((i-1)m+j)$-th component  denotes the intensity value at pixel $(i,j).$  	We define a discrete gradient operator,
	\begin{equation}\label{eq:gradient}
		D \h u:= \left[\begin{array}{l}
		D_x\\
		D_y
		\end{array}
		\right]\h u,
	\end{equation}
where $D_x,D_y$ are the finite forward difference operator \tp{with periodic boundary condition} in the  horizontal and vertical directions, respectively. We denote  $N := mn$ and the Euclidean spaces by  $\mathcal X:=\mathds{R}^{N}, \mathcal Y:=\mathds{R}^{2N}$, then $\h u\in \mathcal X$ and $D\h u\in \mathcal Y$. We can apply the standard norms, e.g., $L_1,L_2,$  on vectors $\h u$ and $D\h u.$ For example, the $L_1$ norm on the gradient, i.e., $\|D\h u\|_1$, is the anisotropic TV regularization \cite{osherE03}. Throughout the paper, we use TV and ``$L_1$ on the gradient'' interchangeably.
Note that the isotropic TV \cite{rudinOF92} is the $L_{2,1}$ norm, i.e.,
$\|(D_x\h u,D_y\h u)^T\|_{2,1},$
although Lou et al.~\cite{louZOX15} claimed to consider  a weighted difference of anisotropic and isotropic TV based on  the $L_1$-$L_2$ functional \cite{yinEX14,louYHX14,maLH17,louY18} (isotropic TV is not the $L_2$ norm on the gradient.)

We examine the $L_1/L_2$ penalty on the gradient in a constrained formulation,
	\begin{equation}\label{eq:grad_con}
	\min_{\h u}  \frac{\| D\h u \|_1}{\|  D \h u \|_2} \quad \mathrm{s.t.} \ \  A \h u =\h b.
	\end{equation}
One way  to solve for \eqref{eq:grad_con} involves the following equivalent form
	\begin{equation}\label{eq:grad_con2}
	\min_{\h u, \h d, \h h }  \frac{\| \h d \|_1}{\| \h h \|_2} \quad \mathrm{s.t.} \ \  A \h u =\h b, \  \h d = D\h u,\  \h h = D\h u,
	\end{equation}
	with  two auxiliary variables $\h d$ and  $\h h$. For more details, please refer to \cite{l1dl2} that presented a proof-of-concept example  for  MRI reconstruction.
Since the splitting scheme \eqref{eq:grad_con2} involves two block variables of $\h u$ and $(\h d, \h h)$,
	 the existing ADMM convergence results  \cite{guo2017convergence,pang2018decomposition,wang2019global} are not applicable. 
An alternative approach was discussed in our preliminary work \cite{wang2020limitedct} for an unconstrained minimization problem,
	\begin{equation}\label{equ:split_model_uncon}
	\min_{\h u, \h h}  \frac{\| D\h u \|_1}{\| \h h \|_2}+\frac{\lambda}{2} \|A \h u -\h b\|_2^2 \quad \mathrm{s.t.} \quad \h h = D\h u,
	\end{equation}
	where $\lambda>0$ is a weighting parameter.
By only introducing one variable $\h h,$  the new splitting scheme \eqref{equ:split_model_uncon} can guarantee the ADMM framework with subsequential convergence.
	 
In this paper, we  incorporate the  splitting scheme \eqref{equ:split_model_uncon} to solve the constrained problem  \eqref{eq:grad_con}, which is crucial to reveal theoretical properties of the gradient-based regularizations for image reconstruction, as elaborated in \Cref{sec:toy_example}. 
\tp{Another contribution of this work lies in the convergence analysis, specifically for different optimality conditions  of the constrained problem, as opposed to unconstrained formulation presented in \cite{wang2020limitedct}.}
It is true that the constrained formulation  limits our experimental design in a noise-free fashion, but it helps us to draw conclusions solely on the model, ruling out the influence from other nuisances such as noises and tuning parameters. Our model \eqref{eq:grad_con} is parameter-free, while there is a parameter $\lambda$ in the unconstrained problem \eqref{equ:split_model_uncon}.	 


\section{Empirical studies}\label{sec:toy_example}

	We aim to demonstrate the superiority of $L_1/L_2$ on the gradient over TV  for a super-resolution problem \cite{candes2014towards}, in which a sparse vector can be exactly recovered via the $L_1$ minimization. 
	A mathematical model for \textit{super-resolution} is expressed as
\begin{equation}\label{eq:SRdiscrete}
b_k=\frac 1 {\sqrt{N}}\sum_{j=0}^{N-1} u_j e^{-i2\pi kj/N}, \qquad |k|\leq f_c,
\end{equation}
where $i$ is the imaginary unit, $\h u\in\mathbb R^N$ is a vector to be recovered, and $\h b\in\mathbb C^n$ consists of the given low frequency measurements with $n=2f_c+1<N$.
Recovering $\h u$ from $\h b$ is referred  to as super-resolution  in the sense that the underlying signal $\h u$ is defined on a fine grid with spacing $1/N$, while a direct inversion of $n$ frequency data yields a signal defined on a coarser grid with spacing $1/n$. For simplicity, we use matrix notation to rewrite \eqref{eq:SRdiscrete} as $\h b=S_nF \h u$, where $S_n$ is a sampling matrix that collects the required low frequencies and $F$ is the Fourier transform matrix. 
A sparse signal can be represented by
$
\h u=\sum_{j\in T} c_j\h e_{j},
$
where $\h e_j$ is the $j$-th canonical basis in $\mathbb R^N$, $T$ is the support set of $\h u$, and $\{c_j\}$ are coefficients. Following the work of \cite{candes2014towards}, the sparse spikes are required to be sufficiently separated to guarantee the exact recovery of the $L_1$ minimization. To make the paper self-contained, we provide the definition of minimum separation in Definition~\ref{def:MS} and an exact recovery condition in Theorem~\ref{thm:MS}.

\begin{definition}{(Minimum Separation \cite{candes2014towards})}\label{def:MS}
For an index set $T\subset \{1,\cdots,N\}$, the minimum separation (MS) of $T$ is defined as the closest wrap-around distance between any two elements from $T$,
	\begin{equation}
\triangle (T) := \min_{(t,\tau)\in T:t\neq \tau} \min\left\{|t-\tau|, \ N-|t-\tau|\right\}.
	\end{equation}
\end{definition}

\begin{theorem}{\cite[Corollary 1.4]{candes2014towards}}\label{thm:MS}
	Let $T$ be the support of $\h u$. If the minimum separation of $T$ obeys
	\begin{equation}\label{eq:MSdistance}
	\triangle(T)\geq \frac{1.87N}{f_c},
	\end{equation}
	then $\h u\in \mathbb R^N$ is the unique solution to the constrained $L_1$ minimization problem,
	\begin{equation}\label{eq:SRviaL1}
	\min_{\h u} \|\h u\|_1 \quad \mbox{\rm s.t.} \quad S_nF\h u=\h b.
	\end{equation}
\end{theorem}

We empirically extend the analysis from sparse signals to sparse gradients.  For this purpose, we construct a one-bar step function of length 100 with the first and the last $s$ elements taking value 0, and the remaining elements equal to 1, as illustrated in \Cref{fig:1bar}. The gradient of such signal is 2-sparse with MS  to be $\min(2s,100-2s)$ due to wrap-around distance. By setting 
$f_c=2,$ we only take $n=5$ low frequency measurements, and   reconstruct the signal by minimizing either $L_1$ or $L_1/L_2$ on the gradient. For simplicity, we adopt the CVX MATLAB toolbox \cite{cvx} for solving the TV model, 
\begin{equation}\label{eq:tv-con}
   \min_{\h u} \|D\h u\|_1 \quad \mathrm{s.t.} \quad A\h  u=\h  b,
\end{equation}
where we use $A=S_nF$ to be consistent with our setting \eqref{eq:grad_con}.  Note that the TV model  \eqref{eq:tv-con} is parameter free, while we need to tune an algorithmic parameter  for $L_1/L_2$. Please refer to \Cref{sect:model} for more details on the $L_1/L_2$ minimization, in which   one subproblem can be solved by CVX, and \Cref{sect:exp_alg} for \tp{sensitivity analysis} on this parameter.

	\begin{figure}[htp]
		\begin{center}
			\begin{tabular}{cc}
				\includegraphics[width=0.4\textwidth]{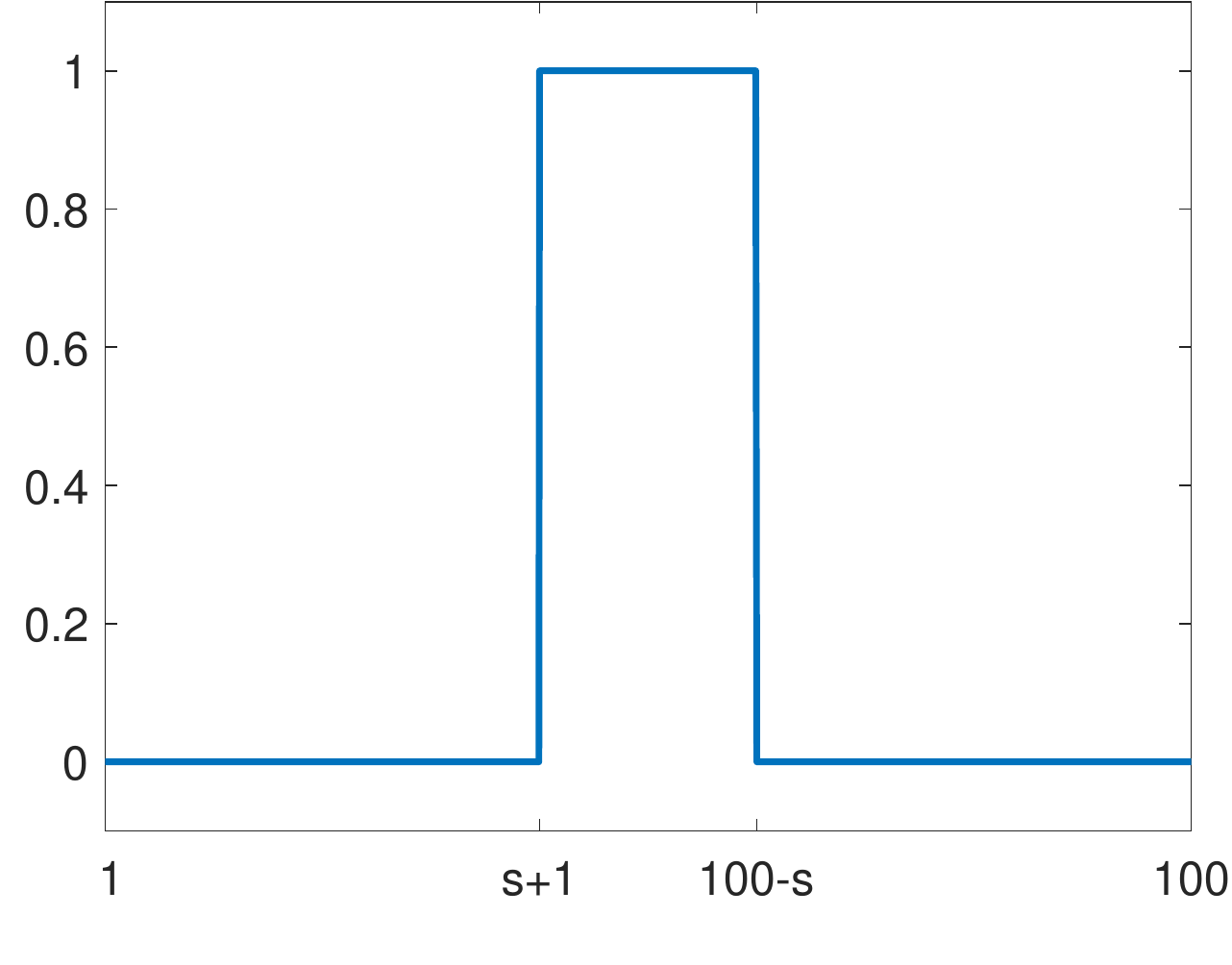}  &
				\includegraphics[width=0.4\textwidth]{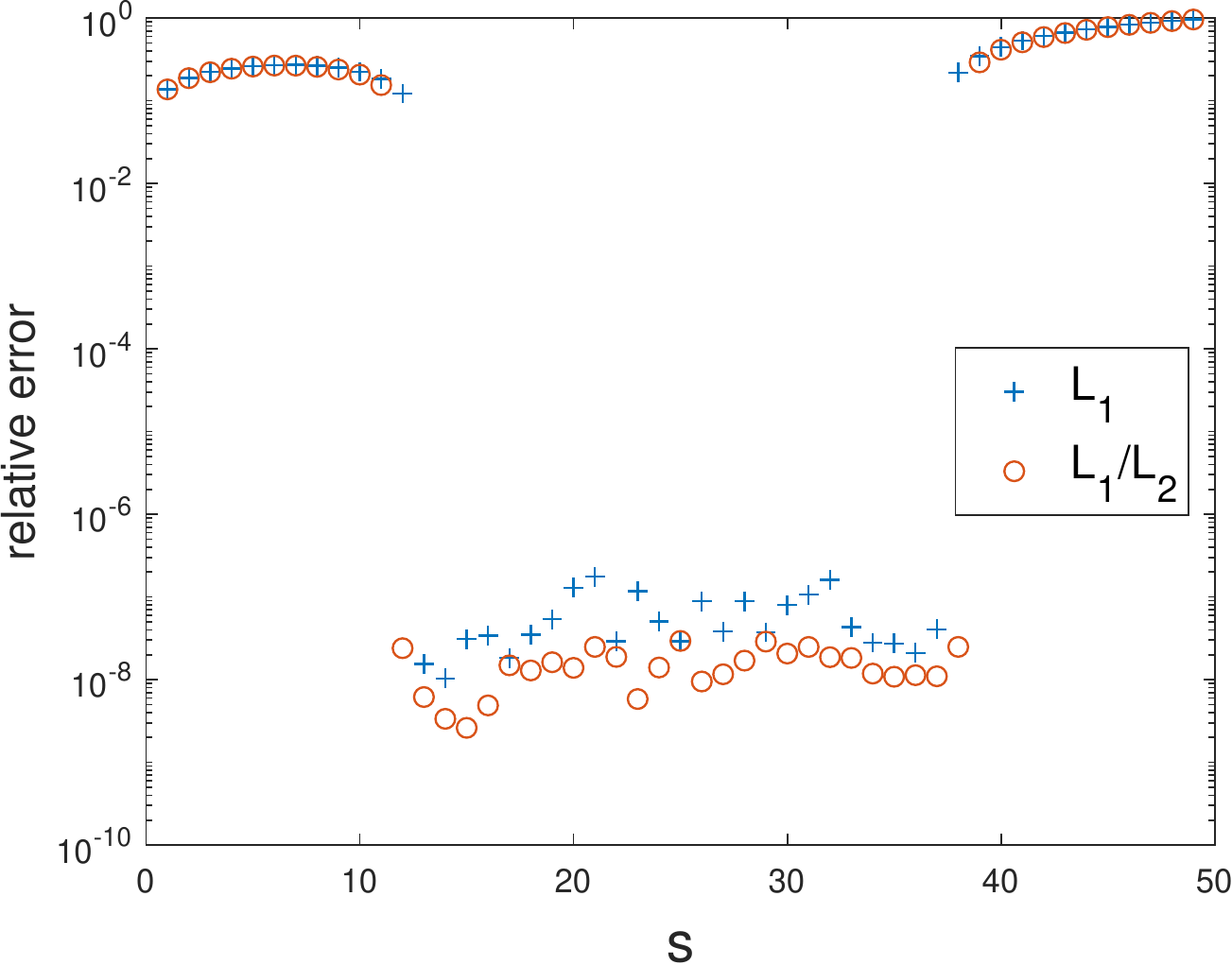}
			\end{tabular}
		\end{center}
		\caption{A general setting of a one-bar step function (left)  and reconstruction errors with respect to $s$ (right) by minimizing $L_1$ or $L_1/L_2$ on the gradient. \tp{The exact recovery interval  by $L_1$ is $s\in [13,37],$ which is smaller than $[12,38]$ by $L_1/L_2.$} }\label{fig:1bar}
	\end{figure}

\begin{figure}[htp]
		\begin{center}
			\begin{tabular}{cc}
				\includegraphics[width=0.4\textwidth]{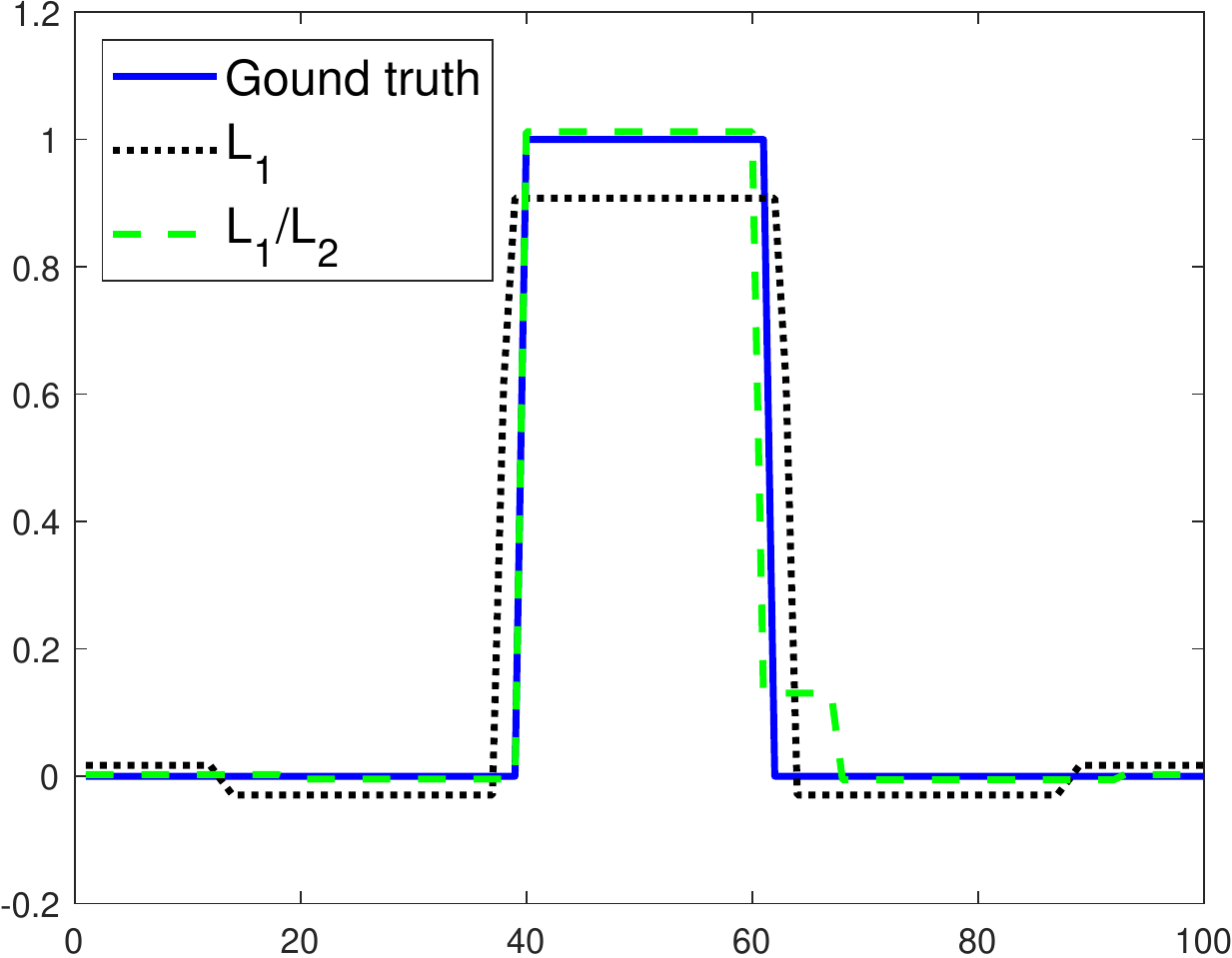} &
				\includegraphics[width=0.4\textwidth]{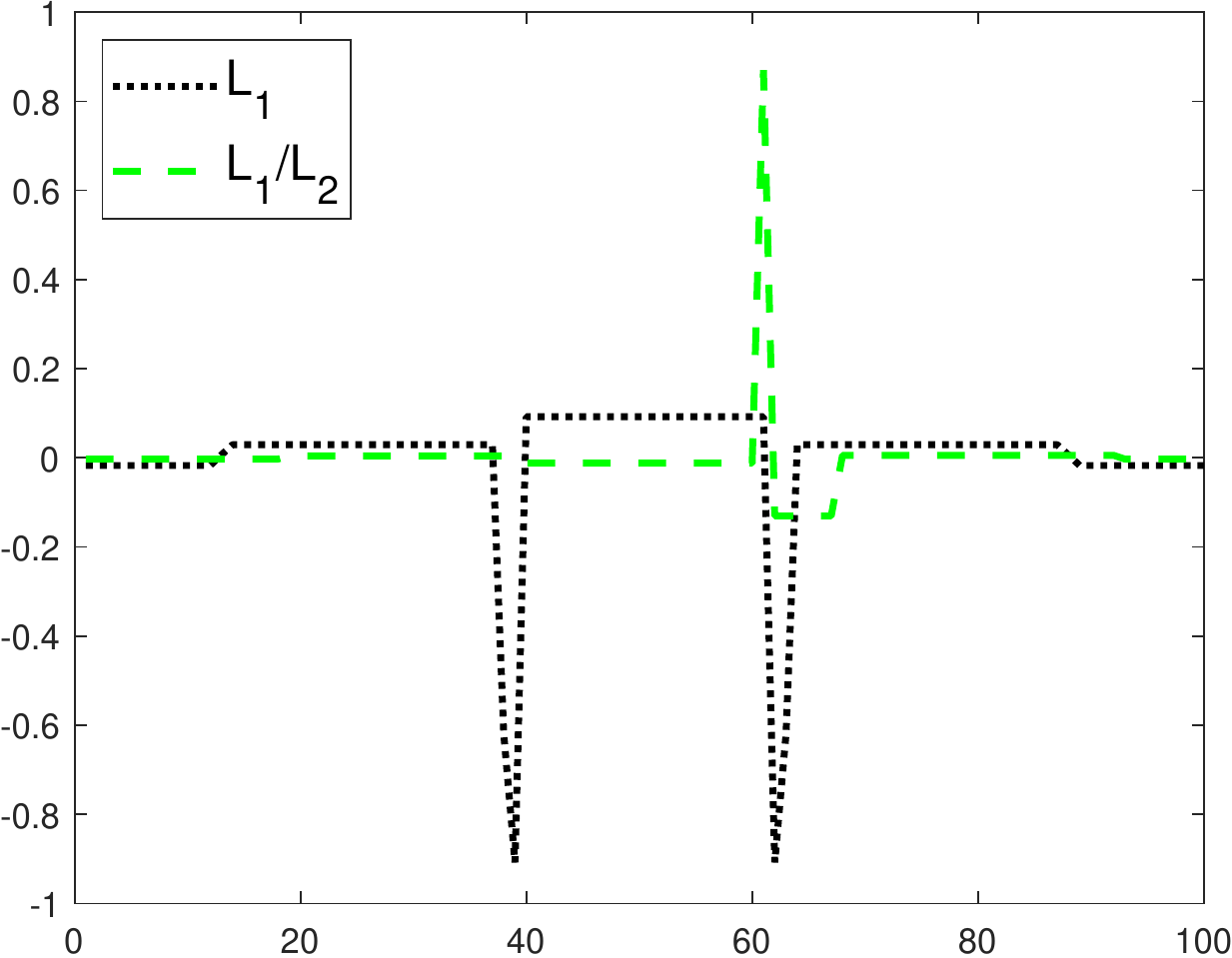} 
			\end{tabular}
		\end{center}
		\caption{A particular one-bar example (left) where
		both $L_1$ and $L_1/L_2$ models fail to find the solution. The different plot  (right) highlights that $L_1$ results in larger oscillations compared to $L_1/L_2.$}\label{fig:toy_example1bar}
	\end{figure}

By varying the value of $s$ that changes MS of the spikes in gradient, we compute the relative errors between the reconstructed solutions and the ground-truth signals. If we define an exact recovery for its relative error smaller than $10^{-6}$, we observe in \Cref{fig:1bar} that the exact recovery by $L_1$ occurs at $s\in [13,37],$ which implies that MS is larger than or equal to $26$. 
This phenomenon suggests
that  Theorem~\ref{thm:MS} might hold for sparse gradients by replacing the $L_1$ norm with the total variation.	\Cref{fig:1bar} also shows the exact recovery by $L_1/L_2$ at $s\in[12,38],$  meaning that $L_1/L_2$ can deal with less separated spikes than $L_1$.  Moreover, we further study the reconstruction results at $s=39,$ where both models fail to find the true sparse signal.  The restored solutions by these two models as well as the different plots between restored and ground truth are displayed in \Cref{fig:toy_example1bar}, showing that  our ratio model has smaller relative errors than $L_1$.

\begin{figure}[htp]
		\begin{center}
			\begin{tabular}{cc}
				\includegraphics[width=0.39\textwidth]{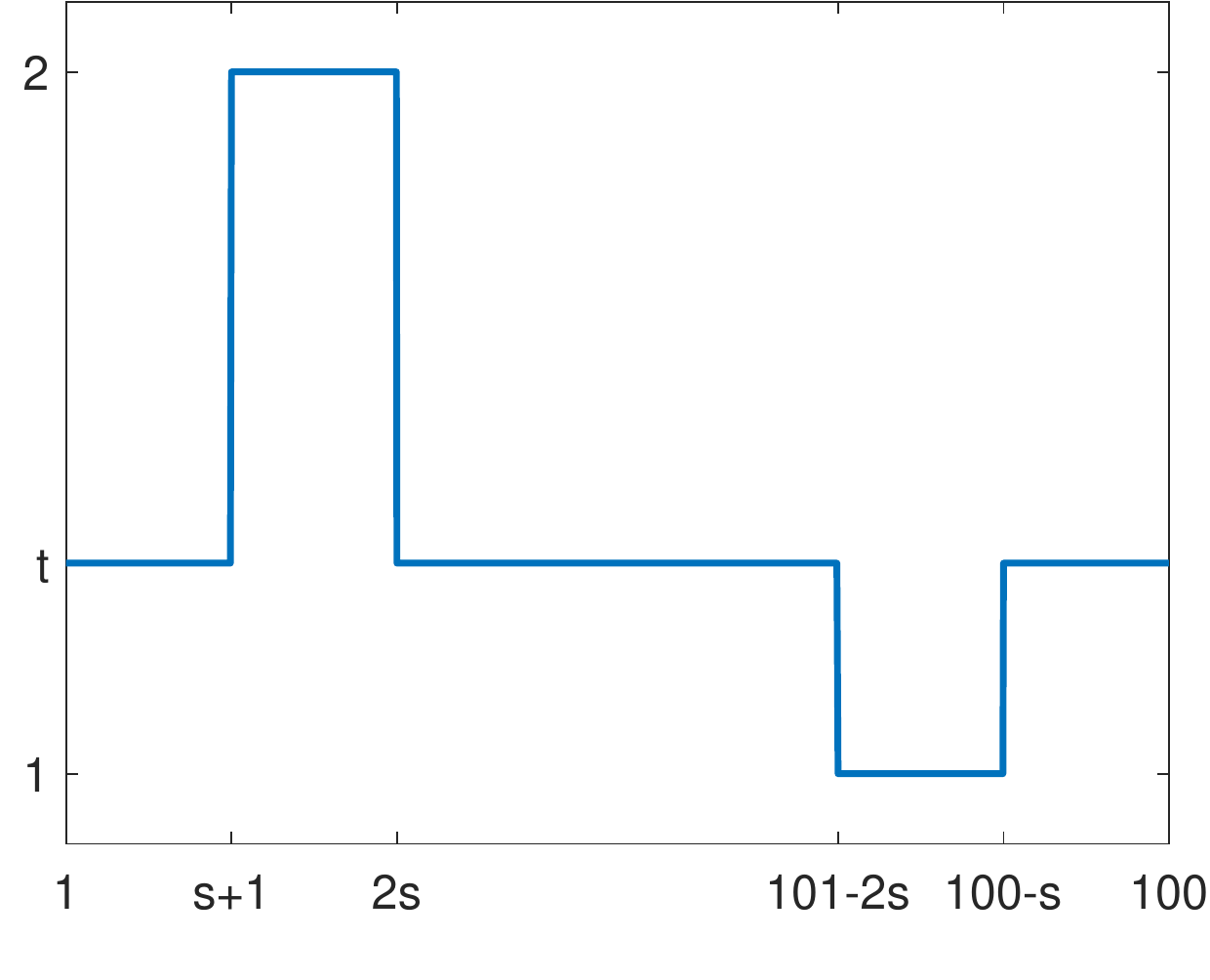} &
				\includegraphics[width=0.4\textwidth]{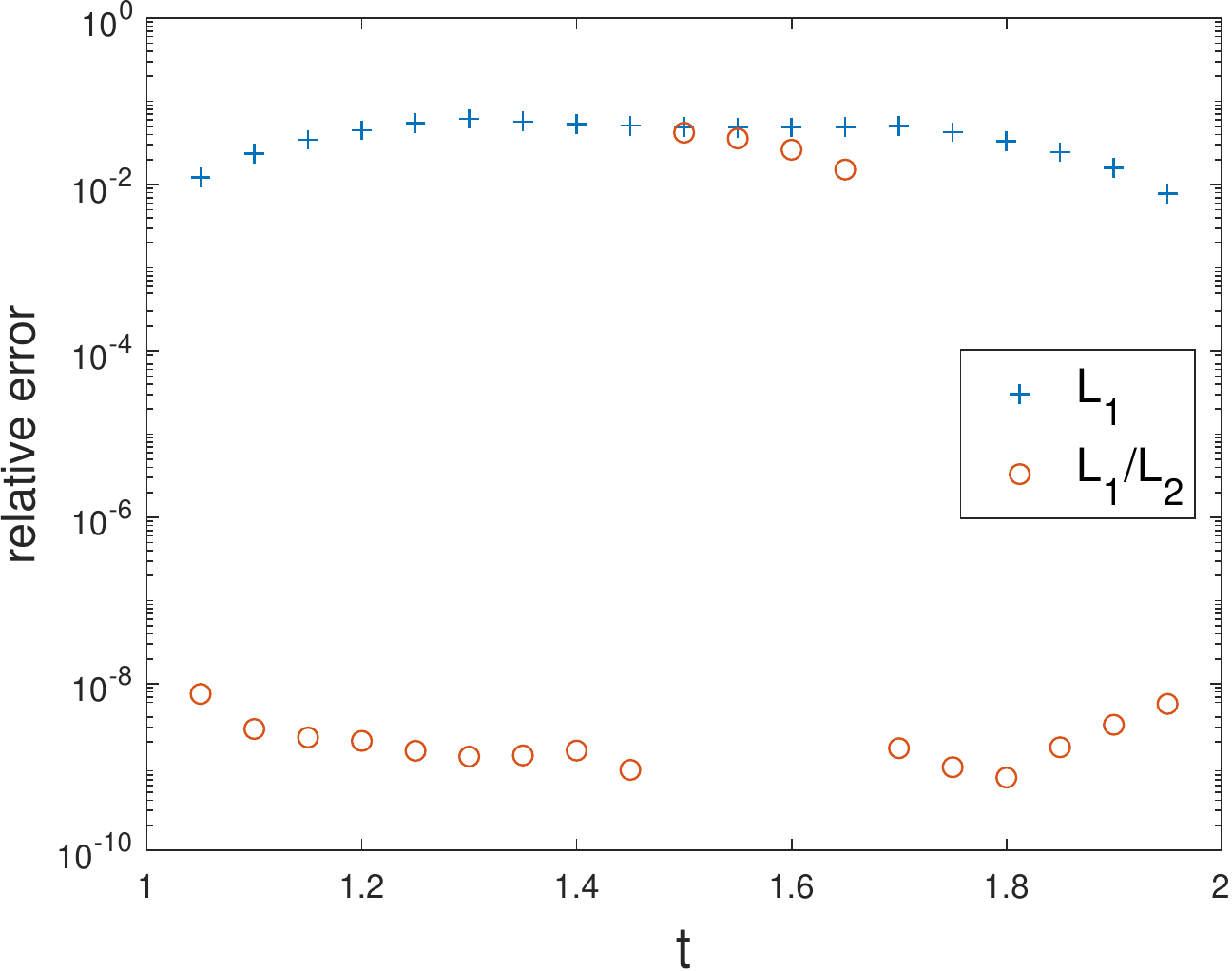} \\
			\end{tabular}
		\end{center}
		\caption{A general setting of a two-bar step function (left)  and reconstruction errors with respect to $t$ (right) by minimizing $L_1$ or $L_1/L_2$ on the gradient, \tp{showing that $L_1/L_2$ is better at preserving image contrast than $L_1$ (controlled by $t$). } }\label{fig:2bar}
	\end{figure}

	\begin{figure}[htp]
		\begin{center}
			\begin{tabular}{cc}
				\includegraphics[width=0.4\textwidth]{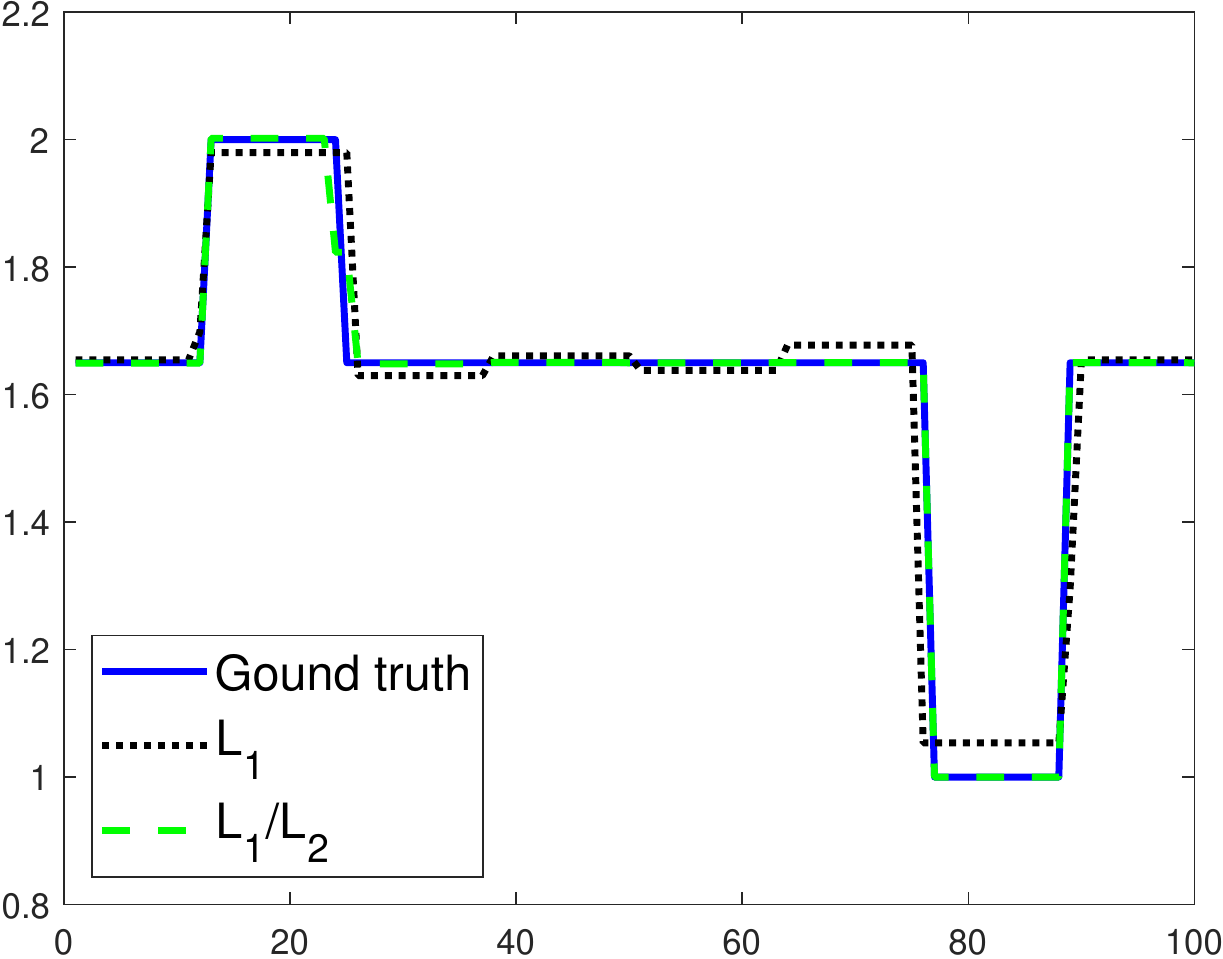}  &
				\includegraphics[width=0.4\textwidth]{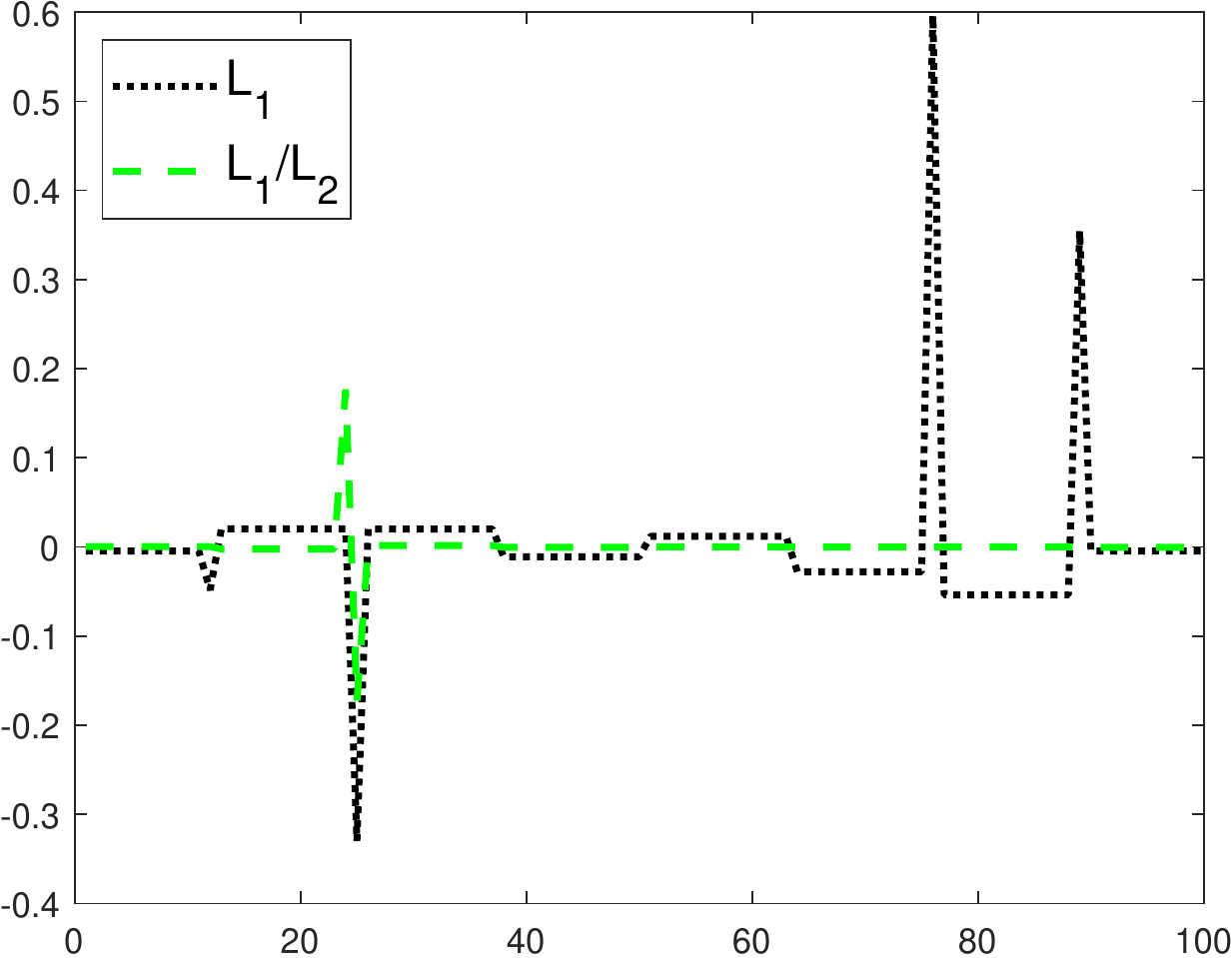} 
			\end{tabular}
		\end{center}
		\caption{
		A particular two-bar example where
		both $L_1$ and $L_1/L_2$ models fail to find the solution. The different plot is shown on the right.}\label{fig:toy_example2bar}
	\end{figure}

\Cref{fig:toy_example1bar} illustrates that the TV solution can not reach the top of the bar in the ground-truth, which is referred to as \textit{loss-of-contrast}.
Motivated by this well-known  drawback of TV, we postulate that the signal contrast may affect the performance of $L_1$ and $L_1/L_2$. To verify, we  examine a two-bar step function, in which the contrast varies by the relative heights of the two bars. Following MATLAB's notation, we set $\h u(s+1:2s)=2, \h u(\text{end}-2s+1:\text{end})=1,$ and the value of remaining elements   uniformly as $t$; see \Cref{fig:2bar} for a general setting. We fix $s=12$, and  vary the value of $t\in(1, 2)$ to generate signals with different intensity contrasts. Considering four spikes in the gradient, we set $f_c=4$ or equivalently 9 low-frequency measurements to reconstruct the signal. The reconstruction errors are plotted in \Cref{fig:2bar}, which shows that  $L_1$ fails in all the cases, and $L_1/L_2$ can find the signals
except for $t \in [1.5, 1.65]$. We further examine a particular case of  $t=1.65$ in \Cref{fig:toy_example2bar},
where both models fail to get an exact recovery, but 
$L_1/L_2$ yields smaller oscillations than $L_1$ near the edges. \Cref{fig:2bar,fig:toy_example2bar} demonstrate that $L_1/L_2$ is better at preserving image contrast than $L_1$. 


We verify that all the solutions of $L_1$ and $L_1/L_2$ satisfy the linear constraint $A\h u=\h b$ with high accuracy thanks for CVX. We further investigate when the $L_1$ approach fails, and discover that it  yields a solution that has a smaller $L_1$ norm compared to the $L_1$ norm of the ground-truth, which implies that $L_1$ is not sufficient to enforce gradient sparse. On the other hand, $L_1/L_2$ solutions often have higher objective value than the ground-truth, which calls for a better algorithm that can potentially find the ground-truth. We also want to point out that  $L_1/L_2$ solutions depend on  initial conditions. In \Cref{fig:1bar} and \Cref{fig:2bar},  we present the smallest relative errors  among 10 random vectors for initial values. 

\tp{The minimum separation distance in 1D (Definition~\ref{def:MS}) can be naturally extended to 2D. In fact, there are two types of minimum separation definitions in 2D: one  uses the $L_{\infty}$  norm to measure the distance \cite{candes2014towards}, while another definition is called \textit{Rayleigh regularity} \cite{donoho1992superresolution,morgenshtern2016super}. The
exact recovery for 2D sparse vectors was characterized in \cite{morgenshtern2016super} with additional restriction of positive signals. Both distance definitions were empirically examined in \cite{louYX16}  for point-source super-resolution. When extending to 2D sparse gradient, one can compute the gradient norm at each pixel, and separation distance can be defined as the distance between any two locations with non-zero gradient norm. To the best of our knowledge, there is no analysis on the exact recovery of sparse gradients, no matter whether it is in 1D or 2D. In \Cref{sec:toy_example},
we devote   some empirical evidences, showing  a similar relationship between sparse gradient recovery and minimum separation as Theorem~\ref{thm:MS}, which calls for a theoretical justification in the future. Once the extension from 1D sparse vectors to 1D sparse gradients is established, it is expected that the analysis can be applied to sparse gradients in 2D to facilitate theoretical analysis in imaging applications.}

	\section{The proposed approach}
	\label{sect:model}
	
Starting from \eqref{eq:grad_con}, we  incorporate an additional box constraint in the model, i.e.,
	\begin{eqnarray}\label{eq:grad_con_l1l2_box}
	&&	\min_{\h u} \frac{\| D\h u \|_1}{\|D\h u \|_2} \quad \mathrm{s.t.} \quad A\h u=\h b,\  \h u\in \tp{[p,q]^N}.
	\end{eqnarray}
		The notation \tp{$\h u\in [p,q]^N$} means that every element of $\h u$ is bounded by $[p,q].$
The box constraint  is reasonable
for image processing applications \cite{chan2012multiplicative,kk_bound_const}, since pixel values are usually bounded by $[0,1]$ or $[0, 255]$. 
On the other hand,	the box constraint is particularly helpful  for the $L_1/L_2$ model to prevent its divergence \cite{L1dL2_accelerated}.

We use the indicator function to rewrite \eqref{eq:grad_con_l1l2_box} into the following equivalent form
	\begin{equation}\label{equ:split_model_box}
	\min_{\h u,\h h}  \frac{\| D\h  u \|_1}{\| \h h \|_2} +\Pi_{A\h u=\h b}(\h u)+ \Pi_{\tp{[p,q]^N}}(\h u) \quad \mathrm{s.t.} \quad  D\h  u = \h h,
	\end{equation}
where $\Pi_{S}(\h t)$ denotes the indicator function that forces $\h t$ to belong to a feasible set $S$, i.e.,
	\begin{equation}\label{eq:indicator}
	\Pi_S(\h t) =
	\begin{cases}
	0	&	 \mbox{if } \h t \in S
	\\
	+\infty	&	\text{otherwise}.
	\end{cases}
	\end{equation}
The  augmented Lagrangian function corresponding to \eqref{equ:split_model_box} can be expressed as,
	\begin{equation}\label{eq:augL4con}
	\mathcal{L}(\h u, \h h;\h g) =\frac{\| D\h u \|_1}{\| \h h \|_2}+\Pi_{A\h u=\h b}(\h u)+\Pi_{[p,q]^N}(\h u)+ \langle \rho \h g,D\h u -\h h\rangle + \frac{\rho}{2}\|D\h u -  \h h \|_2^2,
	\end{equation}
	where $\h g$ is a dual variable and $\rho$ is a positive parameter. 
	Then  ADMM iterates as follows,
	\begin{equation} \label{ADMML1overL2}
	\left\{\begin{array}{l}
\h	u^{(k+1)}=\arg\min_{\h u} \mathcal{L} (\h u, \h h^{(k)};\h g^{(k)})\\
	\h h^{(k+1)}=\arg\min_{\h h} \mathcal{L} (\h u^{(k+1)}, \h h;\h g^{(k)})\\
	\h g^{(k+1)} = \h  g^{(k)} + D\h u^{(k+1)} - \h h^{(k+1)}.
	\end{array}\right.
	\end{equation}
	The update for $\h h$ is the same as in \cite{l1dl2}, which has a closed-form solution of
	\begin{equation}\label{eq:update-h}
	\h h^{(k+1)} =
	\begin{cases}
	\tau^{(k)}  (D\h u^{(k+1)} + \h g^{(k)}) &  \mbox{if } D\h u^{(k+1)} + \h g^{(k)} \neq \h 0\\
	\h r^{(k)}	& \text{otherwise},
	\end{cases}
	\end{equation}
	where  $\h r^{(k)}$ is a random vector with the $L_2$ norm being $\sqrt[3]{\frac{\|D\h u^{(k+1)}\|_1}{\rho}}$ and $\tau^{(k)} = \frac{1}{3} + \frac{1}{3}(\xi^{(k)} + \frac{1}{\xi^{(k)}})$ for
	\begin{equation*}\label{eq:update_root}
		\xi^{(k)} = \sqrt[3]{\frac{27\eta^{(k)} + 2 + \sqrt{(27\eta^{(k)}+2)^2 - 4}}{2} } \quad\mbox{and}\quad \eta^{(k)} = \frac{\|D\h u^{(k+1)}\|_1}{\rho \|D\h u^{(k+1)} + \h g^{(k)}\|_2^3}.
	\end{equation*}	
	
We elaborate on the $\h u$-subproblem in \eqref{ADMML1overL2}, which can be expressed by the box constraint, i.e.,
	\begin{equation}\label{eq:u-update4box}
\h 	u^{(k+1)} = \arg\min_{\h u} \frac{\|D\h u \|_1}{\| \h h^{(k)} \|_2}+   \frac{\rho}{2}\| D\h  u-\h h^{(k)} +\h g^{(k)}\|_2^2 \quad \mathrm{s.t.} \quad A\h  u=\h  b, \h u\in[p,q]^N.
	\end{equation}
To solve for \eqref{eq:u-update4box}, we introduce two variables, $\h  v$ for the box constraint and $\h d$ for the gradient, thus  getting
	\begin{equation}\label{eq:u-update4box-split}
	\min_{\h u,\h d,\h v} \frac{\|\h d \|_1}{\| \h h^{(k)} \|_2}+   \frac{\rho}{2}\|  D\h  u-\h h^{(k)} +\h g^{(k)}\|_2^2 +\Pi_{\tp{[p,q]^N}}(\h v) \quad \mathrm{s.t.} \quad  \h u =\h  v,  D\h  u= \h d, A\h u=\h b.
	\end{equation}
	The  augmented Lagrangian function corresponding to \eqref{eq:u-update4box-split} becomes 
	\begin{equation}\label{eq:grad_aug_l1_box}
	\begin{split}
	\mathcal{L}^{(k)}(\h u,\h d, \h v; \h w, \h y,\h  z)=& \frac{\|\h d\|_1}{\|\h h^{(k)}\|_2}  +\frac{\rho}{2} \|D\h  u-\h h^{(k)} +\h g^{(k)} \|_2^2  + \Pi_{\tp{[p, q]^N}}(\h v)  \\
		& + \langle \beta \h w, \h u-\h v\rangle + \frac \beta 2 \|\h u-\h v\|_2^2\\
		& + \langle \gamma \h y,D\h  u -\h d\rangle + \frac{\gamma}{2}\| D\h  u -\h d \|_2^2\\
	& +\langle \lambda \h z, A\h u-\h b\rangle +\frac{\lambda}{2}\|A\h u -\h  b\|_2^2,	
	\end{split}
	\end{equation}
	where $\h w, \h y,\h z$ are dual variables and $\beta, \gamma,  \lambda$ are positive parameters. Here we have $k$ in the superscript of $\mathcal L$ to indicate that it is the Lagrangian for the $\h u$-subproblem in \eqref{ADMML1overL2} at the $k$-th iteration.
	The ADMM framework to minimize \eqref{eq:u-update4box-split} leads to
	\begin{equation}\label{ADMM_l1_argmin_con}
	\left\{\begin{array}{l}
	\h u_{j+1} =\arg\min_{\h u} \mathcal{L} ^{(k)}(\h u, \h d_j, \h v_j;\h w_j, \h y_j, \h z_j)\\
	\h d_{j+1} = \arg\min_{\h d} \mathcal{L} ^{(k)}(\h u_{j+1}, \h d, \h v_j;\h w_j, \h y_j, \h z_j) \\
	\h v_{j+1} =\arg\min_{\h v} \mathcal{L} ^{(k)}(\h u_{j+1}, \h d_{j+1}, \h v;\h w_j, \h y_j, \h z_j)\\
\h w_{j+1} = \h w_j + \h u_{j+1} - \h v_{j+1}\\
	\h y_{j+1} = \h y_{j} + D \h u_{j+1} - \h d_{j+1}\\
	\h z_{j+1} =\h z_j +A\h u_{j+1} - \h b,
	\end{array}\right.
	\end{equation}
	where the subscript $j$ represents the inner loop index, as opposed to
	the superscript $k$ for outer iterations  in \eqref{ADMML1overL2}.  	By taking derivative of $\mathcal{L}^{(k)}$ with respect to $\h u$, we obtain a closed-form solution given by
	\begin{equation}\label{eq:u_con_sub_box}
	\begin{split}
\h 	u_{j+1} &= \Big(\lambda A^TA+ (\rho+\gamma)D^TD+\beta I\Big)^{-1}\Big(\lambda A^T (\h b-\h z_j) \\
	& \quad + \gamma D^T(\h d_j -\h y_{j})   +\rho D^T(\h h^{(k)}-\h g^{(k)})+ \beta (\h v_{j}-\h w_{j})\Big),
	\end{split}
	\end{equation}
	where $I$ stands for the identity matrix.
	When the matrix $A$ involves frequency measurements, e.g. in super-resolution and MRI reconstruction, the update in \eqref{eq:u_con_sub_box} can be implemented efficiently by the fast Fourier transform (FFT) \tp{for periodic boundary conditions when defining the derivative operator $D$ in \eqref{eq:gradient}.}	For a general system matrix $A$, we adopt the conjugate gradient descent iterations \cite{Optimization_2006book_Nocedal} to solve for \eqref{eq:u_con_sub_box}.

	The $\h d$-subproblem in \eqref{ADMM_l1_argmin_con} also has a closed-form solution, i.e.,
	\begin{equation}\label{ADMM_l1con_d}
	\h d_{j+1} = \mathbf{ shrink}\left(D\h u_{j+1} + \h y_{j}, \frac{1}{\gamma \|\h h^{(k)} \|_2}\right),
	\end{equation}
	where  $\mathbf{shrink}(\h x, \mu) = \mathrm{sign}(\h x)\odot \max\left\{|\h x|-\mu, 0\right\}$ is called \textit{soft shrinkage} and $\odot$ denotes element-wise multiplication. We update $\h v$ by a projection   onto the $[p,q]$-box constraint,  which is given by
$$
	\h 	v_{j+1}=\min\left\{\max\{\h u_{j+1}+\h w_{j}, p\}, q\right\}.
$$


In summary, we present an ADMM-based  algorithm to minimize the $L_1/L_2$ on the gradient subject to a linear system with   the box constraint
in \Cref{alg:l1dl2_box}. 
If we only run one iteration of the $\h u$-subproblem \eqref{ADMM_l1_argmin_con}, the overall ADMM iteration \eqref{ADMML1overL2}  is equivalent to  the previous approach \cite{l1dl2}. 

\begin{algorithm}[t]
		\caption{The $L_1/L_2$ minimization on the gradient}
		\label{alg:l1dl2_box}
		\begin{algorithmic}[1]
			\STATE{Input: a linear operator $A$, observed data $\h b$,  and a bound  $[p,q]$ for the original image}
			\STATE{Parameters: $\rho, \lambda,\gamma, \beta$, kMax, jMax,  and  $\epsilon \in \mathds{R}$}
			\STATE{Initialize: $ \h h, \h d, \h g, \h v, \h w, \h y, \h z=\h 0$,   and $k,j = 0$}
			
			\WHILE{$k < $ kMax or $\|\h u^{(k)}-\h u^{(k-1)}\|_2/\|\h u^{(k)}\|_2 > \epsilon$}
			
			\WHILE{$j < $ jMax or $\|\h u_{j}-\h u_{j-1}\|_2/\|\h u_{j}\|_2 > \epsilon$}
			\STATE{$\h u_{j+1} = (\lambda A^TA+(\rho+\gamma)D^TD + \beta I)^{-1}(\lambda A^T (\h b-\h z_j) + \gamma D^T(\h d_j -\h y_{j})  $  \\ $ \quad \quad \quad  +\rho D^T(\h h^{(k)}-\h g^{(k)})+\beta (\h v_j-\h w_j))$
			}
			\STATE{$\h d_{j+1} = \mathbf{ shrink}\left(D\h u_{j+1} + \h y_{j}, \frac{1}{\gamma \|\h h^{(k)} \|_2}\right)$}
				\STATE{$
			\h v_{j+1}=\min\left\{\max\{\h u_{j+1}+\h w_{j}, p\}, q\right\}$}
			\STATE{$
			\h w_{j+1} = \h w_{j} + \h u_{j+1} - \h v_{j+1}  
			$ }
			\STATE{$\h y_{j+1} = \h y_{j}+ D \h u_{j+1} - \h d_{j+1} $}
\STATE{$\h z_{j+1} = \h z_{j} +  A \h u_{j+1}- \h b$ }
			\STATE{$j = j + 1$}
			\ENDWHILE
			\RETURN $\h u^{(k+1)} = \h u_j$
\STATE{\tp{Update $\h h^{(k+1)}$ by \eqref{eq:update-h}. }}
			\STATE $\h g^{(k+1)} = \h  g^{(k)} + D\h u^{(k+1)} - \h h^{(k+1)} $
			\STATE{$k = k+1$ and $j = 0$}
			\ENDWHILE
			\RETURN $\h u^\ast = \h u^{(k)}$ \end{algorithmic}
	\end{algorithm}

	\section{Convergence analysis}\label{sect:conv}

We intend to establish the convergence of  \Cref{alg:l1dl2_box} with the box constraint, which is extensively tested in the experiments. 
Since our  ADMM framework \eqref{ADMML1overL2}  share the same structure with the unconstrained formulation,
we adapt some analysis in
\cite{wang2020limitedct} to prove the subsequential convergence for  the proposed model \eqref{eq:grad_con_l1l2_box}. 
	For example, we make the same assumptions as in \cite{wang2020limitedct},
			\begin{assumption}\label{assmp}
		 ${\cal N}(D)\bigcap{\cal N}(A)=\{\h 0\}$, 
		where $\cal N$ denotes the null space and $D$ is defined in \eqref{eq:gradient}. In addition, the norm of $\{ \h h^{(k)}\}$ generated by \eqref{ADMML1overL2} has a  lower bound, i.e., there exists a positive constant $\epsilon$ such that $\|\h h^{(k)}\|_2\geq \epsilon, \ \forall k$. 
	\end{assumption} 
\tp{\remark{We  have  $\|\h h\|_2>0$ in the $L_1/L_2$ model as the denominator shall not be zero. It is true that $\|\h h\|_2>0$   does not imply a uniform lower bound of $\epsilon$ such that  $\|\h h\|_2>\epsilon$ in Assumption 1. Here we can redefine the divergence of an algorithm by including the case of $\|\h h^{(k)}\|_2< \epsilon$, which can be checked numerically with a pre-set value of $\epsilon$.}}
	
Unlike the unconstrained model \eqref{equ:split_model_uncon}, the \tp{strong} convexity of $\mathcal{L}(\h u, \h h, \h g)$ with respect to $\h u$ does not hold due to the indicator function $\Pi_{A\h u=\h b}(\h u)$. Besides, we have additional dual variable $\h w$ which is not in the unconstrained model. To avoid redundancies to  \cite{wang2020limitedct}, we focus on the different strategies to the unconstrained case, such as optimality conditions and subgradient of the indicator function, when proving convergence for the constrained problem, e.g., in Lemmas \ref{lem:suff_decr}-\ref{thm:stantionary_con} and \Cref{theo:box}.



	\begin{lemma}{(sufficient descent)}\label{lem:suff_decr}
		Under \Cref{assmp},
		the sequence $\{\h u^{(k)},\h h^{(k)},\h g^{(k)}\}$  generated by 
		\eqref{ADMML1overL2}
		satisfies 
		\begin{equation}\label{ineq:lem_suff_decr}
		\mathcal{L}(\h u^{(k+1)}, \h h^{(k+1)};\h g^{(k+1)})\le \mathcal{L}(\h u^{(k)}, \h h^{(k)}; \h g^{(k)})-c_1\|\h u^{(k+1)}-\h u^{(k)}\|_2^2-c_2\|\h h^{(k+1)} - \h h^{(k)}\|_2^2,
		\end{equation}
		where  $c_1$ and $c_2$ are two positive constants for a sufficiently large $\rho$.
	\end{lemma}
	
\begin{proof}
Denote $\sigma$ as the smallest eigenvalue of the matrix $A^T A+D^T D.$  We show $\sigma$ is strictly positive. If $\sigma=0,$ there exists \tp{a nonzero} vector $\h x$ such that  $\h x^T(A^T A+D^T D) \h x = 0$. It is straightforward that $\h x^T A^T A \h x\ge0$ and $\h x^T D^T D\h x\ge0$, so one shall have $\h x^T A^T A \h x=0$ and $\h x^T D^T D\h x=0$, which contradicts ${\cal N}(D)\bigcap{\cal N}(A)=\{\h 0\}$ in \Cref{assmp}. Therefore, there exists a positive $\sigma>0$ such that
		\[
		\h v^T(A^T A + D^TD) \h v \geq  \sigma \|\h v\|_2^2, \quad \forall \h v.
		\]
		By letting $\h v=\h u^{(k+1)}-\h u^{(k)}$ and using $A\h u^{(k+1)}=A\h u^{(k)}=\h b$, we have
\begin{equation}\label{ineq:grad_bound}
   		\|D (\h u^{(k+1)}-\h u^{(k)})\|_2^2\geq \sigma\|\h u^{(k+1)}-\h u^{(k)}\|_2^2. 
\end{equation}

  We express the $\h u$-subproblem in \eqref{ADMML1overL2} equivalently as
		\begin{align*}
	\h	u^{(k+1)} = \arg&\min_{\h u} \frac{\| D\h u \|_1}{\| \h h^{(k)} \|_2}+  \frac{\rho}{2}\|  D \h u - \h h^{(k)}+\h g^{(k)}\|_2^2 \\ &\mathrm{s.t.} \quad A\h u = \h b, p-\h u\leq \h 0,  \mathrm{and} \ \h u- q\leq \h 0.
		\end{align*}
	 The optimality conditions state that  $A\h u^{(k+1)}=\h b$ and there exist three sets of  vectors $\h w_i (i=1,2,3)$  such that
		\begin{equation}\label{eq:sub-grad-l14lem}
	\h	0\tp{=}	\frac  {\h p^{(k+1)}}{\| \h h^{(k)} \|_2} + \rho D^T(D\h u^{(k+1)}- \h h^{(k)}+\h g^{(k)})+ A^T \h w_1 - \h w_2+\h w_3,
		\end{equation}
		with $\h p^{(k+1)}\in\partial\|D\h  u^{(k+1)}\|_1$. 
By the complementary slackness, we have $\h w_2,\h w_3\geq \h 0$ and 
\begin{equation}\label{eq:slack}
  (p-\h u^{(k+1)})\odot \h w_2  =(\h u^{(k+1)}-q)\odot \h w_3=\h 0,  
\end{equation}
which also holds for $\h u^{(k)}.$ Using the definition of subgradient, $A\h u^{(k+1)}=A\h u^{(k)}=\h b$, and  \eqref{ineq:grad_bound}-\eqref{eq:slack}, we obtain that
		\begin{equation*}
		\begin{split}
		& \mathcal{L}(\h u^{(k+1)}, \h h^{(k)}; \h g^{(k)})- \mathcal{L}(\h u^{(k)}, \h h^{(k)}; \h g^{(k)})\\
		\leq &  \langle\frac{\h p^{(k+1)}}{\| \h h^{(k)} \|_2}, \h u^{(k+1)}-\h u^{(k)} \rangle+ \frac {\rho} 2 \|D \h u^{(k+1)}-\tp{\h f^{(k)}}\|_2^2-\frac {\rho} 2 \|D \h u^{(k)}-\tp{\h f^{(k)}}\|_2^2\\
		= & -\left\langle  \h w_1, A\h u^{(k+1)}-A\h u^{(k)}\right\rangle +\left\langle \h w_2, \h u^{(k+1)}- \h u^{(k)}\right\rangle- \tp{\langle\h w_3, \h u^{(k+1)}-\h u^{(k)}\rangle}\\
		& -\rho\langle D\h u^{(k+1)}-\tp{\h f^{(k)}}, D\h u^{(k+1)}-D\h u^{(k)}\rangle+ \frac {\rho} 2 \|D\h u^{(k+1)}-\tp{\h f^{(k)}}\|_2^2-\frac {\rho} 2 \|D\h u^{(k)}-\tp{\h f^{(k)}}\|_2^2\\
		= &-\frac {\rho} 2 \|D\h u^{(k+1)}-D\h u^{(k)}\|_2^2\leq - \frac {\sigma\rho} 2 \|\h u^{(k+1)}-\h u^{(k)}\|_2^2,
		\end{split}
		\end{equation*}
		where \tp{$\h f^{(k)} = \h h^{(k)}-\h g^{(k)}$}.
	The  bounds of $\mathcal{L}\left(\h u^{(k+1)}, \h h^{(k+1)}; \h g^{(k)}\right)-\mathcal{L}\left(\h u^{(k+1)}, \h h^{(k)}; \h g^{(k)}\right)$ 
	and $\mathcal{L}\left(\h u^{(k+1)}, \h h^{(k+1)}; \h g^{(k+1)}\right)-\mathcal{L}\left(\h u^{(k+1)}, \h h^{(k+1)}; \h g^{(k)}\right)$ exactly follow  \cite[Lemma 4.3]{wang2020limitedct} 
	for the unconstrained formulation,
	and hence we omit the rest of the proof.
\end{proof}	

\tp{\remark{\Cref{lem:suff_decr} requires $\rho$ to be   sufficiently large so that two parameters $c_1$ and $c_2$ are positive. Following the proof of \cite[Lemma 4.3]{wang2020limitedct}, $c_1$ and $c_2$  can be explicitly expressed as 
\begin{equation}
    c_1 = \frac{\sigma \rho}{2}-\frac{16\tp{N}}{\rho \epsilon^4} \quad \text{ and }\quad 
 c_2=\frac{\rho \epsilon^3-6M}{2\epsilon^3}-\frac{16M^2}{\rho \epsilon^6},\end{equation}
where $M = \sup_k \|D \h u^{(k)}\|_2$. Note that the assumption on $\rho$ is a sufficient condition  to ensure the convergence, and we observe  in practice that a relatively small $\rho$ often yields good performance. }}
	
		\begin{lemma}{(subgradient bound)}\label{thm:stantionary_con}
		Under \Cref{assmp}, there exists  a vector $\bm\eta^{(k+1)} \in \partial {\cal L}(\h u^{(k+1)}, \h h^{(k+1)}; \h g^{(k+1)})$ and a constant $\gamma>0$
		such that
		\begin{eqnarray}\label{eq2}\|\bm\eta^{(k+1)}\|_2^2\le \gamma \left( \|\h h^{(k+1)}-\h h^{(k)}\|_2^2 + \|\h g^{(k+1)}-\h g^{(k)}\|_2^2\right).
		\end{eqnarray}
	\end{lemma}

	\begin{proof}
We define
		\begin{eqnarray*}\label{con:equ:sym2}
		\bm{\eta}_1^{(k+1)} :=  \frac{\h p^{(k+1)}}{\|\h h^{(k+1)}\|_2}+A^T \h w_1 - \h w_2+\h w_3+\rho D^T(D\h  u^{(k+1)}-\h h^{(k+1)}+\h  g^{(k+1)}).
		\end{eqnarray*}
Clearly by the subgradient definition, we can prove that $A^T\h w_1\in \partial \Pi_{A\h u=\h b}(\h u^{(k+1)})$ and $\h w_3-\h w_2\in \partial \Pi_{\tp{[p,q]^N}}(\h u^{(k+1)}), $ 
which implies that $\bm\eta_1^{(k+1)}  \in \partial_{\h u}{\cal L}(\h u^{(k+1)},\h h^{(k+1)},\h g^{(k+1)}).$ 
		Combining the definition of $\bm{\eta_1}^{(k+1)}$ with \eqref{eq:sub-grad-l14lem} leads to
		\begin{equation}    \label{equ:eta_bound}
		\bm	\eta_1^{(k+1)}= -\frac{\h p^{(k+1)}}{\|\h h^{(k)}\|_2}+\frac{\h p^{(k+1)}}{\|\h h^{(k+1)}\|_2}+\rho D^T \left(\h h^{(k)}-\h h^{(k+1)}\right)+\rho D^T\left(\h g^{(k+1)}-\h g^{(k)}\right).
		\end{equation}
\tp{To estimate an  upper bound of $\|\bm	\eta_1^{(k+1)}\|_2,$ we apply the chain rule of sub-gradient, i.e., $\partial \|D \h u\|_1 =D^T \h q$, where $\h q = \{\h q\ | \ \langle \h q, D\h u \rangle=\|D \h u\|_1, \|\h q\|_\infty\leq 1\},$ thus leading to $\|\h p^{(k+1)}\|_2\leq \|D^T\|_2\|\h q^{(k+1)}\|_2 \leq 4 \sqrt{N}$. Therefore, we have
\begin{equation*}
        \left\|\frac{\h p^{(k+1)}}{\|\h h^{(k+1)}\|_2} - \frac{\h p^{(k+1)}}{\|\h h^{(k)}\|_2} \right\|_2\leq \frac{1}{\epsilon^2} \| \h h^{(k+1)}-\h h^{(k)}\|_2\|\h p^{(k+1)}\|_2\leq \frac{4\sqrt{\tp{N}}}{\epsilon}\|\h h^{(k+1)}-\h h^{(k)}\|_2. 
\end{equation*} 
It further follows from \eqref{equ:eta_bound} that 
\begin{equation}
    \|\bm	\eta_1^{(k+1)}\|_2 \leq (\frac{4\sqrt{N}}{\epsilon}+2\sqrt{2}\rho ) \left\|\h h^{(k)}-\h h^{(k+1)}\right\|_2+2\sqrt{2}\rho \left\|\h g^{(k+1)}-\h g^{(k)}\right\|_2.
\end{equation}
}
 
	We can also define $\bm\eta_2^{(k+1)}, \bm\eta_3^{(k+1)}$ such that 
				\begin{equation*}
			\begin{split}
				\bm\eta_2^{(k+1)} & \in \partial_{\h h}{\cal L}(\h u^{(k+1)},\h h^{(k+1)},\h g^{(k+1)})\\
				\bm\eta_3^{(k+1)} & \in \partial_{\h g}{\cal L}(\h u^{(k+1)},\h h^{(k+1)},\h g^{(k+1)}),
			\end{split}
		\end{equation*}
and estimate the upper bounds of 	$\|\bm\eta_2^{(k+1)}\|_2$ and $\|\bm\eta_3^{(k+1)}\|_2.$ 
\tp{By denoting $\bm \eta^{(k+1)} = (\bm \eta_1^{(k+1)}, \bm \eta_2^{(k+1)}, \bm \eta_3^{(k+1)})\in \partial {\cal L}(\h u^{(k+1)}, \h h^{(k+1)}; \h g^{(k+1)}), $ }
the remaining proof is the same as in \cite[Lemma 4.4]{wang2020limitedct}.
	\end{proof}

	
\begin{theorem}\label{theo:box}
		(subsequential convergence)
		Under Assumption 1 and a sufficiently large $\rho$,  the sequence \tp{$\{\h u^{(k)}, \h h^{(k)}, \h g^{(k)} \}$}  generated by
\eqref{ADMML1overL2}
		always has a subsequence  convergent to a \tp{stationary point $(\h u^\ast, \h h^\ast, \h g^\ast)$ of $\mathcal{L}$, namely, $\h 0\in \partial \mathcal{L}(\h u^\ast, \h h^\ast, \h g^\ast)$. } 
	\end{theorem}
	\begin{proof}
		\tp{Since $\h u^{(k)}\in[p,q]^N$ is bounded,} then  $\|D \h u^{(k)}\|_1$ is bounded; i.e., there exists a constant $M>0$ such that  $\|D \h u^{(k)}\|_1\leq M.$
	The optimality condition of the $\h h$-subproblem in \eqref{ADMML1overL2} leads to 
		\begin{equation}\label{eq:opt4hsub}
		-\dfrac{a^{(k+1)}}{\tp{\|\h h^{(k+1)}\|^3_2}}\h h^{(k+1)} + \rho\left(\h h^{(k+1)} -D\h u^{(k+1)}-\h g^{(k)}\right) =\h 0,
		\end{equation}
		where $a^{(k)}:=\|D \h u^{(k)}\|_1$. Using the dual update $- \h g^{(k+1)}=\h h^{(k+1)} -D\h u^{(k+1)}-\h g^{(k)},$  we have
		\begin{equation}
		\h g^{(k+1)}=-\frac{a^{(k+1)}}{\rho}\frac{\h h^{(k+1)}}{\|\h h^{(k+1)}\|^3_2}.
		\end{equation}
Due to $\|\h h^{(k)}\|_2\geq \epsilon$ in \Cref{assmp}, we get
		\[
		\|\h g^{(k)}\|_2=\left\|\frac{\tp{a^{(k)}}}{\rho}
		\frac{\h h^{(k)}}{\tp{\|\h h^{(k)}\|^3_2}}\right\|_2\leq\frac{M}{\rho\epsilon^2},
		\]
	which implies the boundedness of $\{\h g^{(k)}\}$. 	It follows from the $\h h$-update \eqref{eq:update-h} that 
	$\{\h h^{(k)}\}$ is also bounded. 
Therefore, the Bolzano-Weierstrass Theorem guarantees that the sequence  $\{\h u^{(k)},\h h^{(k)},\h g^{(k)}\}$ has a  convergent subsequence,
		denoted by $(\h u^{(k_j)},\h h^{(k_j)},\h g^{(k_j)})\rightarrow (\h u^*,\h h^*,\h g^*),$ as $k_j\rightarrow \infty$.	
		In addition, we can estimate that
		\begin{equation*}
			\begin{split}
				& \mathcal{L}(\h u^{(k)}, \h h^{(k)}; \h g^{(k)})\\ 
				= & \frac{\|D \h u^{(k)} \|_1}{\| \h h^{(k)} \|_2}+\Pi_{A\h u=\h b}(\h u^{(k)})+
				\Pi_{\tp{[p, q]^N}}(\h u^{(k)}) +\frac{\rho}{2}\| \h h^{(k)} - \tp{D} \h u^{(k)}-\h g \|_2^2- \frac{\rho}{2}\|\h g^{(k)}\|_2^2\\
				\geq & \frac{\| D \h u^{(k)} \|_1}{\| \h h^{(k)} \|_2} - \frac{M^2}{2\rho\epsilon^4},
			\end{split}
		\end{equation*}
		which gives a lower bound of ${\cal L}$ owing to the boundedness of $\h u^{(k)}$ and $\h h^{(k)}$.
	It further follows from  \Cref{lem:suff_decr} that $\mathcal L(\h u^{(k)},\h h^{(k)},\h g^{(k)})$ converges due to its \tp{monotonicity}. 	
		
		We then  sum the inequality \eqref{ineq:lem_suff_decr}  from $k=0$ to $K$, thus getting
		\begin{eqnarray*}
			&&	\mathcal{L} (\h u^{(K+1)}, \h h^{(K+1)}; \tp{\h g^{(K+1)}})\\
			&\le& \mathcal{L} (\h u^{(0)}, \h h^{(0)}; \h g^{(0)})-c_1\sum_{k=0}^K\|\h u^{(k+1)}-\h u^{(k)}\|_2^2-c_2\sum_{k=0}^K\|\h h^{(k+1)} - \h h^{(k)}\|_2^2.
		\end{eqnarray*}
		Let $K\rightarrow \infty,$ we have both summations of $\sum_{k=0}^{\infty}\|\h u^{(k+1)}-\h u^{(k)}\|_2^2$ and $\sum_{k=0}^{\infty}\|\h h^{(k+1)} - \h h^{(k)}\|_2^2$ are  finite, indicating that
		$\h u^{(k)}-\h u^{(k+1)}\rightarrow 0$, $\h h^{(k)} - \h h^{(k+1)}\rightarrow 0$. Then by \cite[Lemma 4.2]{wang2020limitedct}, we get $\h g^{(k)}-\h g^{(k+1)} \rightarrow 0 $.
		By  $(\h u^{(k_j)},\h h^{(k_j)},\h g^{(k_j)})\rightarrow (\h u^*,\h h^*,\h g^*),$ we have  $(\h u^{(k_j+1)},\h h^{(k_j+1)},\h g^{(k_j+1)})\rightarrow (\h u^*,\h h^*,\h g^*), A\h u^*=\h b$ (as $A\h u^{(k_j)}=\h b$), and $D \h u^*=\h h^*$ (by the update of $\h g$). It further follows from Lemma \ref{thm:stantionary_con} that $\h 0\in \tp{\mathcal \partial L(\h u^*,\h h^*, \h g^*)}$ and hence $(\h u^*,\h h^*, \h g^\ast)$ is a \tp{stationary point of \eqref{eq:augL4con}. }
	\end{proof}

\tp{
Lastly, we discuss the global convergence, i.e., the entire sequence converges, which is stronger than the subsequential convergence as in
\Cref{theo:box}, under a stronger assumption that  the augmented Lagrangian $\mathcal L$ has the Kurdyka-\L ojasiewicz (KL) property \cite{bolte2007lojasiewicz}; see \Cref{def:KL}.
The global convergence of the proposed scheme \eqref{ADMML1overL2} is characterized in \Cref{theowhole}, which   can be proven in a similar way as \cite[Theorem 4]{li2015global}.
Unfortunately, the KL property is an open problem for the $L_1/L_2$ functional, not to mention $L_1/L_2$ on the gradient.
\begin{definition} \label{def:KL}(KL property, \cite{KLproperties})
	We say a proper closed function
$h:\!\mathbb R^n\rightarrow(-\infty,+\infty]$
	satisfies the KL property at a point ${\hat {\h x}}\in {\text{\rm dom}}\partial h$ if there exist a constant $\nu\in(0,\infty]$, a neighborhood
	$U$ of ${\hat {\h x}},$ and a continuous concave function $\phi:\;[0,\nu)\rightarrow[0,\infty)$ with
	$\phi(0)=0$ such that
	\begin{itemize}
		\item[(i)] $\phi$ is continuously differentiable on $(0,\nu)$ with $\phi'>0$ on $(0,\nu)$;
		\item[(ii)] for every $\h x\in U$ with $h({\hat {\h x}})< h(\h x)< h(\hat {\h x}) +\nu$, it holds that
		\begin{eqnarray*}\phi'(h(\h x)-h(\hat {\h x})){\text{\rm{dist}}}(\h 0,\partial h({\h x}))\ge 1, \end{eqnarray*}
	\end{itemize}
	where $ \text{\rm dist}(\h x, C) $ denotes the distance  from a point $\h x$ to a closed set $C$ measured in $\|\cdot \|_2$ with a convention of $\text{\rm dist}(\h 0,\emptyset ):=+\infty$.
\end{definition}
\begin{theorem}\label{theowhole}
(global convergence) Under the Assumption 1 and a sufficiently large $\rho$,
		the sequence $\{\h u^{(k)},\h h^{(k)},\h g^{(k)}\}$  generated by
		\eqref{ADMML1overL2}.
 If the augmented Lagrangian $\mathcal L$ has the KL property,
$\{\h u^{(k)},\h h^{(k)},\h g^{(k)}\}$ converges to a stationary point of (\ref{eq:augL4con}). 
\end{theorem}
}
\begin{proof} \tp{The proof is almost the same as \cite[Theorem 4]{li2015global}, thus omitted here.}\end{proof}

	\section{Experimental results}
	\label{sec:experiments}

	In this section, we test the proposed algorithm on three prototypical imaging applications: super-resolution, MRI reconstruction, and limited-angle CT reconstruction. As analogous to  Section~\ref{sec:toy_example},  super-resolution refers to recovering a 2D image from low-frequency measurements, i.e., we restrict the data within a square in the center of the frequency domain. The data measurements for the MRI reconstruction are taken along \tp{radial} lines in the frequency domain; such a \tp{radial} pattern \cite{candes2006robust} is referred to as a mask. The sensing matrix for the CT reconstruction is the Radon transform \cite{AM01}, while the term ``limited-angle'' means the rotating angle does not cover the entire circle \cite{gwang_limitedct,zhang2006comparative,wang2011low}.

	We evaluate the performance in terms of the relative error (RE) and  \tp{the peak signal-to-noise ratio (PSNR), defined by }
	\begin{equation*}
		\text{RE}(\h u^\ast,\tilde{\h u}) := \frac{\|\h u^\ast-\tilde{\h u}\|_2}{\|\tilde{\h u}\|_2} \quad \text{and} \quad 	   \tp{ {\rm PSNR}(\h u^\ast, \tilde{\h u}) := 10 \log_{10} \frac{N P^2}{\|\h u^\ast-\tilde{\h u}\|_2^2},}
	\end{equation*} 
	where $\h u^\ast$ is the restored image, $\tilde{\h u}$ is the ground truth, and \tp{$P$ is the maximum peak value of $\tilde{\h u}.$} 
	
To ease with parameter tuning, 
we scale the pixel value  to $[0,1]$ for the original images in each application and rescale the solution back after computation. Hence the box constraint is set as $[0,1]$. We start by discussing some algorithmic behaviors regarding  the box constraint, the maximum number of inner iterations, and \tp{sensitivity analysis} on algorithmic parameters in \Cref{sect:exp_alg}. The remaining sections are organized by specific applications. We compare the proposed $L_1/L_2$ approach with total variation ($L_1$ on the gradient) \cite{rudinOF92} \tp{and two  nonconvex regularizations: $L_p$ for $p=0.5$  and $L_1$-$\alpha L_2$ for $\alpha = 0.5$  on the gradient as suggested in \cite{louZOX15}. To solve for the $L_p$ model, we replace the soft shrinkage  \eqref{ADMM_l1con_d} by the proximal operator corresponding to $L_p$ that was derived in \cite{Xu2012}, and apply the same ADMM framework as the $L_1$ minimization. } 
To have a fair comparison,  we incorporate the $[0,1]$ box constraint in $L_1$, $L_p$, $L_1$-$\alpha L_2$, and $L_1/L_2$ models.  We implement all these competing methods by ourselves and tune the parameters to achieve the smallest RE to the ground-truth. Due to the constrained formulation, no noise is added.
We set the initial condition of $\h u$ to be a zero vector for all the methods. The stopping criterion  for the proposed Algorithm~\ref{alg:l1dl2_box} is when the relative error between two consecutive iterates is smaller than $\epsilon=10^{-5}$ for both inner and outer iterations. 
All the numerical experiments are carried out in a  desktop with CPU (Intel i7-9700F, 3.00 GHz) and MATLAB 9.8 (R2020a).

\subsection{Algorithmic behaviors}\label{sect:exp_alg}

	We  discuss three computational aspects of the proposed
\Cref{alg:l1dl2_box}. In particular, we want to analyze the influence of the box constraint, the maximum number of inner iterations (denoted by jMax), and the algorithmic parameters on the reconstruction results of MRI and CT problems. We use MATLAB's built-in function {\tt phantom}, which is  called  the Shepp-Logan (SL) phantom, to test on 6 \tp{radial} lines for MRI and 
$45^\circ$ scanning range for CT.   
	The analysis is assessed in terms of objective values $\frac{\| D \h u^{(k)} \|_1}{\| D \h u^{(k)} \|_2}$ and RE$(\h u^{(k)}, \tilde{\h u})$ versus the CPU time.

	\begin{figure}[t]
		\begin{center}
			\begin{tabular}{cc}
				MRI (jMax = 5) & CT (jMax = 1)\\
				\includegraphics[width=0.4\textwidth]{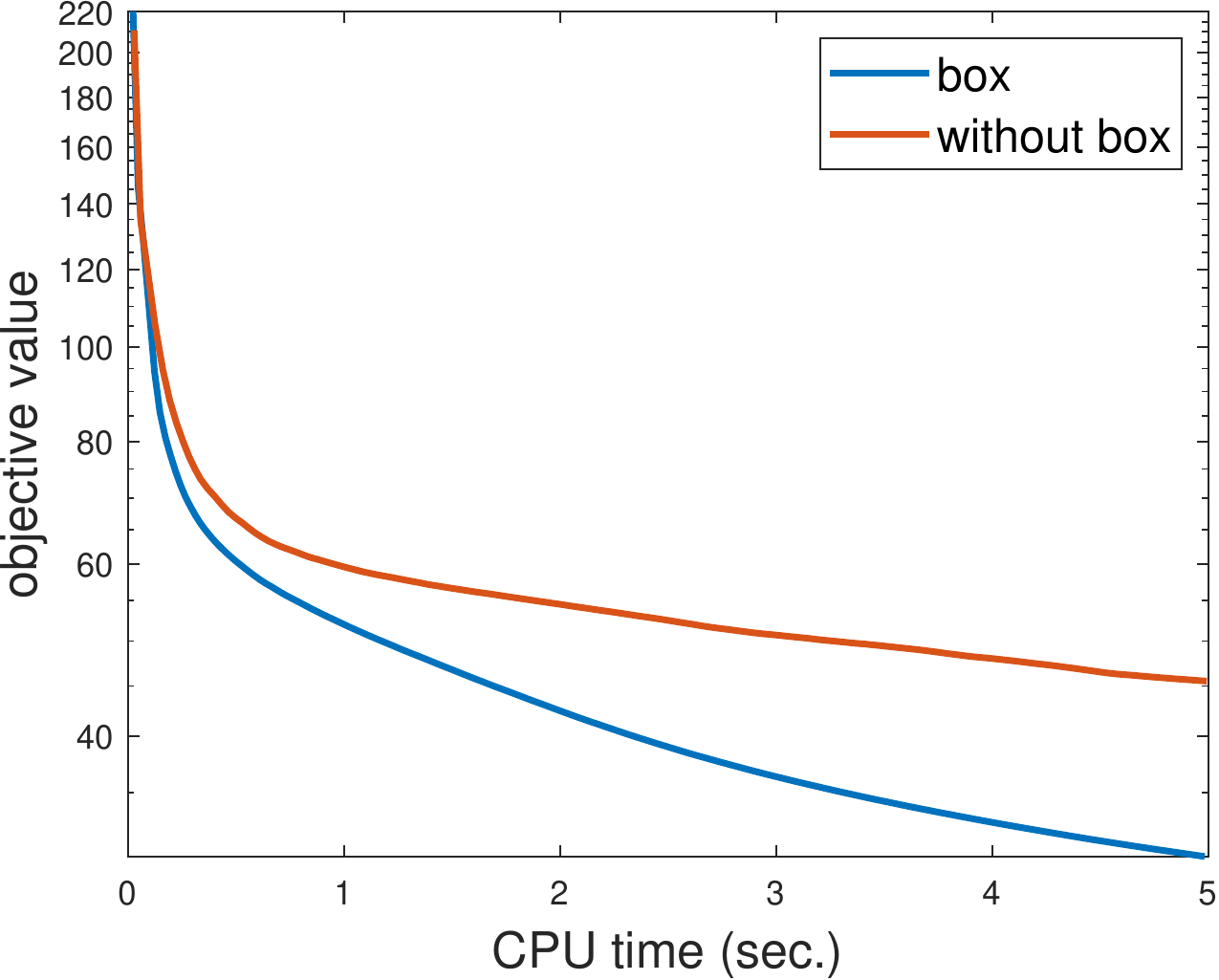} &
				\includegraphics[width=0.4\textwidth]{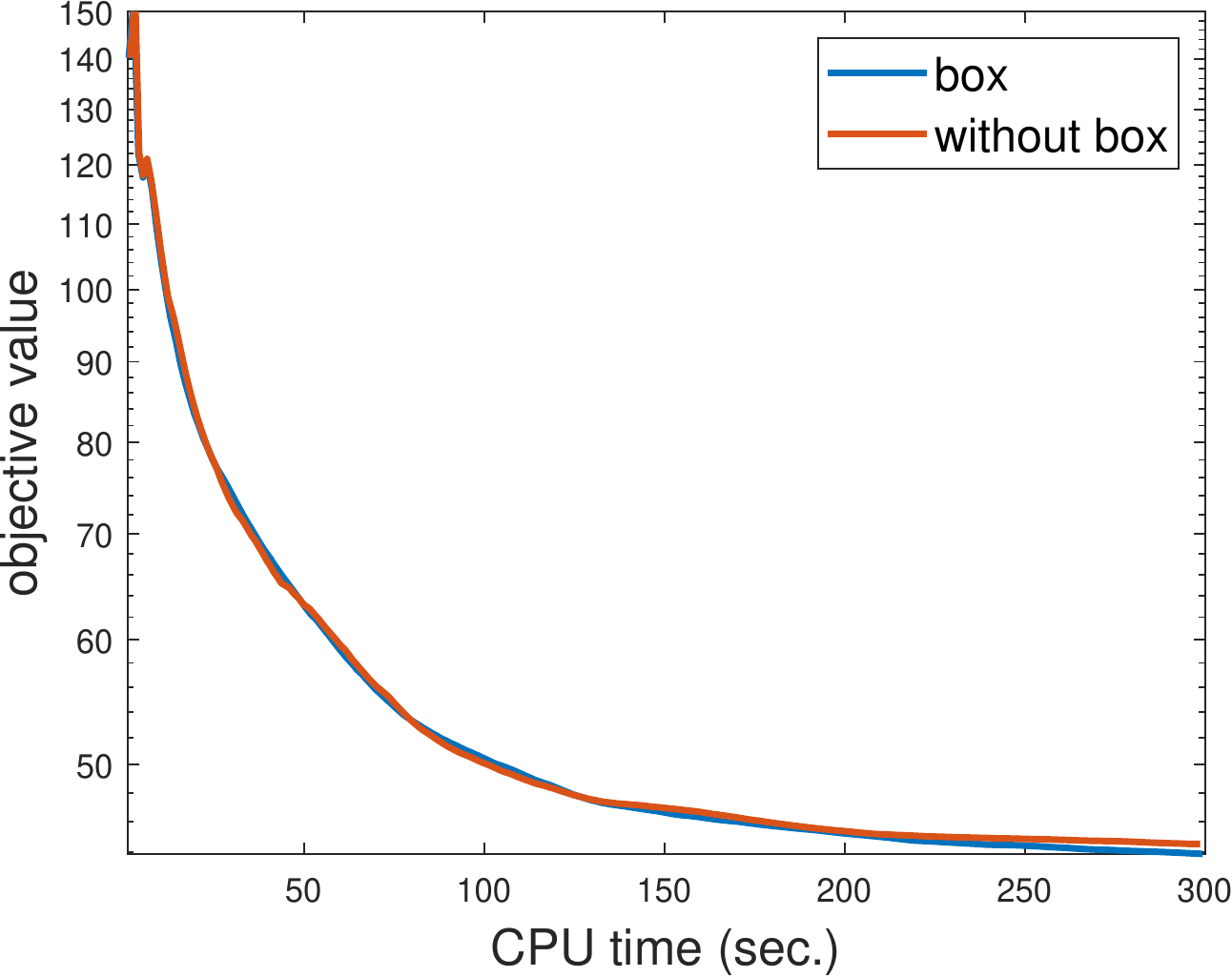}  \\
				\includegraphics[width=0.4\textwidth]{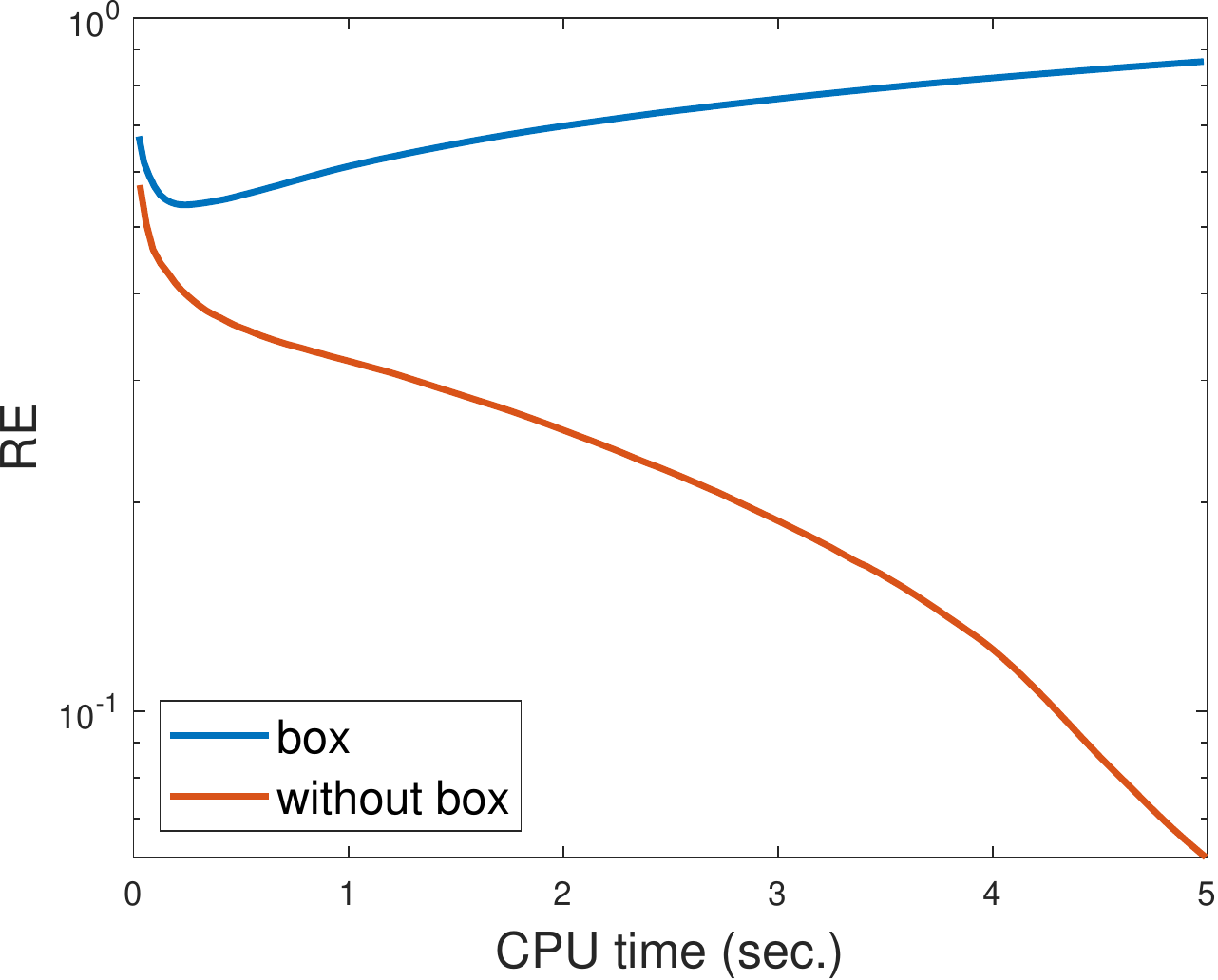} &
				\includegraphics[width=0.4\textwidth]{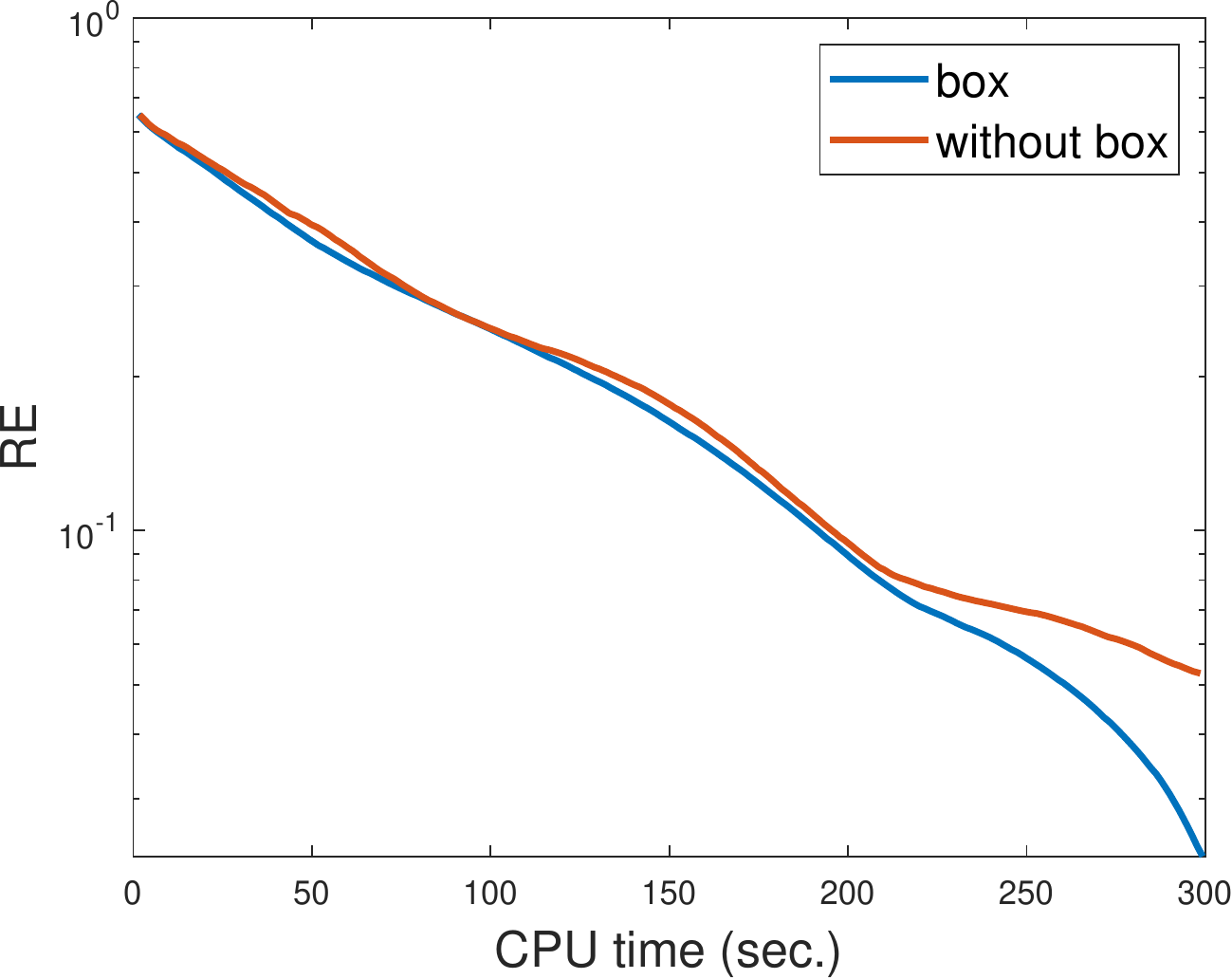} 
			\end{tabular}
		\end{center}
		\caption{The effects of  box constraint on   objective values (top) and relative errors (bottom) for MRI (left) and CT (right) reconstruction problems. }\label{fig:effect_box}
	\end{figure}

In \Cref{fig:effect_box}, we present algorithmic behaviors of the box constraint for both MRI and CT problems, in which we set jMax to be 5 and 1, respectively (we will discuss the effects of inner iteration number shortly.)   In the MRI problem, the box constraint is  critical; without it, our algorithm converges to another local minimizer, as RE goes up. With the box constraint, the objective values decrease faster than in the no-box case, and the relative errors drop down monotonically. In the CT case, the influence of box is minor but we can see a faster decay of RE  than the no-box case after 200 seconds.  In the light of these observations,  we only consider the algorithm with a box constraint for the rest of the experiments. 
	
	\begin{figure}[t]
		\begin{center}
			\begin{tabular}{cc}
				MRI (jMax = 1,3,5,10) & CT (jMax = 1,3,5,10)\\
				\includegraphics[width=0.4\textwidth]{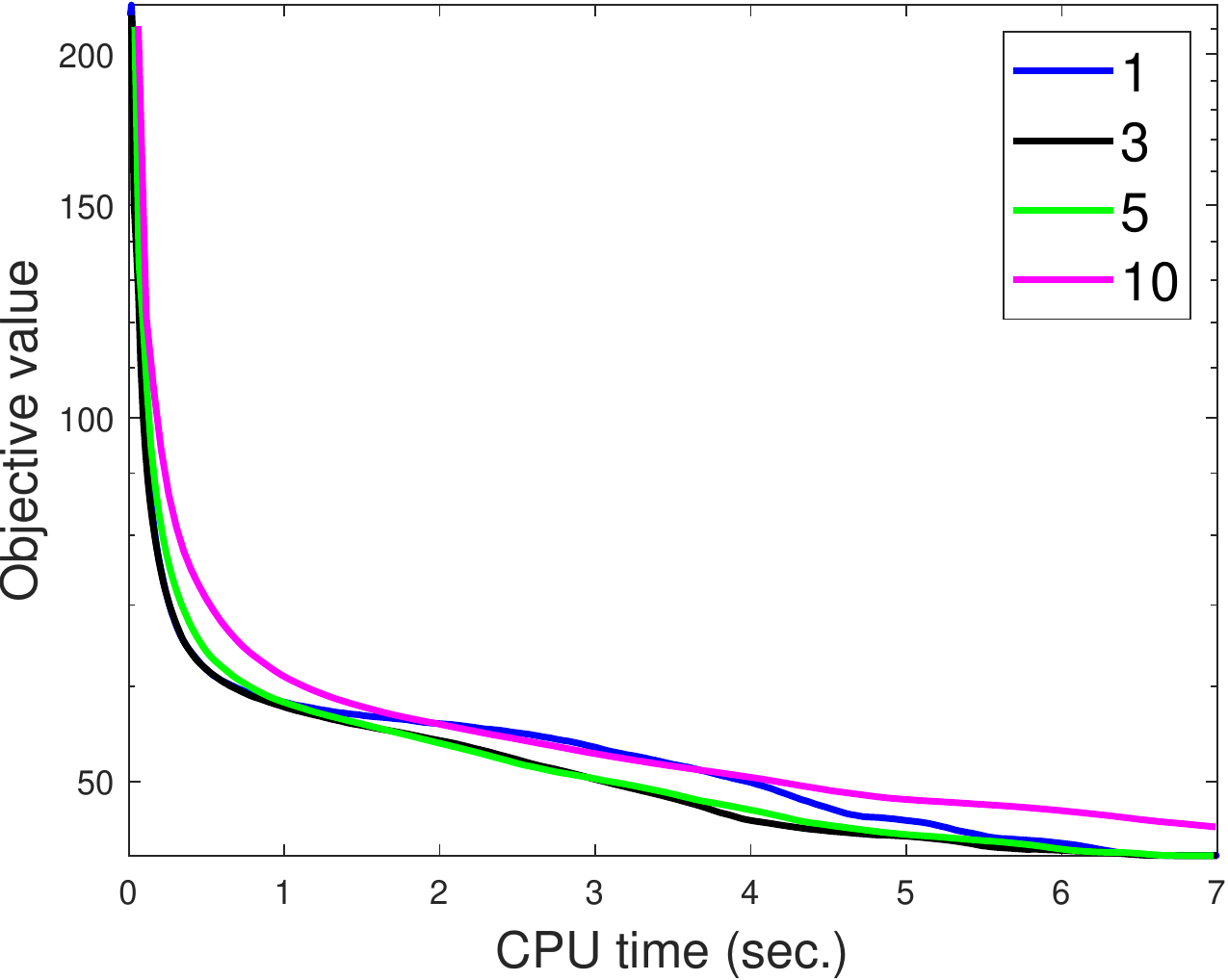} &
				\includegraphics[width=0.4\textwidth]{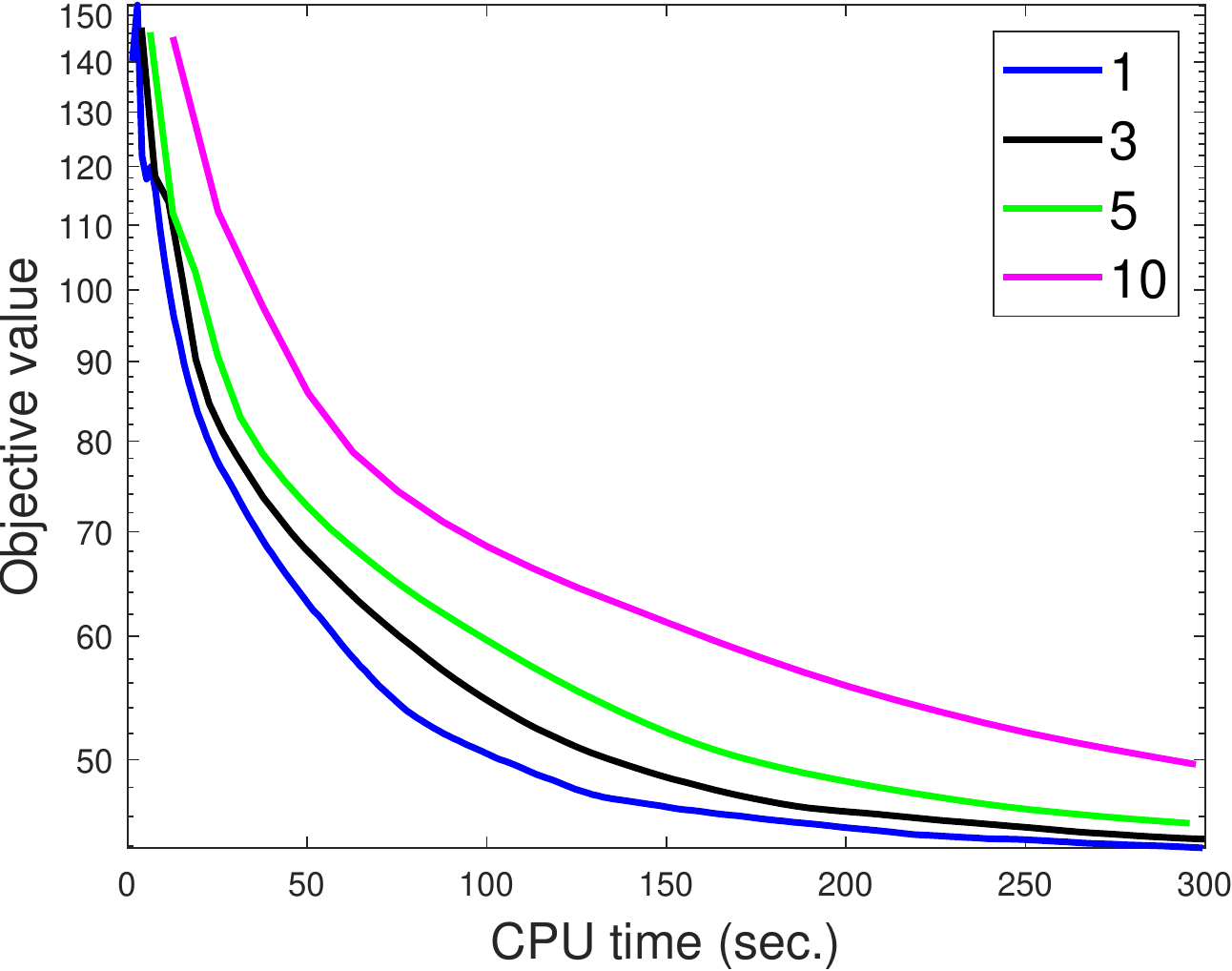}  \\
				\includegraphics[width=0.4\textwidth]{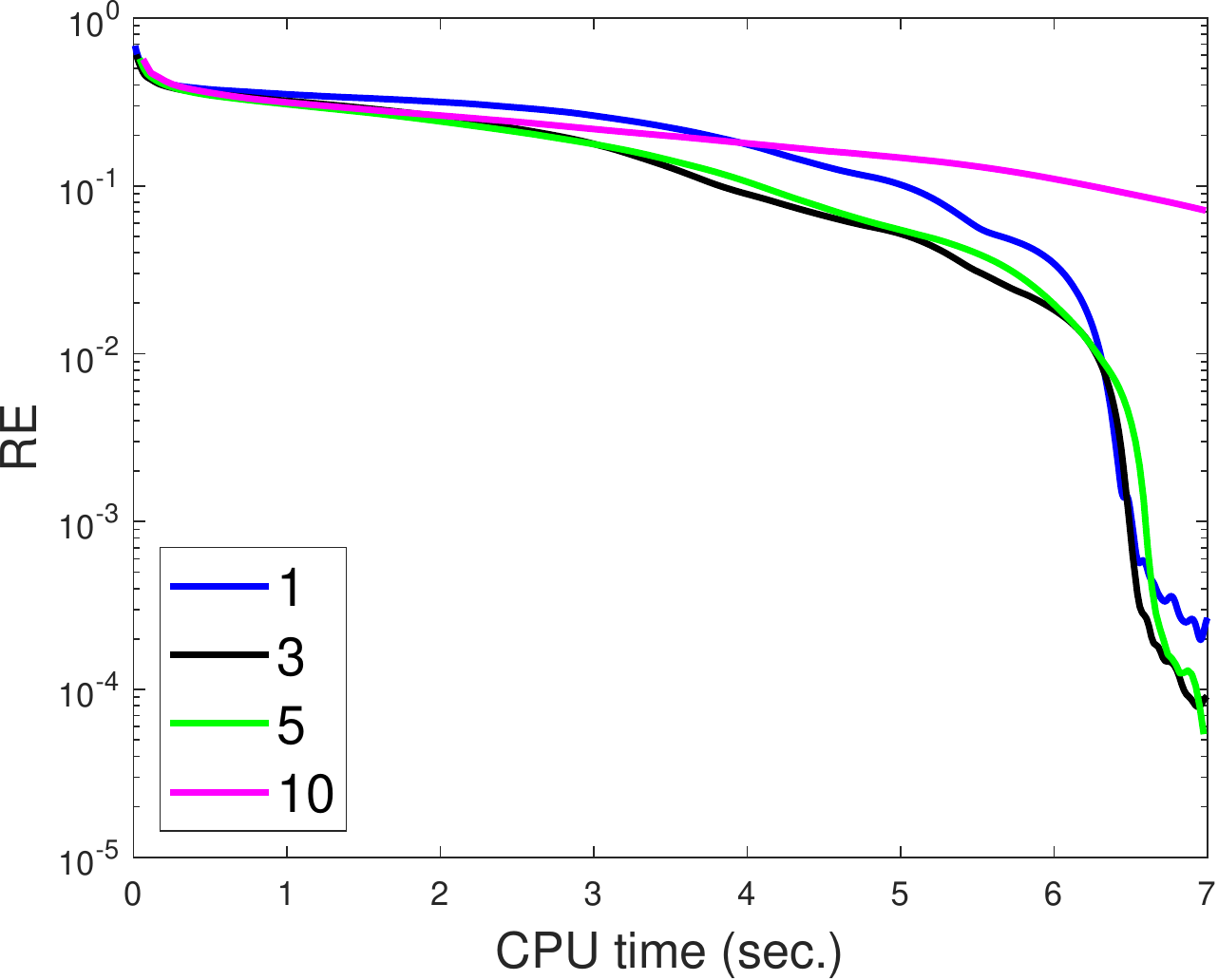} &
				\includegraphics[width=0.4\textwidth]{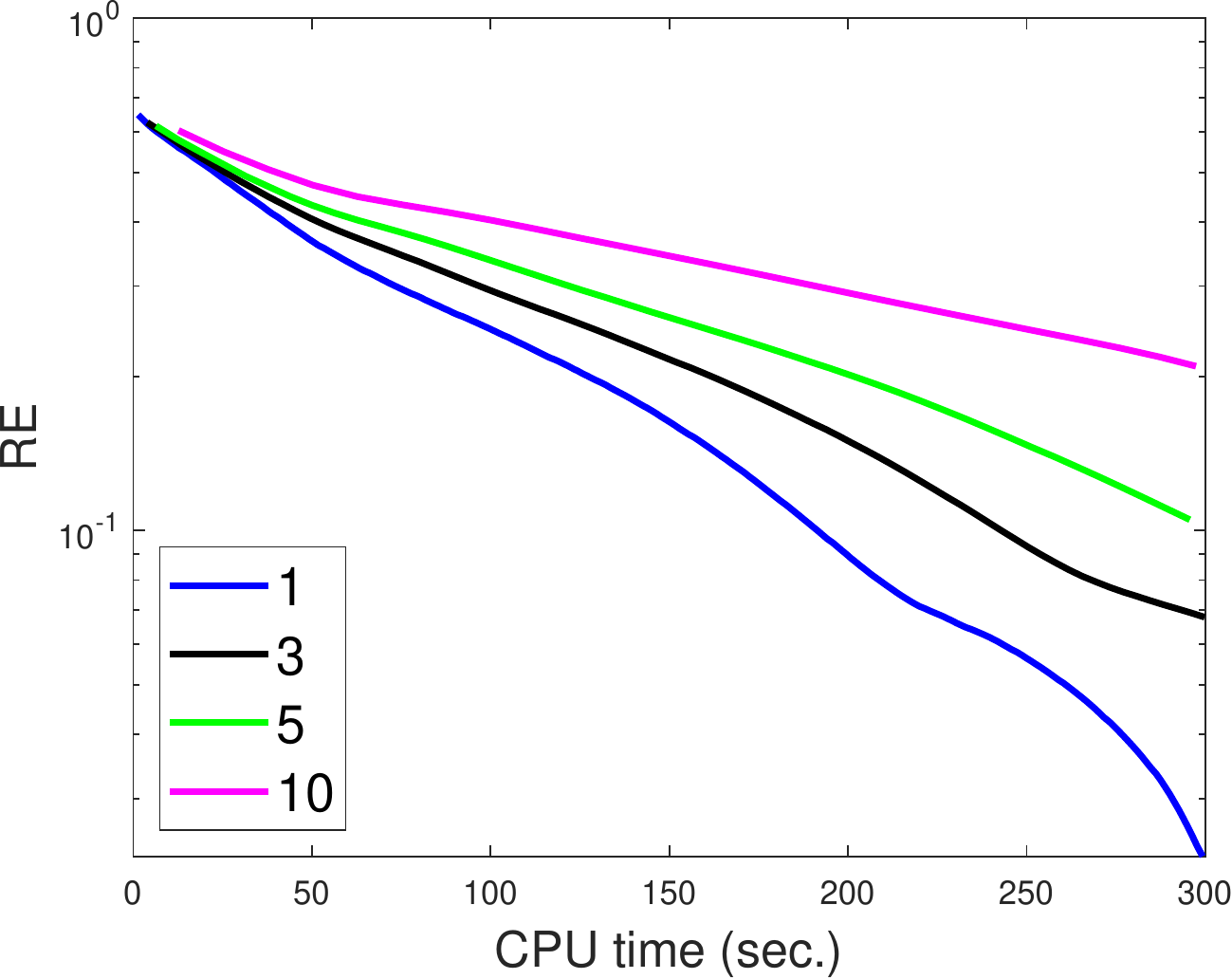} 
			\end{tabular}
		\end{center}
		\caption{The effects of  the maximum number in the inner loops (jMax) on objective values (top) and  relative errors (bottom) for MRI (left) and CT (right) reconstruction problems. }\label{fig:effect_inner}
	\end{figure}

We then study the effect of jMax on MRI/CT reconstruction problems in \Cref{fig:effect_inner}. 
We 	fix the maximum outer iterations as 300, and examine four possible jMax values: 1, 3, 5 and 10.  
	 In the case of MRI,  jMax = 10 causes the slowest decay of both objective value and RE. Besides, we observe that only one inner iteration, which is equivalent to our previous approach \cite{l1dl2}, is not as efficient as  more inner iterations to reduce the RE in the MRI problem. The CT results are slightly different, as one inner iteration seems sufficient to yield satisfactory results. 
	 The disparate behavior of CT to MRI is probably due to inexact solutions by CG iterations. In other words, more inner iterations do not improve the accuracy. 
	 Following \Cref{fig:effect_inner}, 
	we set jMax to be 5 and 1 in MRI and CT, respectively, for the rest of the experiments.

Lastly, we study the sensitivity of the parameters $\lambda, \rho,\beta$ in our proposed algorithm to provide strategies for parameter selection. For simplicity, we set $\gamma=\rho$  as their corresponding auxiliary variables represent $D\h u$.  In the MRI reconstruction problem, we examine three values of $\lambda\in\{ 100, 1000, 10000\}$ and two settings of the number of maximum outer iterations, i.e., kMax $\in\{500, 1000\}.$ For each combination of $\lambda$ and kMax, we vary parameters $(\rho, \beta)\in (2^i,2^j), $ for \tp{$i, j \in [-4, 4], $}
and   plot the RE  in \Cref{fig:sa_param_MRI}. We observe that small values of $\rho$ work well in practice, although we need to assume a sufficiently large value for $\rho$ when proving the convergence results in  \Cref{theo:box}. Besides, a larger kMax value leads to 
larger valley regions for the lowest RE, which verifies that only $\rho$ and $\beta$ affect the convergence rate. \Cref{fig:sa_param_MRI} suggests that our algorithm is generally insensitive to all these parameters $\beta, \rho$ and $\lambda$ as long as $\rho$ is small. 
Similarly in the  CT reconstruction, we set $\lambda\in \{0.005, 0.05, 0.5\},$ kMax $\in\{100,300\}$, and $(\rho, \beta)\in (2^i,2^j), $ for \tp{$i, j \in [-4, 4], $}
\Cref{fig:sa_param_CT} shows that $\rho$ and $\beta$ can be selected in a wide range, especially for large number of outer iterations. But our algorithm is sensitive to $\lambda$ for the CT problem, as $\lambda=0.005$ or $0.5$ yields larger errors than $\lambda=0.05$. 
In the light of this \tp{sensitivity analysis}, we can tune  parameters by finding the optimal combination among a candidate set for each problem, specifically paying attention to the value of $\lambda$ in the limited-angle CT reconstruction.

\begin{figure}[t]
		\begin{center}
			\begin{tabular}{ccc}
				$\lambda=100$ & $\lambda=1000$ & $\lambda=10000$ \\
				\includegraphics[width=0.305\textwidth]{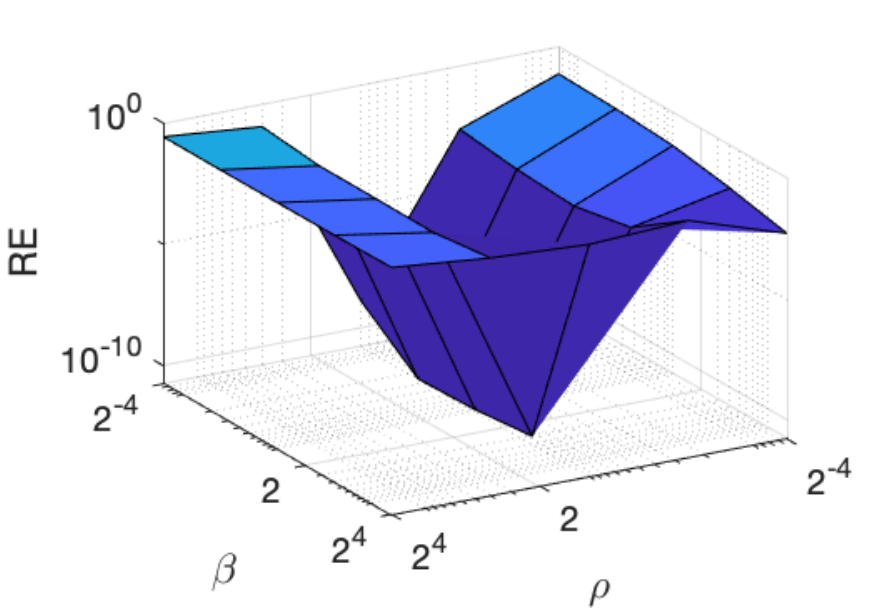} &
				\includegraphics[width=0.305\textwidth]{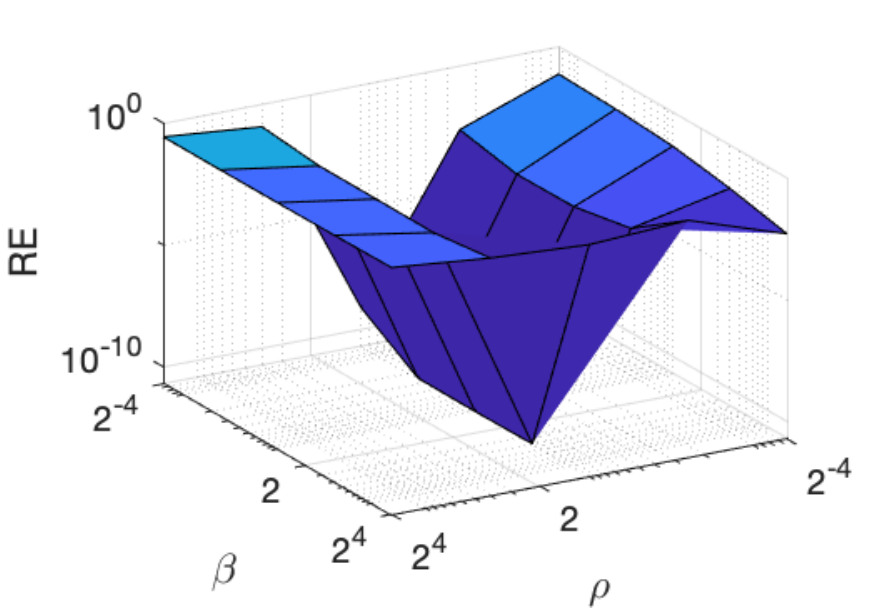} &
				\includegraphics[width=0.305\textwidth]{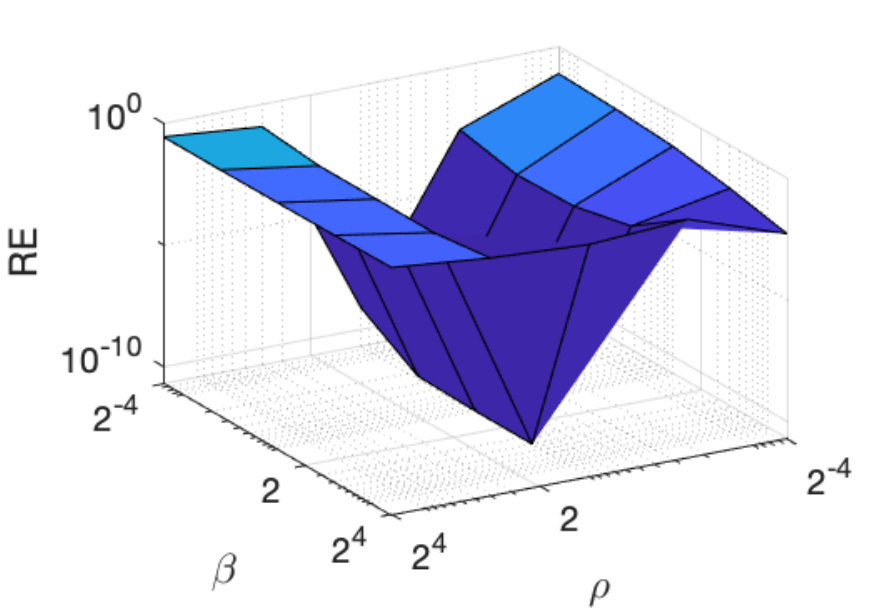}  \\
				\includegraphics[width=0.305\textwidth]{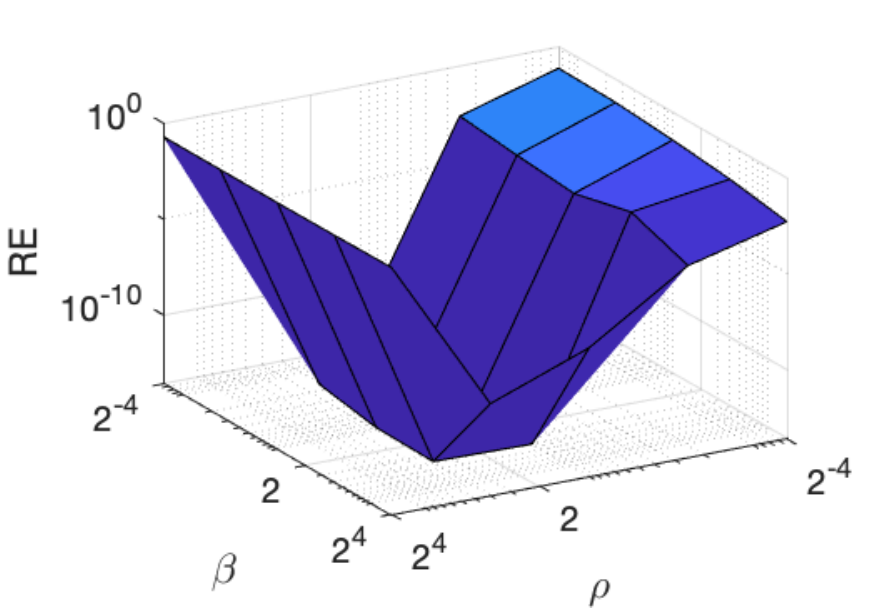} &
				\includegraphics[width=0.305\textwidth]{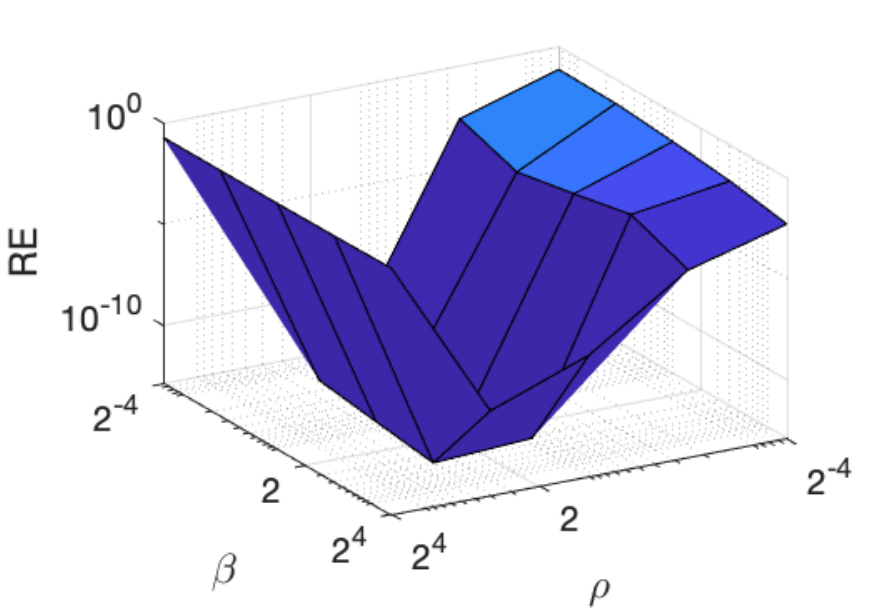} &
				\includegraphics[width=0.305\textwidth]{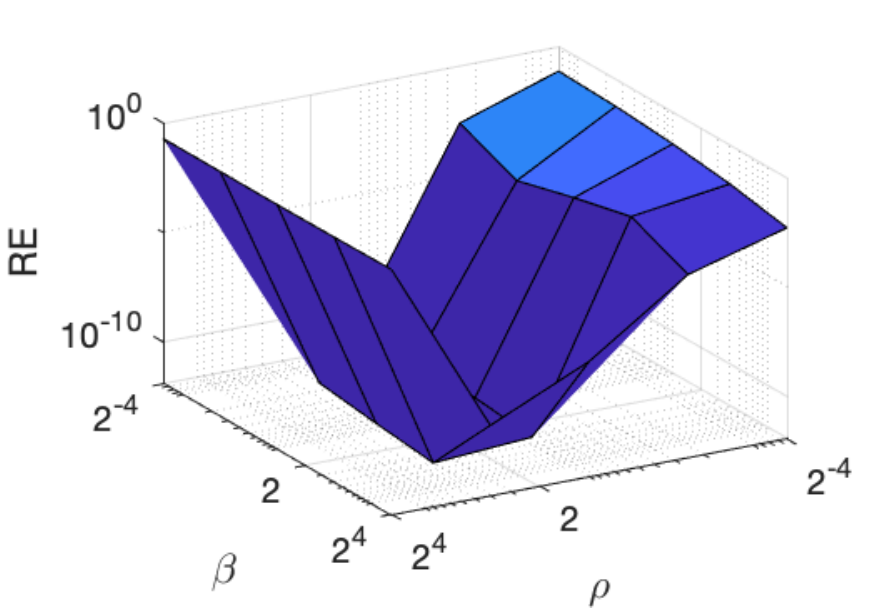}  
			\end{tabular}
		\end{center}
		\caption{\tp{The relative errors  with respect to the parameters $\lambda, \rho,\beta$ in \Cref{alg:l1dl2_box} for MRI reconstruction when kMax is 500 (top) or 1000  (bottom).} }\label{fig:sa_param_MRI}
	\end{figure}
	
	\begin{figure}[t]
		\begin{center}
			\begin{tabular}{ccc}
				$\lambda=0.005$ & $\lambda=0.05$ & $\lambda=0.5$ \\
				\includegraphics[width=0.30\textwidth]{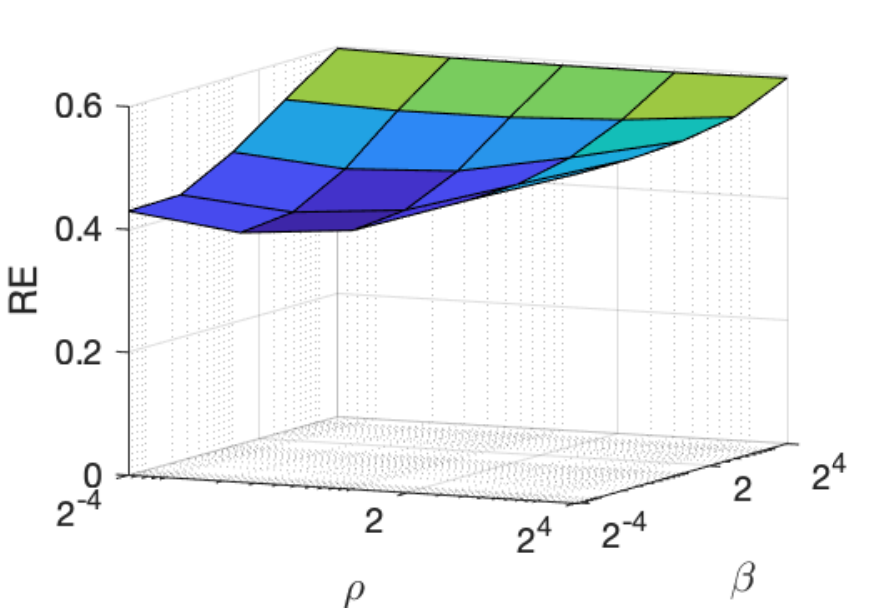} &
				\includegraphics[width=0.30\textwidth]{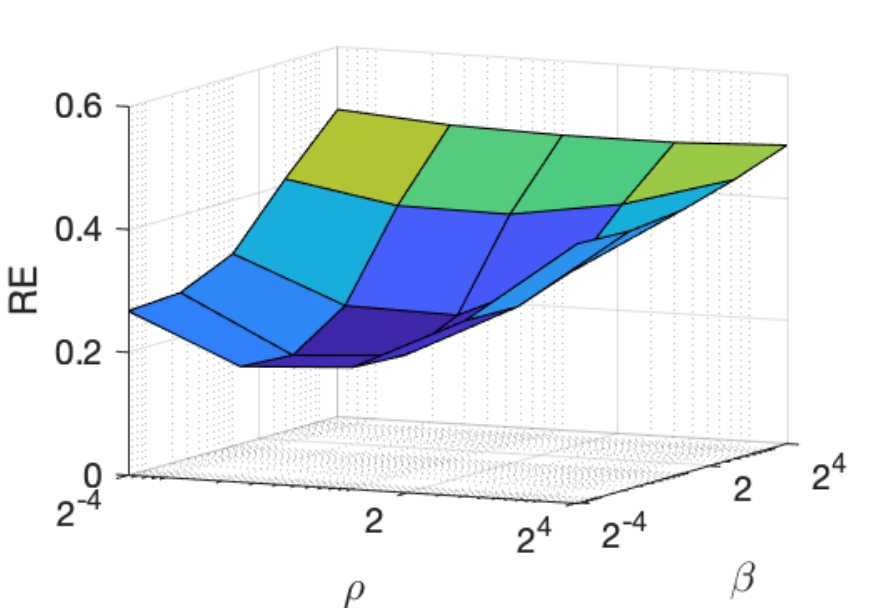} &
				\includegraphics[width=0.30\textwidth]{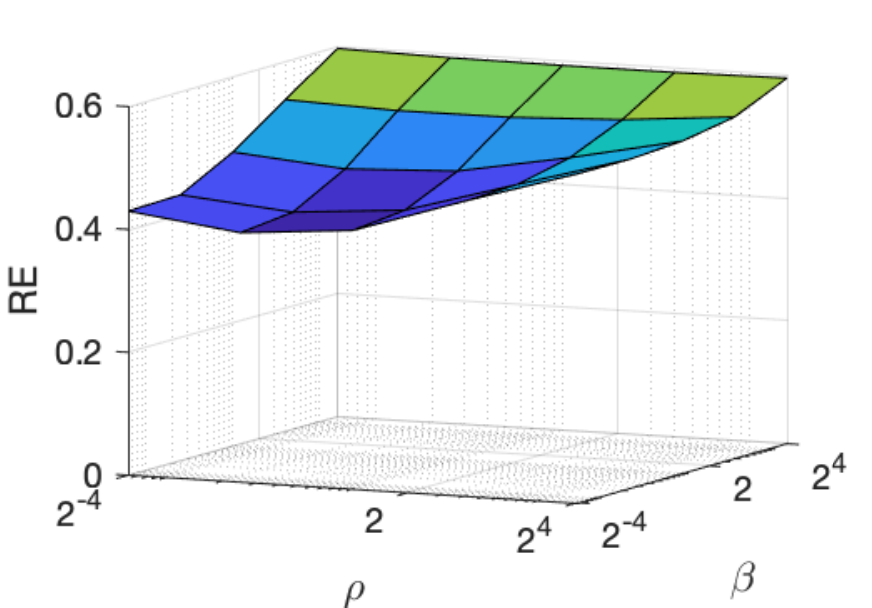}  \\
				\includegraphics[width=0.30\textwidth]{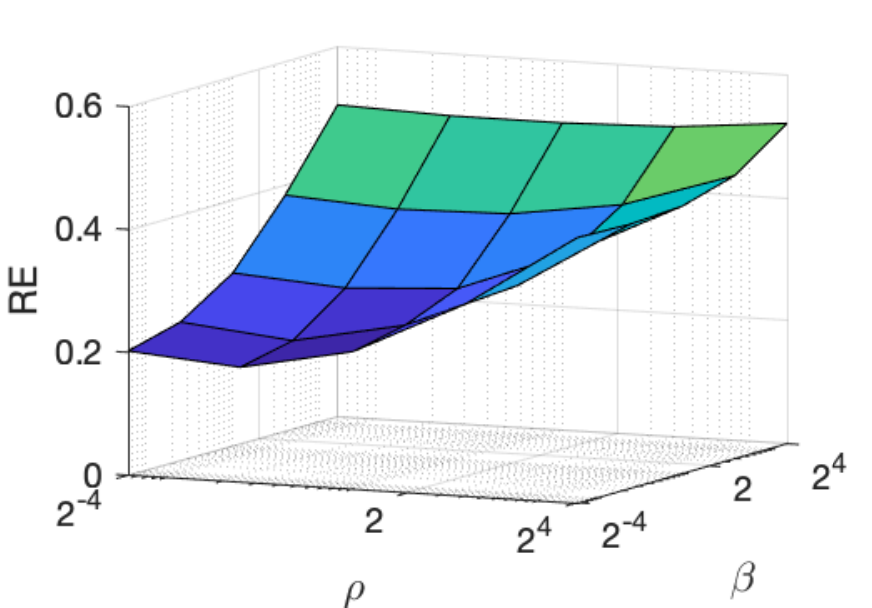} 
				&
				\includegraphics[width=0.30\textwidth]{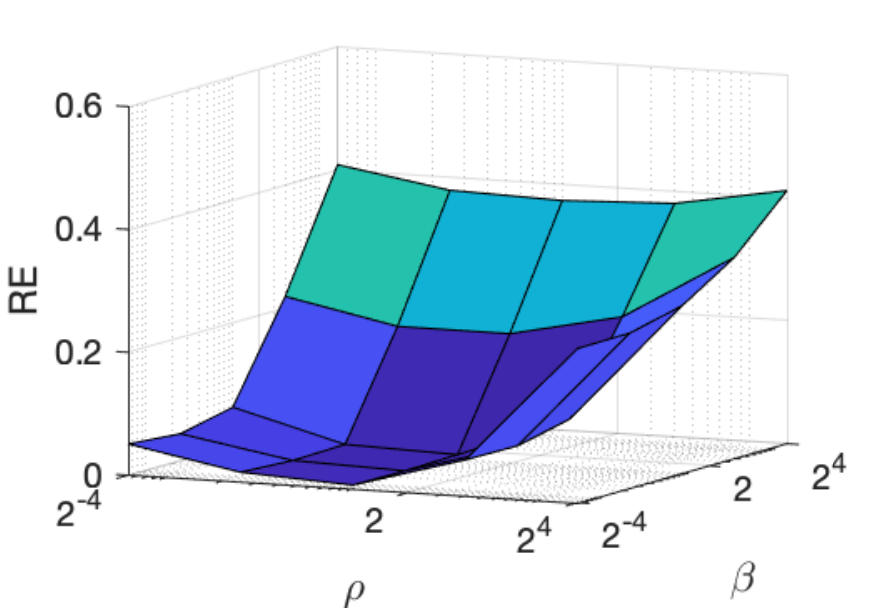} &
				\includegraphics[width=0.30\textwidth]{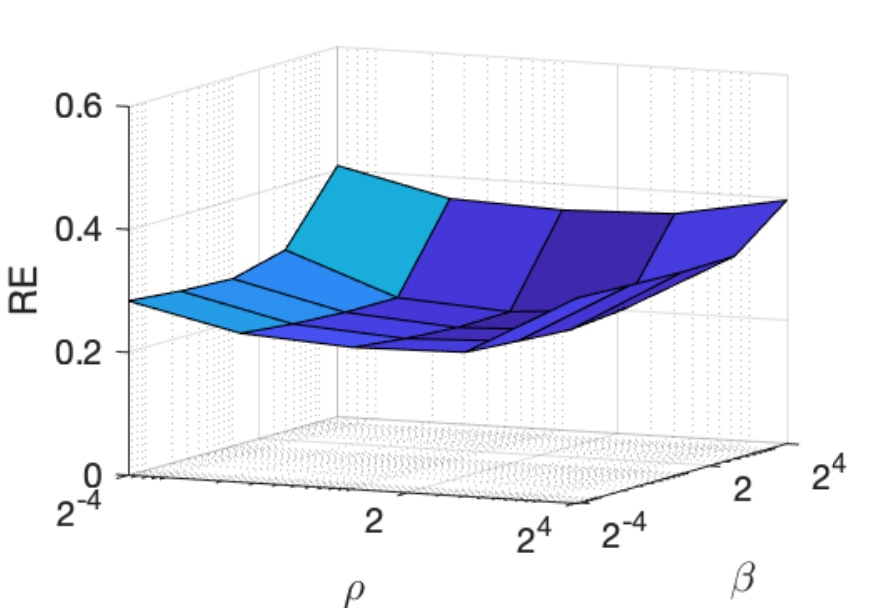}  
			\end{tabular}
		\end{center}
		\caption{\tp{The relative errors  with respect to the parameters $\lambda, \rho,\beta$ in \Cref{alg:l1dl2_box} for CT reconstruction when kMax is 100 (top) or 300  (bottom).} }\label{fig:sa_param_CT}
	\end{figure}
	
\subsection{Super-resolution}\label{sec:SR}
We use an  original image from \cite{gonzalez2004digital}  of size $688\times 688$   to illustrate the performance of super-resolution.  As super-resolution is similar to  MRI  in the sense of frequency measurements, we set up the maximum iteration number as 5 according to \Cref{sect:exp_alg}. We restrict the data within a square in the center of the frequency domain (corresponding to low-frequency measurements), and hence
 varying the sizes of the square leads to different sampling ratios.
In addition to regularized methods, we include a direct method of filling in the unknown frequency data by zero, followed by inverse Fourier transform, which is referred to as zero-filling (ZF).
 The visual results of $1\%$ are presented in \Cref{fig:SR}, showing that  \tp{both $L_p$ and $L_1/L_2$  are}  superior over ZF, $L_1$, and $L_1$-$\alpha L_2.$  Specifically, $L_1/L_2$ can recover these thin rectangular bars, while $L_1$ and $L_1$-$\alpha L_2$ lead to thicker bars with white background, which should be gray. 
 In addition, \ \tp{$L_p$ and $L_1/L_2$} can recover the most of the letter `a' in the bottom of the image, compared to the other methods, while \tp{$L_1/L_2$ is better than $L_p$ with more natural boundaries along  the six dots in the middle left of the image.  }
 One drawback of $L_1/L_2$ is that it produces white artifacts near the third square from the left as well as around the letter `a' in the middle.  
 We suspect $L_1/L_2$ is not very stable, and the box constraint forces the black-and-white regions near edges. We do not present quantitative measures for this example, as four noisy squares on the right of the image lead to meaningless comparison, considering that all the methods return   smooth results. 
 

	

		\begin{figure}[htp]
		\begin{center}
			\begin{tabular}{ccc}
				Ground truth &  ZF & $L_1$\\
				\includegraphics[width=0.3\textwidth]{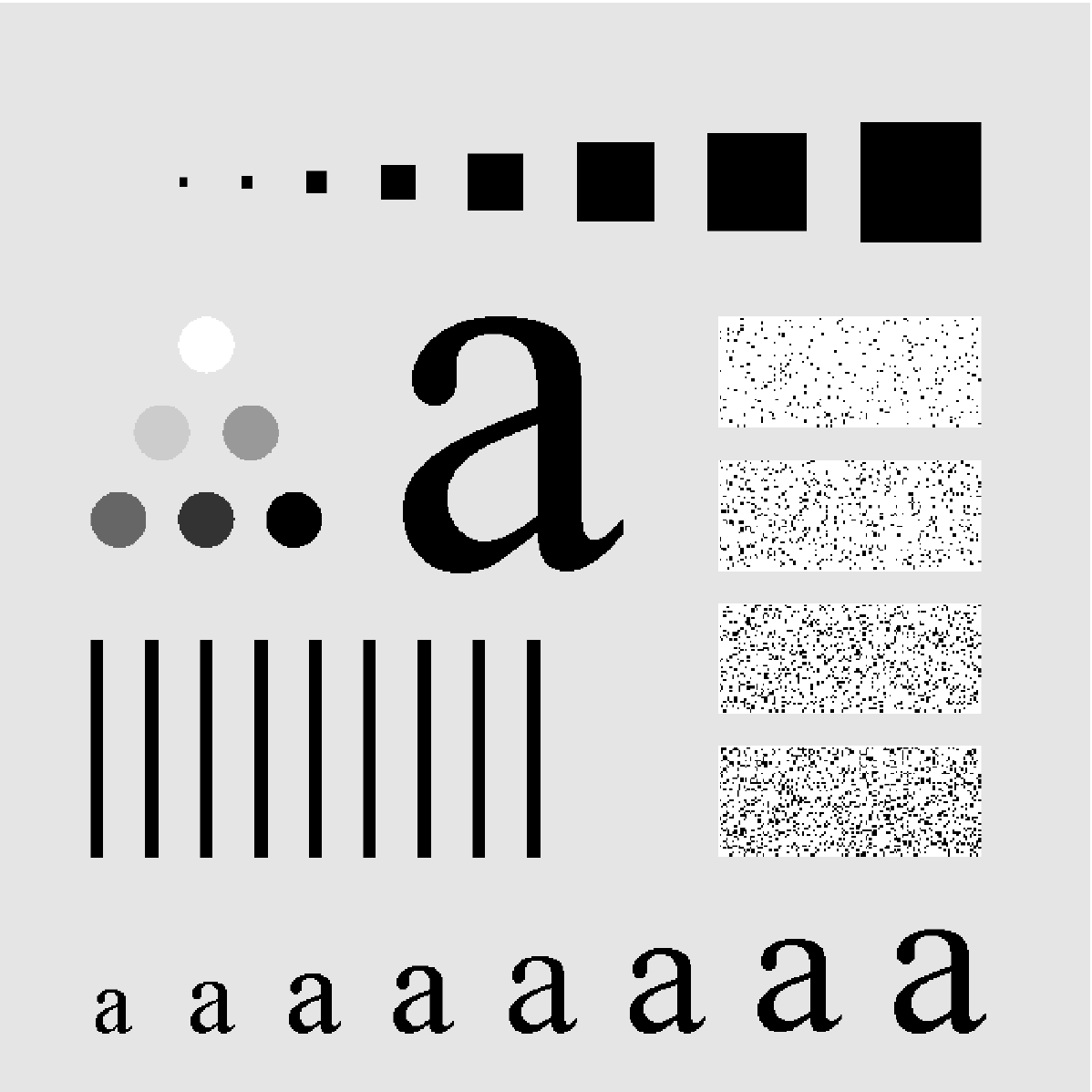} &
				\includegraphics[width=0.3\textwidth]{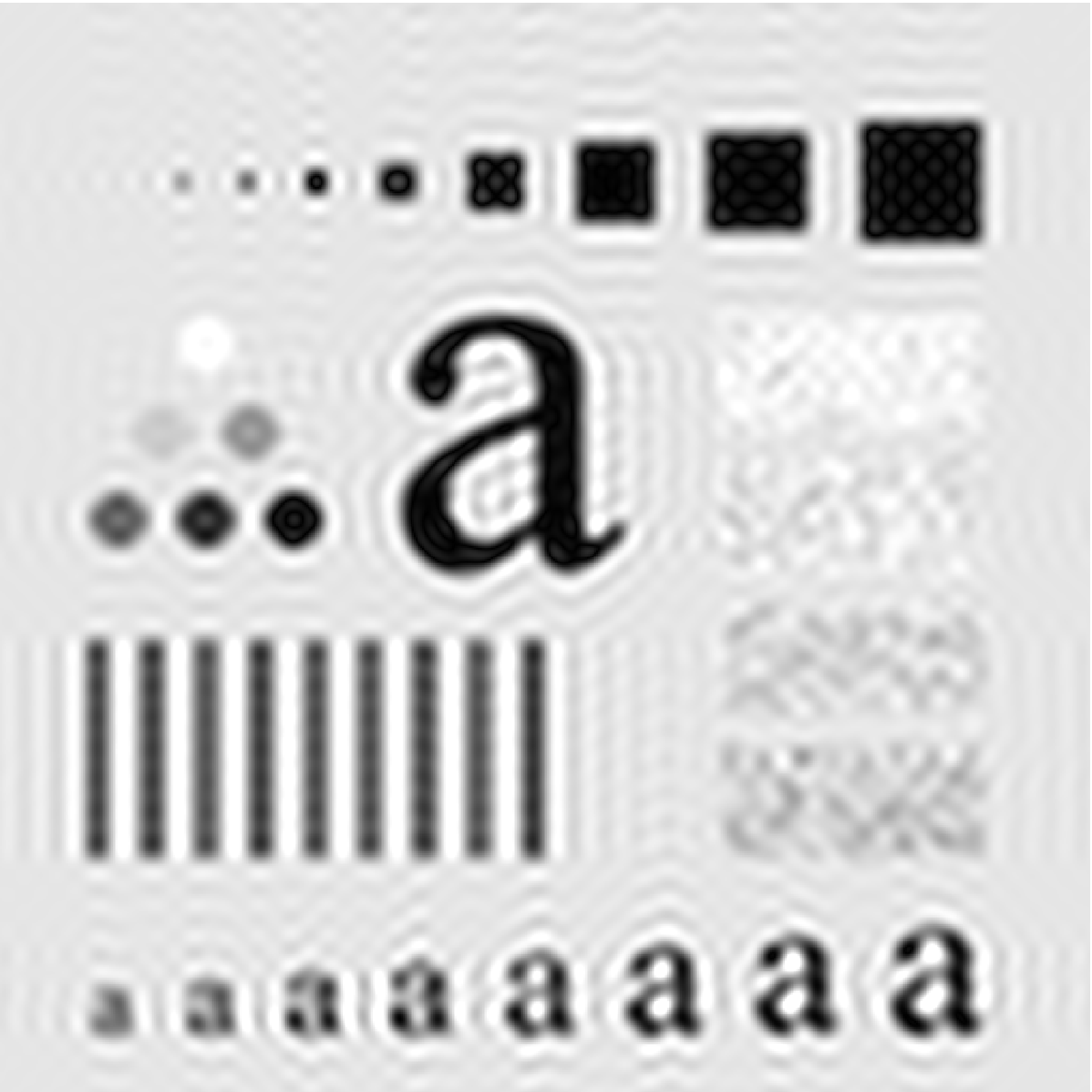} &
				\includegraphics[width=0.3\textwidth]{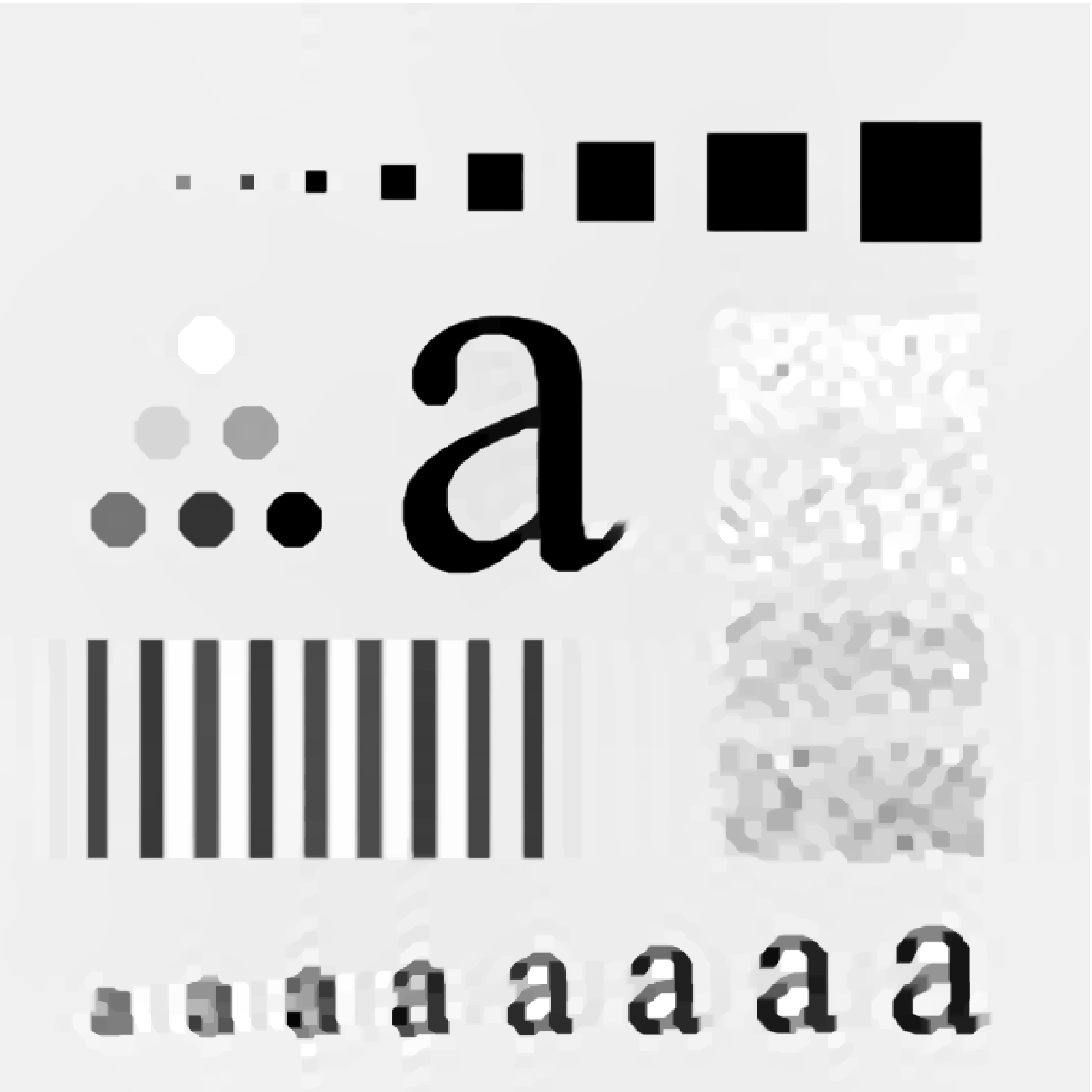} \\
						$L_p$ & $L_1$-$\alpha L_2$ & $L_1/L_2$\\
				\includegraphics[width=0.3\textwidth]{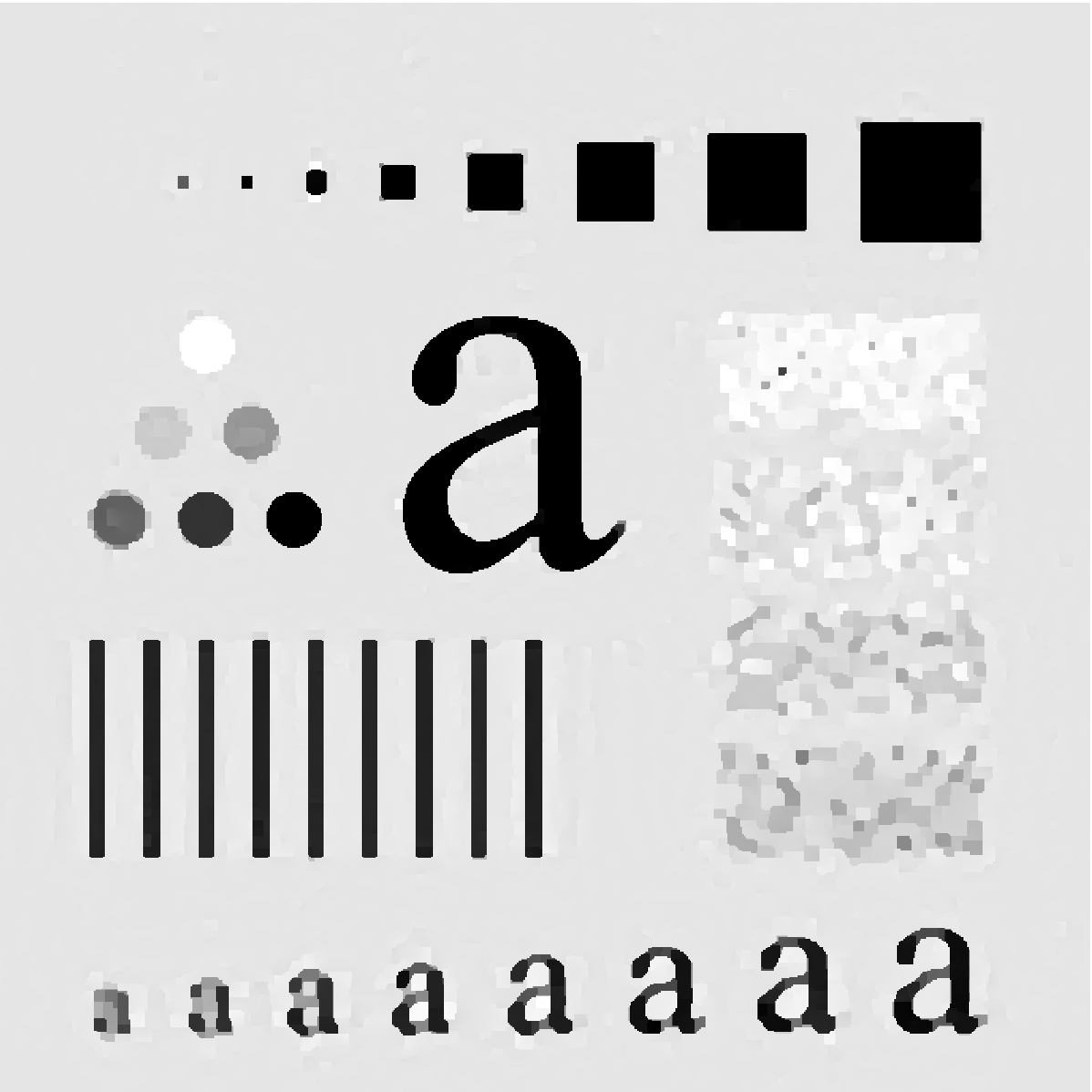} &
				\includegraphics[width=0.3\textwidth]{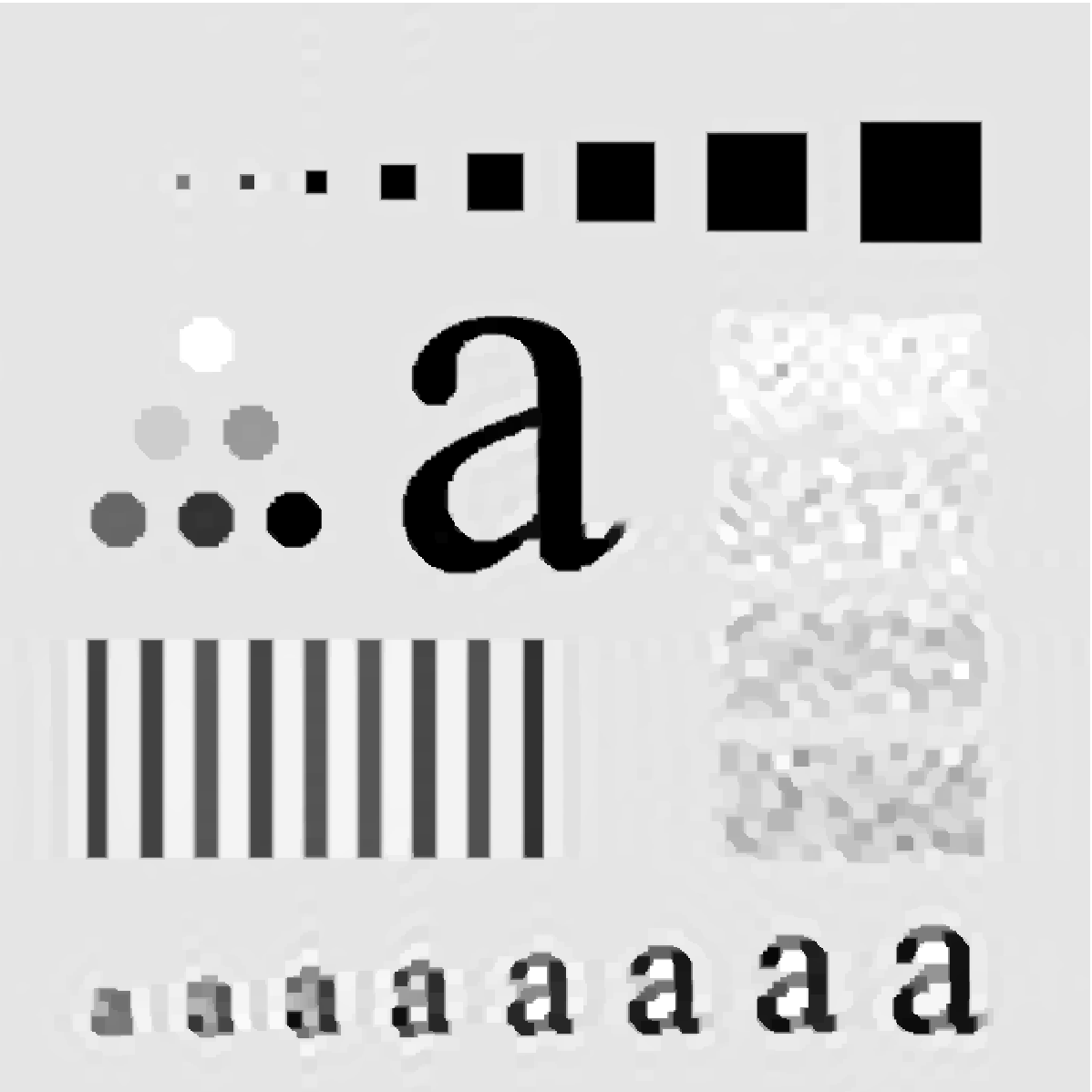} &
				\includegraphics[width=0.3\textwidth]{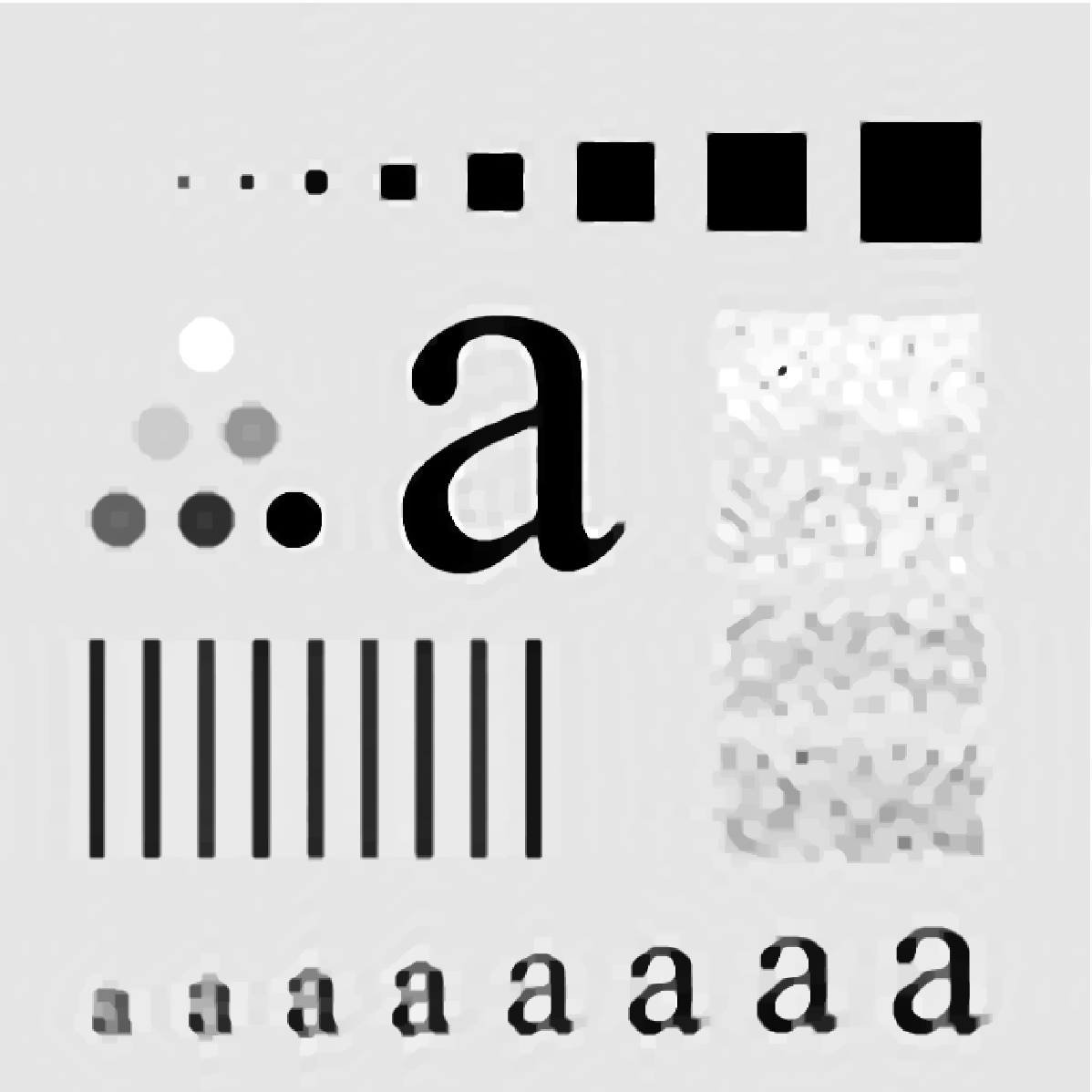}
			\end{tabular}
		\end{center}
		\caption{\tp{Super-resolution from $1\%$ low frequency data.}  }\label{fig:SR}
	\end{figure}

\subsection{MRI reconstruction}

To generate the ground-truth MRI images, we utilize   a simulated brain database \cite{brainweb,brainweb2} that has full three-dimensional data volumes obtained
by an MRI simulator \cite{MRI_simulater} in different modalities such as T1 and T2 weighted images.  As a proof of concept, we extract one slice from the 3D T1 and T2 data as  testing images and take frequency data along  \tp{radial} lines. 
The visual comparisons are presented  for 25 \tp{radial} lines  (about 13.74\% measurements) in  \Cref{fig:MRI_gt}.
 We include the zero-filled method as mentioned in super-resolution, which unfortunately fails to recover the contrast for both T1 and T2. The other regularization methods yield more blurred results than the proposed $L_1/L_2$ approach. Particularly worth noticing is that our proposed model can effectively separate the gray matter and white matter  in the T1 image, as highlighted in the zoom-in regions. Furthermore, we plot the horizontal and vertical profiles in \Cref{fig:MRI_hor_ve}, where we can see clearly that the restored profiles via $L_1/L_2$ are closer to the ground truth than the other approaches, especially near these peaks that can be reached by \tp{$L_p,$} $L_1$-$\alpha L_2,$ and $L_1/L_2$, but not $L_1.$  As a further comparison, we present the MRI reconstruction results under various number of lines (20, 25, and 30) in \Cref{Tab:MRI}, which
 demonstrates significant improvements of $L_1/L_2$ over the other models in term of   PSNR and RE. 


\begin{figure}[htp]
		\begin{center}
				\includegraphics[width=0.19\textwidth]{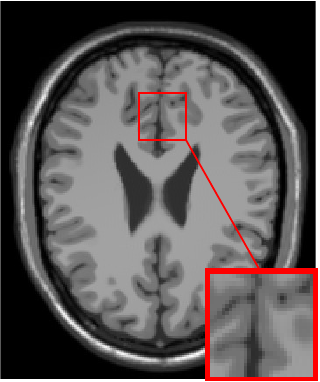}  
				\includegraphics[width=0.19\textwidth]{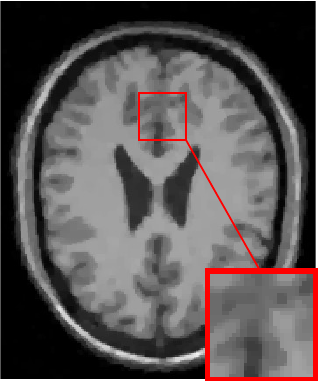} 
				\includegraphics[width=0.19\textwidth]{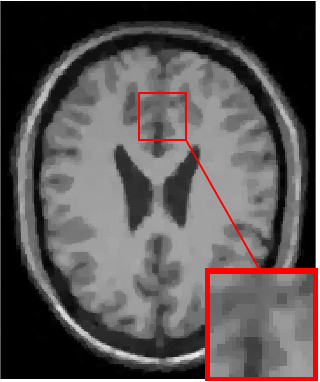} 
				\includegraphics[width=0.19\textwidth]{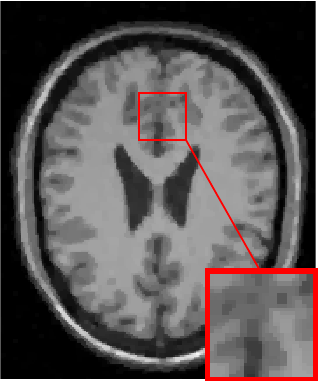} 
				\includegraphics[width=0.19\textwidth]{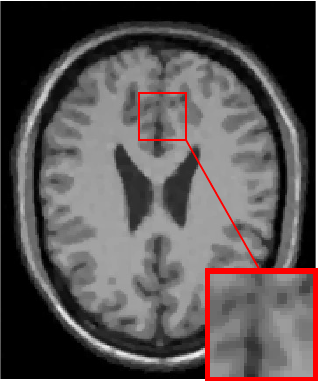} \\
				\includegraphics[width=0.19\textwidth]{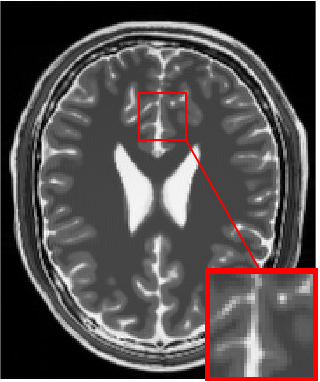} 
				\includegraphics[width=0.19\textwidth]{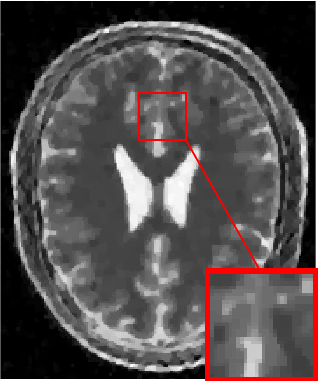} 
				\includegraphics[width=0.19\textwidth]{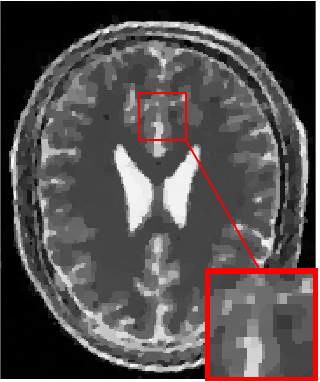} 
				\includegraphics[width=0.19\textwidth]{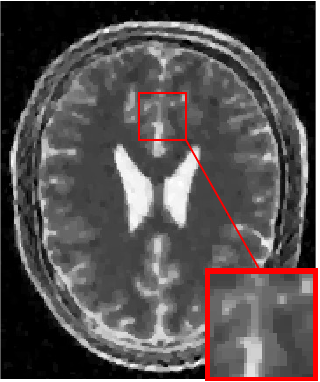} 
				\includegraphics[width=0.19\textwidth]{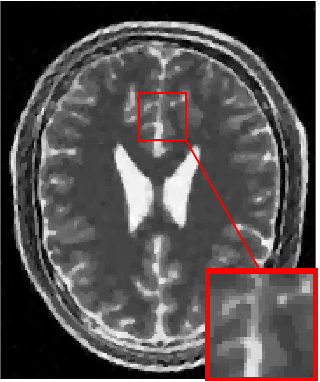} \\
		\end{center}
		\caption{\tp{MRI reconstruction from frequency measurements along 25 \tp{radial} lines of T1 (top row) and T2 (bottom row). From left to right: ground truth, $L_1$, $L_p$, $L_1$-$\alpha L_2$, and $L_1/L_2$. }
		}\label{fig:MRI_gt}
	\end{figure}

	
		\begin{figure}[htp]
		\begin{center}
				\includegraphics[height=0.65\textwidth]{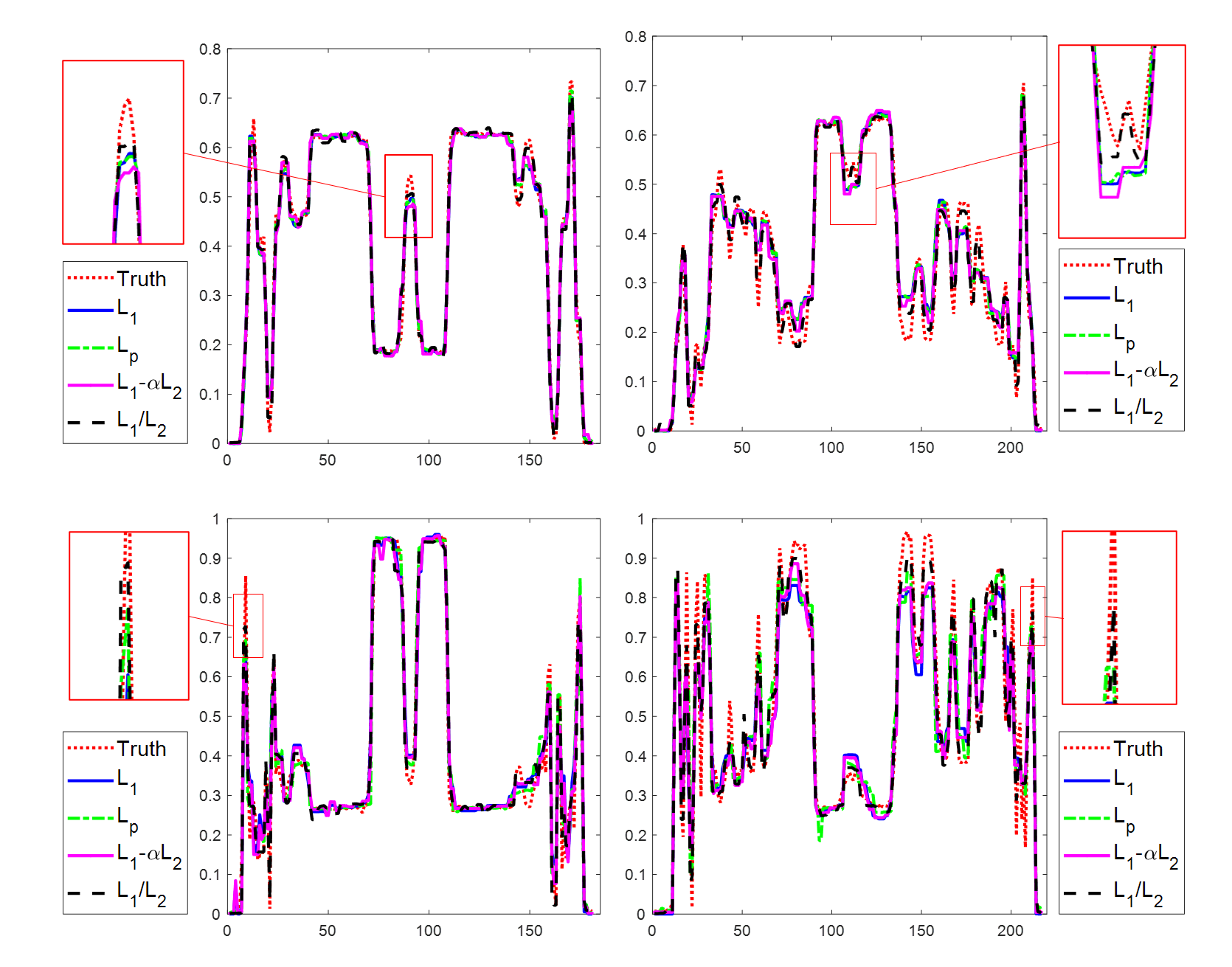} 
		\end{center}
		\caption{\tp{Horizontal (left) and vertical (right) profiles of MRI reconstruction results from $25$ \tp{radial} lines  for T1 (top) and T2 (bottom). } }\label{fig:MRI_hor_ve}
	\end{figure}


		\begin{table}[t]
		\begin{center}
			\scriptsize
			\caption{\tp{MRI reconstruction from   different numbers of radial lines. }  } 
		\tp{	\begin{tabular}{c|c|cc|cc|cc|cc|cc} 
				\hline 
				\multirow{2}{*}{Image} & \multirow{2}{*}{Line} & \multicolumn{2}{c|}{ZF}& \multicolumn{2}{c|}{$L_1$ } & \multicolumn{2}{c|}{$L_p$ } & \multicolumn{2}{c|}{$L_1$-$\alpha L_2$ } & \multicolumn{2}{c}{$L_1/L_2$ }  \\ \cline{3-12} 
				&  & PSNR &   RE & PSNR &   RE & PSNR & RE & PSNR &  RE & PSNR &   RE  \\ \hline
				\multirow{3}{*}{T1}&  20 & 21.26 & 22.13\% &  27.20 & 11.17\% &  27.24 & 11.11\% &  27.41 & 10.90\% &  29.94 & 8.15\%   \\ \cline{2-12} 
				& 25 & 23.42 & 17.26\% &  30.32 & 7.80\% &  30.06 & 8.04\% &  30.34 & 7.78\% &  33.21 & 5.59\%    \\ \cline{2-12}
				& 30 & 24.07 & 16.02\% &  31.92 & 6.48\% &  31.63 & 6.70\% &  31.70 & 6.65\% &  34.84 & 4.63\%
				 \\ \hline	 
				\multirow{3}{*}{T2}&  20 & 17.89 & 33.91\% &  21.12 & 23.37\% &  21.13 & 23.34\% &  21.74 & 21.75\% &  23.50 & 17.76\%     \\ \cline{2-12} 
				& 25 & 18.83 & 30.44\% &  22.92 & 18.99\% &  23.23 & 18.33\% &  23.59 & 17.58\% &  25.80 & 13.63\%   \\ \cline{2-12}
				& 30 & 19.42 & 28.43\% &  24.27 & 16.26\% &  24.76 & 15.36\% &  25.10 & 14.78\% &  27.60 & 11.09\%  
				 \\ \hline	 
			\end{tabular}}\label{Tab:MRI}
			\medskip
		\end{center}
	\end{table}



\subsection{Limited-angle CT reconstruction}
Lastly, we examine the limited-angle CT reconstruction problem on  two standard phantoms: Shepp-Logan  (SL)   by Matlab's built-in command ({\tt phantom}) and  FORBILD (FB) \cite{FB_ph}. Notice that the FB phantom has a very low image contrast and  we display it with the grayscale window of [1.0, 1.2] in order  to reveal its structures; see Figure~\ref{fig:SL_FB_restored}. 
To synthesize the  CT projected data, we discretize both phantoms  at a resolution of $256\times256$. The forward operator $A$ is generated as the discrete Radon transform with a parallel beam geometry sampled at $ \theta_{\rm{Max}}/30$ over a range of   $\theta_{\rm{Max}}$, resulting in a sub-sampled data of size $362\times31$. Note that  we use the same number of projections when we vary ranges of projection angles.  The simulation process is available in the IR and AIR toolbox \cite{Nagy_toolIR,hansen2012air}. 
  Following the discussion in \Cref{sect:exp_alg},  we set jMax$=1$ for the subproblem.  We  compare
 the regularization models with a clinical standard approach, called  simultaneous algebraic reconstruction technique (SART) \cite{SART}. 
  
 As the SL phantom has relatively simpler structures than FB, we  present an extremely limited angle of only $30^\circ$ scanning range in \Cref{Tab:noiseless}, which shows that $L_1/L_2$ achieves  significant improvements over SART, $L_1$, \tp{$L_p$,} and $L_1$-$\alpha L_2$ in terms of \tp{PSNR} and RE. 
	Visually, we present the CT reconstruction results of  $45^\circ$ projection range for SL (SL-$45^\circ$)  and $75^\circ$ for FB (FB-$75^\circ$) in \Cref{fig:SL_FB_restored}. In the first case (SL-$45^\circ$), SART fails to recover the ellipse inside of the skull with such a small range of projection angles. All the regularization methods perform much better owing to their sparsity promoting property.  However, the $L_1$ model is unable to restore the bottom skull and preserve details of some ellipses in the middle. The proposed $L_1/L_2$ method leads an almost exact recovery with a relative error of $0.64\%$ and visually no difference to the  ground truth. 
	In the second case (FB-$75^\circ$), we show the reconstructed images with the window of [1.0, 1.2], and observe some fluctuations inside of the skull. \tp{$L_p$ performs the best, while}
	$L_1/L_2$   restores the most details of the image
among the competing methods.   We plot the horizontal and vertical profiles in \Cref{fig:FB75_vh}, which illustrates that $L_1/L_2$ leads to the smallest fluctuations  compared to the  other methods. We also observe a well-known artifact of 
	the $L_1$ method, i.e., loss-of-contrast, as its profile  fails to reach the height of jump on the intervals such as $[160,180]$ in the left plot and $[220, 230]$ in the right plot of \Cref{fig:FB75_vh}, while $L_1/L_2$ has a good recovery in these regions.


	\begin{figure}
		\begin{center}
					\includegraphics[width=0.19\textwidth]{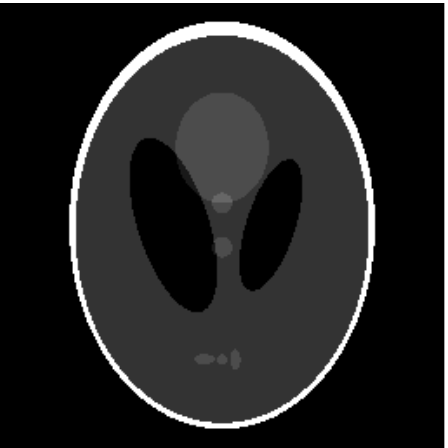} 
				\includegraphics[width=0.19\textwidth]{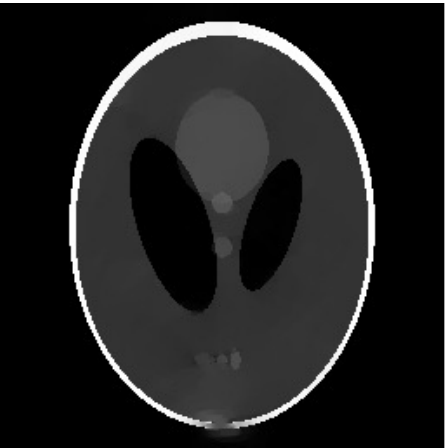} 
				\includegraphics[width=0.19\textwidth]{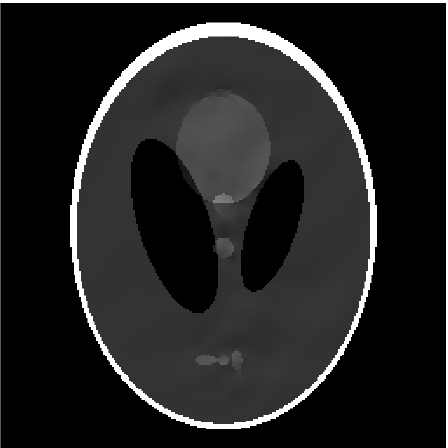} 
				\includegraphics[width=0.19\textwidth]{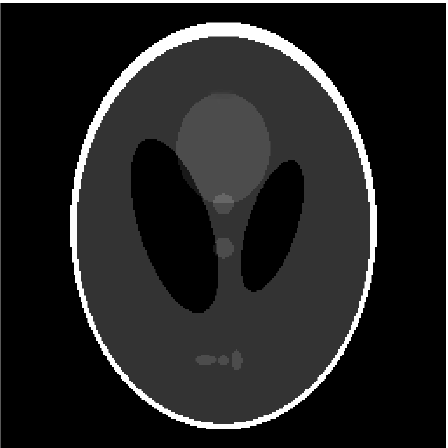} 
				\includegraphics[width=0.19\textwidth]{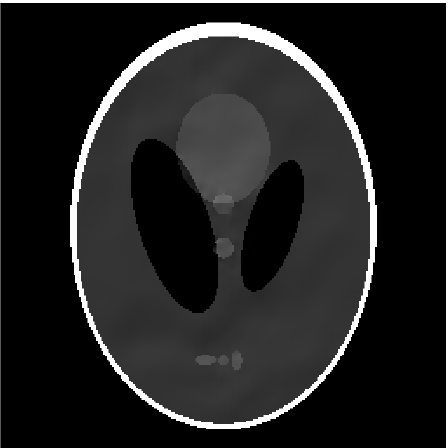} \\
				\includegraphics[width=0.19\textwidth]{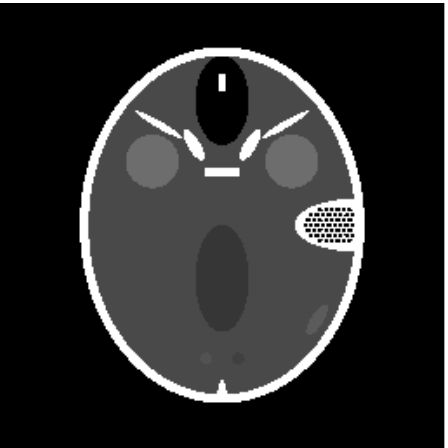} 
				\includegraphics[width=0.19\textwidth]{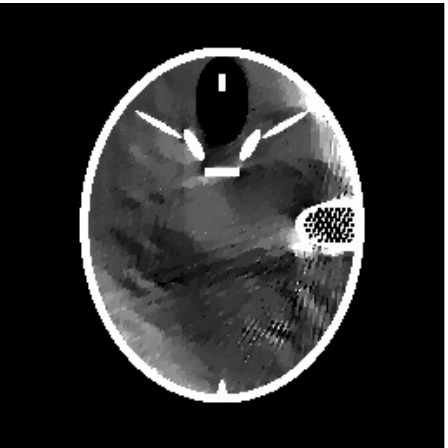} 
	        	\includegraphics[width=0.19\textwidth]{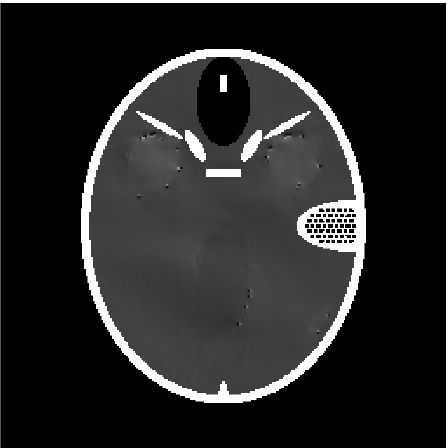} 
				\includegraphics[width=0.19\textwidth]{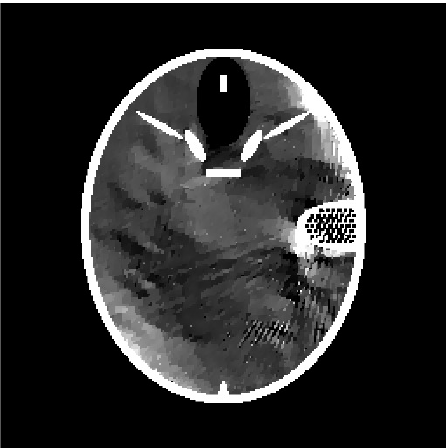}  
				\includegraphics[width=0.19\textwidth]{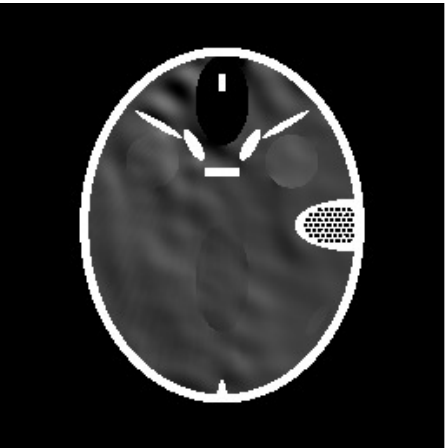} 
		\end{center}
		\caption{\tp{CT reconstruction results of SL-$45^\circ$  (top) and FB-$75^\circ$ (bottom). From left to right: ground truth, $L_1$, $L_p$, $L_1$-$\alpha L_2$, and $L_1/L_2$.  The gray scale window is $[0,1]$ for SL and $[1,1.2]$ for FB. }}\label{fig:SL_FB_restored}
	\end{figure}


	\begin{figure}
		\begin{center}
			\begin{tabular}{cc}
				\includegraphics[width=0.4\textwidth]{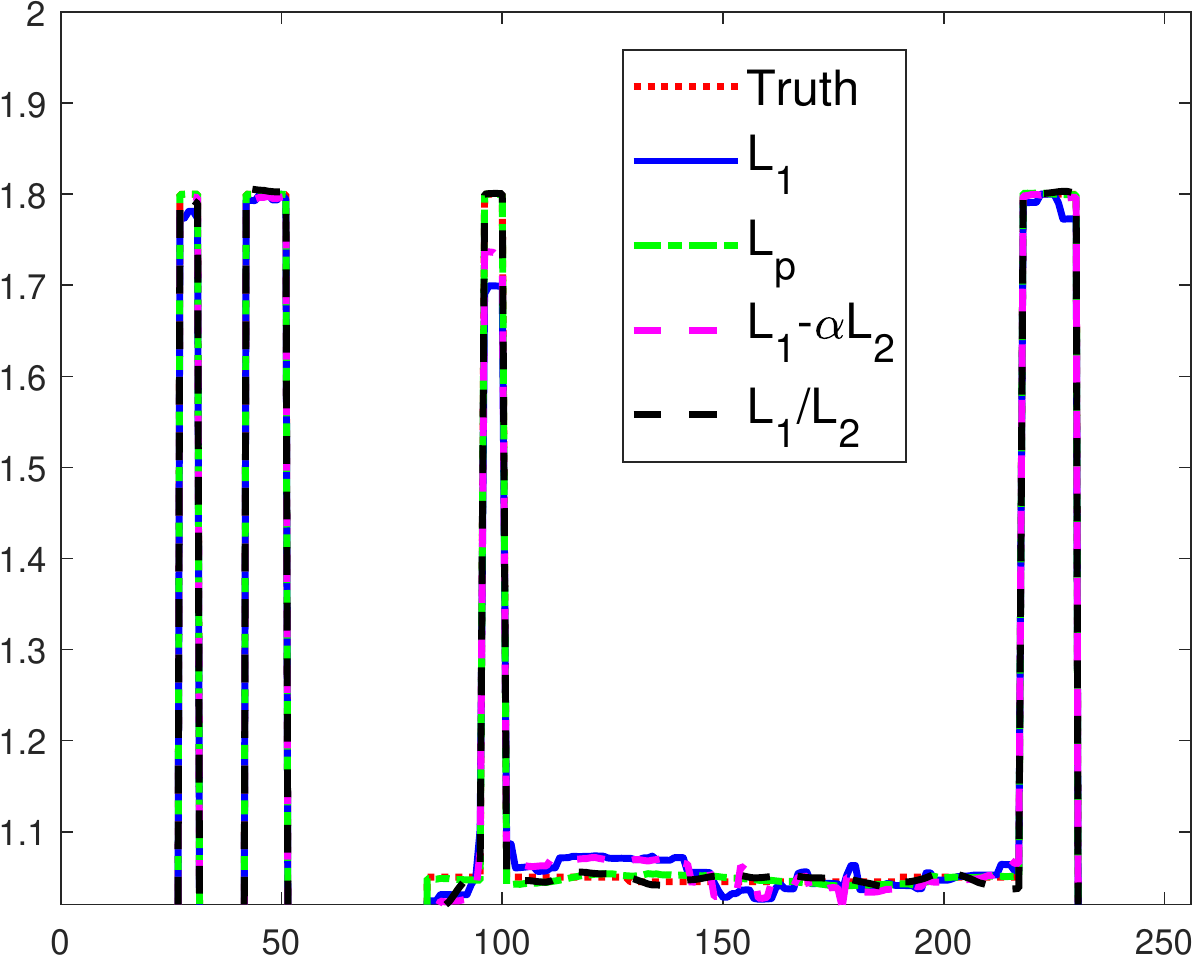} &
				\includegraphics[width=0.4\textwidth]{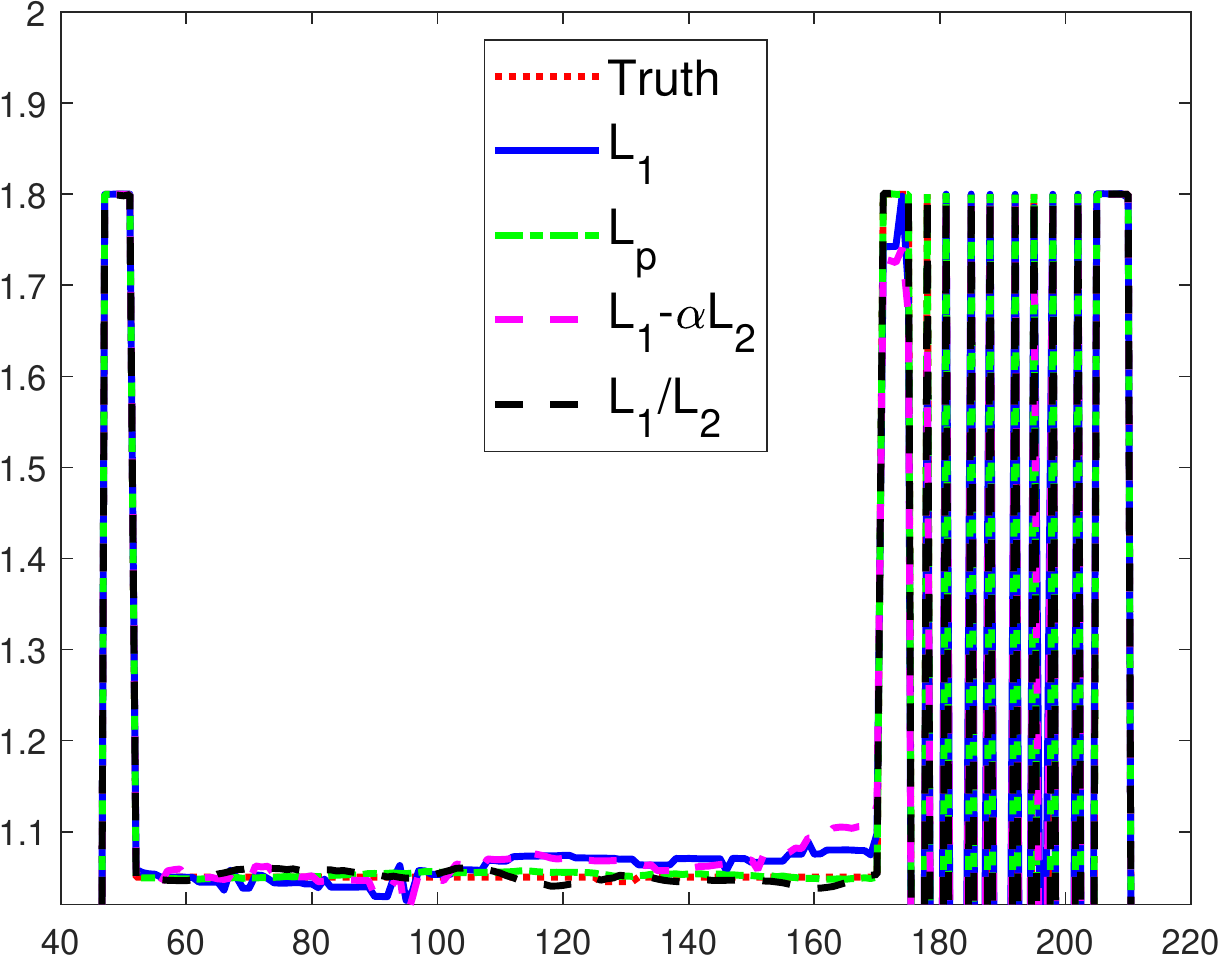} 
			\end{tabular}
		\end{center}
		\caption{\tp{Horizontal and vertical profiles of CT reconstruction results of FB-$75^\circ$.} }\label{fig:FB75_vh}
	\end{figure}


	\begin{table}[htbp]
		\begin{center}
			\scriptsize
			\caption{\tp{ CT reconstruction with difference ranges of scanning  angles. } } 		\tp{
			\begin{tabular}{c|c|cc|cc|cc|cc|cc}
				\hline
				\multirow{2}{*}{phantom} & \multirow{2}{*}{range} & \multicolumn{2}{c|}{SART} & \multicolumn{2}{c|}{$L_1$} & \multicolumn{2}{c|}{$L_p$} & \multicolumn{2}{c|}{$L_1$-$\alpha L_2$} & \multicolumn{2}{c}{$L_1/L_2$} \\ \cline{3-12} 
				& & PSNR  &   RE &  PSNR  &   RE & PSNR  & RE & PSNR  & RE & PSNR  &  RE  \\ \hline
				\multirow{3}{*}{SL}&  $30^{\circ}$ & 15.66 & 66.95\% &   28.32 & 15.57\% &  40.25 & 3.95\% &  38.15 & 5.02\% &   60.77 & 0.37\%   \\ \cline{2-12} 
				& $45^{\circ}$ & 16.08 & 63.78\% &   33.33 & 8.75\% &  44.06 & 2.54\% &  63.34 & 0.28\% &  70.42 & 0.12\%  \\ \cline{2-12} 
				& $60^{\circ}$ & 16.48 & 60.92\% &   43.37 & 2.75\% &  46.50 & 1.92\% &  80.19 & 0.04\% &  73.46 & 0.09\%   \\ \hline
				\multirow{3}{*}{FB}&  $60^{\circ}$ & 15.61 & 40.16\% &   25.43 & 12.96\% &  58.01 & 0.30\% &  26.24 & 11.81\% &  46.97 & 1.09\%  \\ \cline{2-12} 
				& $75^{\circ}$ & 16.14 & 37.79\% &   28.84 & 8.76\% &  59.02 & 0.27\% &  29.49 & 8.13\% &  49.30 & 0.83\%  \\ \cline{2-12} 
				& $90^{\circ}$ & 16.64 & 35.68\% &   69.68 & 0.08\% &  62.05 & 0.19\% &  75.67 & 0.04\% &  70.57 & 0.07\%     \\ \hline
			\end{tabular}}\label{Tab:noiseless}
			\medskip
		\end{center}
	\end{table}

\section{Conclusions and future works}
	\label{sec:conclusions}
	
In this paper, we considered the use of $L_1/L_2$ on the gradient as an objective function to promote sparse gradients for imaging problems. We started from a series of 1D piecewise signal recovery and demonstrated the superiority of the ratio model over  $L_1$, which is widely known as the total variation. To facilitate the discussion on the empirical evidences, we focused on a constrained model, and proposed a splitting algorithm scheme that has provable convergence for ADMM. We conducted extensive experiments to demonstrate that our approach outperforms the state-of-the-art gradient-based approaches. 
Motivated by the empirical studies in \Cref{sec:toy_example}, we will devote ourselves to 
the exact recovery of the TV regularization with respect to the minimum separation of the gradient spikes. 
We are also interested in extending the analysis to the unconstrained formulation, which is widely applicable in image processing.   

\section*{Acknowledgments}
C.~Wang was partially supported by HKRGC Grant No.CityU11301120 and NSF CCF HDR TRIPODS grant 1934568. M.~Tao was supported in part by the Natural Science Foundation of China (No. 11971228) and  the Jiangsu Provincial National Natural Science Foundation of China (No. BK20181257). C-N.~Chuah was partially supported by NSF CCF HDR TRIPODS grant 1934568. J.~Nagy was partially
supported by NSF DMS-1819042 and NIH 5R01CA181171-04. Y.~Lou was partially supported by NSF grant CAREER
1846690.

\section*{References}		
	
	\bibliographystyle{vancouver}
	\bibliography{refer_l1dl2}

\begin{thebibliography}{10}

\bibitem{rudinOF92}
Rudin L, Osher S, Fatemi E.
\newblock Nonlinear total variation based noise removal algorithms.
\newblock Physica D. 1992;60:259--268.

\bibitem{Bredies2010TGV}
Bredies K, Kunisch K, Pock T.
\newblock Total generalized variation.
\newblock SIAM J Imaging Sci. 2010;3(3):492--526.

\bibitem{chen2015Fractional}
Chen D, Chen YQ, Xue D.
\newblock Fractional-order total variation image denoising based on proximity
  algorithm.
\newblock Appl Math Comput. 2015;257(1):537--545.

\bibitem{Zhang2015Atotal}
Zhang J, Chen K.
\newblock A total fractional-Order variation model for image restoration with
  nonhomogeneous boundary conditions and its numerical solution.
\newblock SIAM J Imaging Sci. 2015;8(4):2487--2518.

\bibitem{osherE03}
Osher SJ, Esedoglu S.
\newblock Decomposition of images by the anisotropic Rudin-Osher-Fatemi model.
\newblock Comm Pure Appl Math. 2003;57:1609--1626.

\bibitem{natarajan95}
Natarajan BK.
\newblock Sparse approximate solutions to linear systems.
\newblock SIAM J Comput. 1995:227--234.

\bibitem{CRT}
Cand\`es EJ, Romberg J, Tao T.
\newblock Stable signal recovery from incomplete and inaccurate measurements.
\newblock Comm Pure Appl Math. 2006;59:1207--1223.

\bibitem{fan2001variable}
Fan J, Li R.
\newblock Variable selection via nonconcave penalized likelihood and its oracle
  properties.
\newblock J Am Stat Assoc. 2001;96(456):1348--1360.

\bibitem{zhang2010nearly}
Zhang C.
\newblock Nearly unbiased variable selection under minimax concave penalty.
\newblock Ann Stat. 2010:894--942.

\bibitem{zhang2009multi}
Zhang T.
\newblock Multi-stage convex relaxation for learning with sparse
  regularization.
\newblock In: Adv. Neural. Inf. Process. Syst.; 2009. p. 1929--1936.

\bibitem{shen2012likelihood}
Shen X, Pan W, Zhu Y.
\newblock Likelihood-based selection and sharp parameter estimation.
\newblock J Am Stat Assoc. 2012;107(497):223--232.

\bibitem{louYX16}
Lou Y, Yin P, Xin J.
\newblock Point source super-resolution via non-convex {$L_1$} based methods.
\newblock J Sci Comput. 2016;68:1082--1100.

\bibitem{lv2009unified}
Lv J, Fan Y.
\newblock A unified approach to model selection and sparse recovery using
  regularized least squares.
\newblock Ann Appl Stat. 2009:3498--3528.

\bibitem{zhangX17}
Zhang S, Xin J.
\newblock Minimization of transformed {$L_1$} penalty: closed form
  representation and iterative yhresholding algorithms.
\newblock Comm Math Sci. 2017;15:511--537.

\bibitem{zhangX18}
Zhang S, Xin J.
\newblock Minimization of transformed {$L_1$} penalty: theory, difference of
  convex function algorithm, and robust application in compressed sensing.
\newblock Math Program. 2018;169:307--336.

\bibitem{chartrand07}
Chartrand R.
\newblock Exact reconstruction of sparse signals via nonconvex minimization.
\newblock IEEE Signal Process Lett. 2007;10(14):707--710.

\bibitem{you2019nonconvex}
You J, Jiao Y, Lu X, Zeng T.
\newblock A nonconvex model with minimax concave penalty for image restoration.
\newblock J Sci Comput. 2019;78(2):1063--1086.

\bibitem{l1dl2}
Rahimi Y, Wang C, Dong H, Lou Y.
\newblock A scale invariant approach for sparse signal recovery.
\newblock SIAM J Sci Comput. 2019;41(6):A3649--A3672.

\bibitem{L1dL2_accelerated}
Wang C, Yan M, Rahimi Y, Lou Y.
\newblock Accelerated schemes for the $ {L}_1/{L}_2 $ minimization.
\newblock IEEE Trans Signal Process. 2020;68:2660--2669.

\bibitem{TaoLou2020}
Tao M. Minimization of L1 over L2 for sparse signal recovery with convergence
  guarantee; 2020.
\newblock \url{http://www.optimization-online.org/DB_HTML/2020/10/8064.html}.

\bibitem{wang2020limitedct}
Wang C, Tao M, Nagy J, Lou Y.
\newblock Limited-angle CT reconstruction via the {$L_1/L_2$} minimization.
\newblock SIAM J Imaging Sci. 2021;14(2):749--777.

\bibitem{candesF13}
Cand{\`e}s EJ, Fernandez-Granda C.
\newblock Super-resolution from noisy data.
\newblock J Fourier Anal Appl. 2013;19(6):1229--1254.

\bibitem{boydPCPE11admm}
Boyd S, Parikh N, Chu E, Peleato B, Eckstein J.
\newblock Distributed optimization and statistical learning via the alternating
  direction method of multipliers.
\newblock Found Trends Mach Learn. 2011 Jan;3(1):1--122.

\bibitem{louZOX15}
Lou Y, Zeng T, Osher S, Xin J.
\newblock A weighted difference of anisotropic and isotropic total variation
  model for image processing.
\newblock SIAM J Imaging Sci. 2015;8(3):1798--1823.

\bibitem{yinEX14}
Yin P, Esser E, Xin J.
\newblock Ratio and Difference of $l_1$ and $l_2$ Norms and Sparse
  Representation with Coherent Dictionaries.
\newblock Comm Info Systems. 2014;14:87--109.

\bibitem{louYHX14}
Lou Y, Yin P, He Q, Xin J.
\newblock Computing sparse representation in a highly coherent dictionary based
  on difference of {$L_1$} and {$L_2$}.
\newblock J Sci Comput. 2015;64(1):178--196.

\bibitem{maLH17}
Ma T, Lou Y, Huang T.
\newblock Truncated {$L_1$-$L_2$} models for sparse recovery and rank
  minimization.
\newblock SIAM J Imaging Sci. 2017;10(3):1346--1380.

\bibitem{louY18}
Lou Y, Yan M.
\newblock Fast ${L}_1$-${L}_2$ minimization via a proximal operator.
\newblock J Sci Comput. 2018;74(2):767--785.

\bibitem{guo2017convergence}
Guo K, Han D, Wu T.
\newblock Convergence of alternating direction method for minimizing sum of two
  nonconvex functions with linear constraints.
\newblock Int J of Comput Math. 2017;94(8):1653--1669.

\bibitem{pang2018decomposition}
Pang JS, Tao M.
\newblock Decomposition methods for computing directional stationary solutions
  of a class of nonsmooth nonconvex optimization problems.
\newblock SIAM J Optim. 2018;28(2):1640--1669.

\bibitem{wang2019global}
Wang Y, Yin W, Zeng J.
\newblock Global convergence of {ADMM} in nonconvex nonsmooth optimization.
\newblock J Sci Comput. 2019;78(1):29--63.

\bibitem{candes2014towards}
Cand{\`e}s EJ, Fernandez-Granda C.
\newblock Towards a mathematical theory of super-resolution.
\newblock Comm Pure Appl Math. 2014;67(6):906--956.

\bibitem{cvx}
Grant M, Boyd S. {CVX}: Matlab Software for Disciplined Convex Programming,
  version 2.1; 2014.
\newblock http://cvxr.com/cvx.

\bibitem{donoho1992superresolution}
Donoho DL.
\newblock Superresolution via sparsity constraints.
\newblock SIAM J Math Anal. 1992;23(5):1309--1331.

\bibitem{morgenshtern2016super}
Morgenshtern VI, Candes EJ.
\newblock Super-resolution of positive sources: The discrete setup.
\newblock SIAM Journal on Imaging Sciences. 2016;9(1):412--444.

\bibitem{chan2012multiplicative}
Chan RH, Ma J.
\newblock A multiplicative iterative algorithm for box-constrained penalized
  likelihood image restoration.
\newblock IEEE Trans Image Process. 2012;21(7):3168--3181.

\bibitem{kk_bound_const}
Kan K, Fung SW, Ruthotto L.
\newblock {PNKH-B}: A Projected {Newton-Krylov} Method for Large-Scale
  Bound-Constrained Optimization.
\newblock arXiv preprint arXiv:200513639. 2020.

\bibitem{Optimization_2006book_Nocedal}
Nocedal J, Wright SJ.
\newblock Numerical Optimization.
\newblock Springer; 2006.

\bibitem{bolte2007lojasiewicz}
Bolte J, Daniilidis A, Lewis A.
\newblock The {\L}ojasiewicz inequality for nonsmooth subanalytic functions
  with applications to subgradient dynamical systems.
\newblock SIAM J Optim. 2007;17(4):1205--1223.

\bibitem{li2015global}
Li G, Pong TK.
\newblock Global convergence of splitting methods for nonconvex composite
  optimization.
\newblock SIAM J Optim. 2015;25(4):2434--2460.

\bibitem{KLproperties}
Attouch H, Bolte J, Redont P, Soubeyran A.
\newblock Proximal alternating minimization and projection methods for
  nonconvex problems: An approach based on the Kurdyka-{\L}ojasiewicz
  inequality.
\newblock Mathematics of operations research. 2010;35(2):438--457.

\bibitem{candes2006robust}
Cand{\`e}s EJ, Romberg J, Tao T.
\newblock Robust uncertainty principles: Exact signal reconstruction from
  highly incomplete frequency information.
\newblock IEEE Trans Inf Theory. 2006;52(2):489--509.

\bibitem{AM01}
Avinash C, Malcolm S.
\newblock Principles of computerized tomographic imaging.
\newblock Philadelphia, PA, USA: Society for Industrial and Applied
  Mathematics; 2001.

\bibitem{gwang_limitedct}
Chen Z, Jin X, Li L, Wang G.
\newblock A limited-angle {CT} reconstruction method based on anisotropic {TV}
  minimization.
\newblock Phys Med Biol. 2013;58(7):2119.

\bibitem{zhang2006comparative}
Zhang Y, Chan HP, Sahiner B, Wei J, Goodsitt MM, Hadjiiski LM, et~al.
\newblock A comparative study of limited-angle cone-beam reconstruction methods
  for breast tomosynthesis.
\newblock Med Phys. 2006;33(10):3781--3795.

\bibitem{wang2011low}
Wang Z, Huang Z, Chen Z, Zhang L, Jiang X, Kang K, et~al.
\newblock Low-dose multiple-information retrieval algorithm for X-ray
  grating-based imaging.
\newblock Nucl Instrum Methods Phys. 2011;635(1):103--107.

\bibitem{Xu2012}
Xu Z, Chang X, Xu F, Zhang H.
\newblock $L_{1/2}$ Regularization: A Thresholding Representation Theory and a
  Fast Solver.
\newblock IEEE Trans Neural Netw Learn Syst. 2012;23:1013--1027.

\bibitem{gonzalez2004digital}
Gonzalez RC, Woods RE, Eddins SL.
\newblock Digital image processing using MATLAB.
\newblock Pearson Education India; 2004.

\bibitem{brainweb}
Cocosco CA, Kollokian V, Kwan RKS, Pike GB, Evans AC.
\newblock BrainWeb: Online Interface to a 3D MRI Simulated Brain Database.
\newblock NeuroImage. 1997;5:425.

\bibitem{brainweb2}
Kwan RS, Evans AC, Pike GB.
\newblock MRI simulation-based evaluation of image-processing and
  classification methods.
\newblock IEEE Trans Med Imaging. 1999;18(11):1085--1097.

\bibitem{MRI_simulater}
Kwan RKS, Evans AC, Pike GB.
\newblock An extensible MRI simulator for post-processing evaluation.
\newblock In: International Conference on Visualization in Biomedical
  Computing. Springer; 1996. p. 135--140.

\bibitem{FB_ph}
Yu Z, Noo F, Dennerlein F, Wunderlich A, Lauritsch G, Hornegger J.
\newblock Simulation tools for two-dimensional experiments in x-ray computed
  tomography using the {FORBILD} head phantom.
\newblock Phys Med Biol. 2012;57(13):N237.

\bibitem{Nagy_toolIR}
Gazzola S, Hansen PC, Nagy JG.
\newblock {IR} Tools: a {MATLAB} package of iterative regularization methods
  and large-scale test problems.
\newblock Numer Algorithms. 2019;81(3):773--811.

\bibitem{hansen2012air}
Hansen PC, Saxild-Hansen M.
\newblock {AIR} tools --- {MATLAB} package of algebraic iterative
  reconstruction methods.
\newblock J Comput Appl Math. 2012;236(8):2167--2178.

\bibitem{SART}
Andersen AH, Kak AC.
\newblock Simultaneous algebraic reconstruction technique ({SART}): a superior
  implementation of the {ART} algorithm.
\newblock Ultrason Imaging. 1984;6(1):81--94.

\end{thebibliography}
	
\end{document}